\documentclass[12pt]{article}
\usepackage{amsmath, amsfonts, amssymb,  mathrsfs,  array, stmaryrd,  indentfirst, amsthm,  hyperref, comment, natbib}
\usepackage{graphicx, enumitem, tabularx, slashbox, color}
\usepackage{tcolorbox}
\usepackage{float} 

\usepackage[margin=1.0in]{geometry}
\usepackage{setspace}
\usepackage{natbib}

\numberwithin{equation}{section}
\theoremstyle{plain}
\newtheorem{lemma}{Lemma}[section]
\newtheorem{thm}{Theorem}[section]
\newtheorem{defn}{Definition}[section]
\newtheorem{prop}{Proposition}[section]

\newtheorem{assumption}{Assumption}[section]
\newtheorem{remark}{Remark}[section]

\newcommand{\eps}{\varepsilon}

\newcommand{\bq}{\mathbf{q}}
\newcommand{\by}{\mathbf{y}}
\newcommand{\bbT}{\mathbb{T}}
\newcommand{\bbP}{\mathbb{P}}
\newcommand{\bbR}{\mathbb{R}}

\DeclareMathOperator*{\argmin}{arg\,min}
\DeclareMathOperator*{\argmax}{arg\,max}

\makeatletter
\def\@setcopyright{}
\def\serieslogo@{}
\makeatother

\begin{document}

\title{\textbf{Large tournament games}
\thanks{
Erhan Bayraktar is supported in part by the NSF under grant DMS-1613170 and by the Susan M. Smith Professorship. Jak\v sa Cvitani\'c  is supported in part by the NSF under grant DMS-DMS-1810807. Yuchong Zhang was supported by the NSF under grant DMS-1714607. We are grateful to Marcel Nutz for many stimulating discussions.
}
}

 \author{Erhan Bayraktar\thanks{Department of Mathematics, University of Michigan, 530 Church Street, Ann Arbor, MI 48104, USA, erhan@umich.edu}, \ \ Jak\v sa Cvitani\'c\thanks{Division of the Humanities and Social Sciences, 1200 E. California Blvd, Pasadena, CA 91125, USA, cvitanic@hss.caltech.edu}  \ \ and \ \  Yuchong Zhang\thanks{Department of Statistical Sciences, University of Toronto, 100 St.\ George Street, Toronto, Ontario M5S 3G3, Canada, yuchong.zhang@utoronto.ca}
 }
 \maketitle

\noindent\textbf{Abstract.}
We consider a stochastic tournament game in which each player is rewarded based
on her  rank in terms of the completion time of her own task and is subject to cost of effort. When players are homogeneous and the rewards are purely rank dependent, the equilibrium has a surprisingly explicit characterization, which allows us to conduct comparative statics and obtain explicit solution to several optimal reward design problems. In the general case when the players are heterogenous and payoffs are not purely rank dependent, we prove the existence, uniqueness and stability of the Nash equilibrium of the associated mean field game, and the existence of an approximate Nash equilibrium of the finite-player game.  
Our results have some potential economic implications; e.g., they lend support to government subsidies for  R\&D,  assuming the profits to be made are substantial.





\noindent \textbf{Keywords:} Tournaments, rank-based rewards, mechanism design,  mean field games,  bifurcation diagram of Nash equilibria.


\noindent \textbf{2000 Mathematics Subject Classification:}
 91A13, 91B40, 93E20



{\footnotesize{\tableofcontents}}

\section{Introduction}

To put our mathematical problem of a large population tournament into a real life context, consider the pre-Google time when a lot of players were competing to build a successful internet search engine. This is a game in which the future reward depends mostly on which player would come up first with a superior product, and not so much on the actual date of the invention. 
This is the type of tournament games we mostly  focus on in this paper: many players with the rewards depending only on the ranking of the times needed to complete a task, in which the progress is gradual rather than sudden. 
Examples include:  a population of the same generation in
a country,  with each player  trying to attain the most advanced preparation
for their future career in the shortest amount of time;
a population of   entrepreneurs during the same boom period working on getting their product or start-up to the last stage of development; more generally,
it includes the \emph{racing} type of competitions in which the progress is attained gradually.


We formulate the large population tournament as a mean field game -- a stochastic game with infinitely many small players interacting through their aggregate distribution, introduced by \cite{LasryLions.06a,LasryLions.06b,LasryLions.07}, and \cite{HuangMalhameCaines.06, HuangMalhameCaines.07}. The advantage of the mean field formulation is that the equilibrium has an appealing decentralized structure: each player bases her decisions on her own state variable and a deterministic measure of completion time distribution that is obtained from the solution of a fixed point problem.
One can then use the mean field game solution to construct an approximate Nash equilibrium of the finite-player game. Apart from our application to the analysis of tournaments, mean field games have been applied to macroeconomics by \cite{MR3268061},  analysis of bank runs by \cite{MR3679343}, systemic risk by \cite{MR3325083}, and analysis of queueing systems with strategic servers by \cite{BBC2018}. We also refer to the works of \cite{GueantLasryLions.11}, \cite{BensoussanFrehseYam.13} and \cite{CarmonaDelaRue.17a,CarmonaDelaRue.17b} for further reading.

{What sets our set-up apart} from the above is the rank-based feature of our problem: players are rewarded according to how their \emph{completion times}, which will be modeled as hitting times of controlled diffusions, are ranked. (Each player also has a cost of effort: i.e., she pays a cost for influencing the drift of her own diffusion.) 
Mean field games
with rank-based features are generally difficult
to analyze due to the lack of a priori regularity
of the rank function that depends on the equilibrium. Except for the terminal position ranking game of \cite{BZ16} and the one-stage Poisson game of \cite{MZ17}, one can usually only hope for abstract existence under the weak formulation as in \cite[Example 5.9]{CarmonaLacker15}. Hence one of the \emph{contributions} of our work is
the construction of a solvable mean field model, which is rare outside the classical linear-quadratic
framework. Moreover, even in the general setup without an explicit solution, we are still able to handle reward functions that are discontinuous in the rank variable. A key ingredient to our analysis is recognizing that the optimally controlled state density is a distortion of the density of Brownian motion by a factor related to the Cole-Hopf transformation of the value function. 
Particle systems with rank-based interaction, but no strategic actions, have been widely studied
in the literature; see e.g.\ \cite{Shkolnikov.12}, \cite{MR3055258} and the references therein. A recent paper by
\cite{NadtochiyShkolnikov.17} considers particles interacting through hitting times and presents applications in the analysis of systemic risk.

In economics, the analysis of tournaments goes back to the seminal paper of \cite{LazearRosen81}, which has inspired many papers that followed, including  \cite{AkerlofHolden2012},  \cite{BDLR17}, \cite{Siegel2017} and \cite{FNS18}, among the more recent ones. Most of these works focus on finitely many players or static models.
Recently, there has been active research on models in which the randomness  is driven by  Poisson arrivals, and the player's effort affects the probability of a breakthrough;  see e.g.\ \cite{BEM}, \cite{HKL} and \cite{MZ17}. Those models are aimed more at applications to R\&D with sudden innovation breakthroughs, whereas our model is more appropriate for the cases in which the progress is incremental. 

In our model, when the players are homogeneous in their starting point and efficiency (i.e.\ cost of effort), and the reward function is purely rank-based (except for the dependence on a given deadline), we find that the equilibrium has a semi-explicit formula. This enables us to analyze some interesting comparative statics. In particular, we find that
the aggregate welfare may be increasing in the cost of effort.
This indicates, for example, that for an economy in which building a start-up is a relatively complex endeavor, the complexity may not be such a bad thing -- if it was less complex,  too many entrepreneurs may put in an inefficiently high effort.
We also find that, when the total pie is sufficiently large, high inequality in the rewards has a demoralizing effect for many players, a phenomenon emphasized in \cite{FNS18}.
However, in our model, when the pie is sufficiently small,
the higher prize inequality improves the welfare and the average effort to a certain extent. 
This is because there is  now an effect that decreases
competitiveness -- a higher percentage of players give up, and as  the rewards become more uneven it is  worthwhile for the players who do not give up to exert higher effort.
This effect cannot arise in the setup of  \cite{FNS18} in which the size of the pie is directly influenced by the effort.\footnote{See Section~\ref{sec:un-fixed pie}, however, for a discussion of this effect  in the case in which the size of the pie depends on the completion rate.} 
In the example of internet search engines (or social media sites, or computer operating systems) the total pie is large, but the rewards are uneven, suggesting that there is loss of efficiency in that many players get discouraged from applying effort.
This logic lends some support to having government subsidies, e.g., for\ R\&D in renewable energy,
if those subsidies make the rewards more even, assuming the profits to be made are substantial.

Having a semi-explicit equilibrium also allows us to tackle the problem of \emph{mechanism design}, an area that has not been widely explored in the mean field game literature due to the lack of tractability.
Similar to \cite{ElieMastroliaPossamai.16} and \cite{MZ17}, we also work with Stackelberg equilibrium: the principal or social planner acts as the leader and designs the reward first, and the agents act as followers who then form a Nash equilibrium within themselves.

We also consider an extension of the benchmark model where the total reward, the ``pie", may depend on the population completion rate, as in institutions in which the wealth created by production
increases with the completion rate of individual tasks. In the extreme case where the reward is independent of rank and increasing in the completion rate, our game is a contribution game
in which the interaction is not through the aggregate effort towards a single project as in, for example, \cite{Georgiadis15}, but through the completion rate of many parallel projects. It turns out that, with a completion rate-dependent pie, there may be multiple equilibria all of which can be characterized in a semi-explicit manner. Furthermore, by interpolating between a pure contribution game and a pure competition game, we are able to draw a \emph{bifurcation diagram of the equilibria}.

While we do not have an explicit characterization of equilibrium when the players are heterogeneous or when the reward function is not purely rank-based,  we prove its existence using Schauder or Brouwer's fixed point theorems, and uniqueness under an additional monotonicity condition. 
We also show the stability of the fixed point, which enables us
to approximate the mean field equilibrium by the solution to a finite dimensional system of nonlinear equations.
Finally, we construct an approximate equilibrium for a game with a large finite number of players
from the mean field equilibrium.

The rest of the paper is structured as follows:
we present the model in Section~\ref{sec:setup}; the results for homogeneous players, including comparative statics, optimal reward design and extension to the completion rate-dependent pie in Section~\ref{sec:hom}; general existence, uniqueness, stability and the approximation
to finite-player games in Section~\ref{sec:general_theory}; and a numerical example with heterogeneous players in Section~\ref{sec:het}; 
the numerical method and some minor proofs are provided in Appendix.

\section{Model setup}\label{sec:setup}

We consider a game with a large number of independent players, in which each player can exert effort to move her project forward,
and is rewarded based on the time needed to complete the project and/or the ranking of that time relative to other players. The completion time is modeled as the first time that her project value or progress process reach a target level normalized to level zero. The tournament ends when a given deadline $T\in(0,\infty]$ is reached. The total amount of the reward, or the ``pie", is fixed for now, although we extend in Section~\ref{sec:un-fixed pie} some of our results to a pie that depends on the \emph{terminal completion rate}, i.e.\ the percentage of players who manage to meet the deadline.
  The reader can have in mind
 a population of technology firms competing to build a new app/website
where the reward depends on the relative time of completing a superior product (as in the search engines example from Introduction). 
   More generally, our setup might include  cases in which the players derive  utility from their ranking relative to their peers, even if the ranking is not rewarded by a monetary payment.

The reward a player gets for finishing at time $t$ with rank $r$ is given by a bounded function $R(t,r):\bbR_+\times [0,1]\mapsto \bbR$. We assume $R$ is decreasing in both variables.\footnote{Throughout the paper, increasing and decreasing are understood in the weak sense.} When $T<\infty$, we also impose $R(t,r)\equiv R_\infty$ for all $t>T$ so that everyone who fails to complete her project by the deadline receives the same minimum participation reward or incompletion penalty. Denote the set of reward functions by $\mathcal{R}$.
We start by assuming an infinite population of players, and later study how well it approximates a finite population in the limit.

\subsection{A single player's problem}

The optimization problem of a single player, in an infinite population of players, is given as follows.
Denote by
$\mu\in\mathcal{P}(\mathbb{R}_+)$ the  distribution of the completion times  of the population, and  by $F_\mu$  its cumulative distribution function (c.d.f.). The rank of  the single player who finishes  at time $t$ is given by $F_\mu(t)$, as this is  the fraction of players who finish before or at the same time as her.\footnote{In our game, $\mu$ will be  non-atomic,  hence whether the rank is measured using $F_\mu(t)$ or $F_\mu(t-)$ does not matter.}

Let $R_\mu(t):=R(t,F_{\mu} (t))$, which is then a  bounded and decreasing function of $t$.
We assume the player is risk-neutral and has quadratic cost, and that her state process $X$, representing her distance to completion, follows the stochastic differential equation (SDE)
\begin{equation}\label{SDE}
\begin{aligned}
&dX_s=-a_s ds+\sigma dB_s,
\end{aligned}
\end{equation}
where $B$ is a Brownian motion.
The player's effort process $a$ is \emph{admissible} if it is non-negative, progressively measurable with respect to the filtration of $B$,\footnote{The filtration can be larger; see Remark~\ref{rmk:filtration}.} and yields a unique strong solution of \eqref{SDE} up to the first passage time to level zero, and if that time  is non-atomic.\footnote{The non-atomic condition is for technical reasons, and can be removed if $R_\mu$ is lower semi-continuous for all $\mu\in\mathcal{P}(\mathbb{R}_+)$, which will be the case in our examples.}
Thus,
the optimization problem of the single player has the value function
\begin{equation}\label{control-prob}
v(t,x)=v(t,x;\mu,c):=\sup_a E\left[R_\mu(\tau)-\int_t^{\tau \wedge T} c a_s^2ds \bigg | X_t=x \right],
\end{equation}
where $c>0$ is the cost  parameter and
$\tau=\inf\{s\geq t: X_s=0\}$ is the completion time.

While we assume no discounting for tractability, it can be interpreted as the case in which  the cost of effort and the reward increase exactly at the same rate at which the discount factor decreases.
For example, the salaries  in a certain profession increase over time, as does the cost of education and job searching. If those increases are exactly offset by the players' discount factor, the value function is  as above.\footnote{We can also consider the case in which the reward may decrease with the interval in which completion occurs; see Remark~\ref{3.2}.}

For any $y\neq 0$, denote by  $\tau^\circ_{y}$ the first passage time of a Brownian motion to level $y$.
Its density is well known and  given by (see, e.g.,\ page 197 of \cite{KS.91})
\begin{equation}\label{FPTpdf}
f_{\tau^\circ_{y}}(s)=\frac{|y|}{s\sqrt{2\pi s}}\exp\left(-\frac{y^2}{2s}\right), \ s> 0.
\end{equation}
Also introduce
\begin{equation}\label{u}
u(t,x)=u(t,x;\mu,c):=E\left[\exp\left(\frac{R_\mu(t+\tau^\circ_{x/\sigma})}{2c\sigma^2}\right)\right]
=E\left[\exp\left(\frac{R_\mu(t+\frac{x^2}{\sigma^2}\tau^\circ_{1})}{2c\sigma^2}\right)\right].
\end{equation}
We have the following result.

\begin{prop}\label{prop:taupdf}
The value function $v(t,x)$ of the single player's problem is given by
$$v(t,x)=2c\sigma^2 \ln(u(t,x)).$$
There exists an optimal (Markovian/feedback) action $a^*$ that attains $v(t,x)$, given by
\[a^*(t,x)=-\frac{1}{2c}v_x(t,x) 1_{\{x>0\}}=-\sigma^2 \frac{u_x(t,x)}{u(t,x)} 1_{\{x>0\}}.\]
The corresponding completion time
$\tau^*=\tau^{*,t,x}$ is non-atomic and has probability density function (p.d.f.)
\begin{equation}\label{taupdf}
f_{\tau^*}(s)=\frac{u(s,0)}{u(t,x)}f_{\tau^\circ_{x/\sigma}}(s-t), \quad s\ge t.
\end{equation}
\end{prop}

\begin{proof}
Let $u$ be defined by \eqref{u}, $v:=2c\sigma^2 \ln u$ and $a^\ast:=-v_x/(2c)1_{\{x>0\}}=-\sigma^2 (u_x/u) 1_{\{x>0\}}$.

Step 1. Show that $v$ is a solution to the Hamilton-Jacobi-Bellman (HJB) equation:
\begin{equation}\label{HJB}
v_t+\sup_{a}\left\{-av_x+\frac{1}{2}\sigma^2 v_{xx}-ca^2\right\}=0,
\end{equation}
with boundary condition $v(t,0)=R_\mu(t)$ for $0\le t\le T$, and when $T<\infty$, also the terminal condition $v(T,x)=R_\infty=R_\mu(T+)$ for $x>0$.

By the first order condition, \eqref{HJB} is equivalent to
 \[v_t+\frac{1}{4c}(v_x)^2+\frac{1}{2}\sigma^2 v_{xx}=0,\]
 where the supremum is attained by $a^*$ pointwise.
To show $v$ is a solution to the above equation with the desired boundary and terminal conditions, it is equivalent to show its Cole-Hopf transformation $u=e^{(2c\sigma^2)^{-1}v}$ satisfies
 \begin{equation}\label{heateq}
u_t+\frac{1}{2}\sigma^2 u_{xx}=0,
\end{equation}
with
\[u(t,0)=\exp\left(\frac{R_\mu(t)}{2c\sigma^2}\right), \quad u(T,x)=\exp\left(\frac{R_\infty}{2c\sigma^2}\right) \text{ if } T<\infty.\]
But this is immediate from the Feynman-Kac theorem.

Step 2. Show that the SDE \eqref{SDE} controlled by $a^*\ge 0$ has a unique strong solution $X^*(=X^{*,t,x})$ up to $\tau^*(=\tau^{*,t,x})$, the first passage time of $X^*$ to level zero.

By the monotonicity and boundedness of $R_\mu$, $u$ is decreasing in both $t$ and $x$, and
\begin{equation}\label{u-bdd}
\exp\left(\frac{R_\infty}{2c\sigma^2}\right) \le  u(t,x)\le
\exp\left(\frac{R(0,0)}{2c\sigma^2}\right).
\end{equation}
Since $\tau^\circ_1$ has a smooth density, even if  $R_\mu$ is not continuous, it is not hard to see that $u\in C^{1,2}(\bbR_+\times(0,\infty))$. In fact, direct differentiation of \eqref{u} yields
\begin{align*}
u_x(t,x)= \frac{1}{x}E\left[\exp\left(\frac{R_\mu(t+\frac{x^2}{\sigma^2} \tau^\circ_{1})}{2c\sigma^2}\right)\left(1-\frac{1}{\tau^\circ_{1}}\right)\right],
\end{align*}
and
\[u_{xx}(t,x)= \frac{1}{x^2}E\left[\exp\left(\frac{R_\mu(t+\frac{x^2}{\sigma^2} \tau^\circ_{1})}{2c\sigma^2}\right)\left(\frac{1}{ (\tau^\circ_{1})^2}-\frac{3}{\tau^\circ_{1}}\right)\right].\]
Note that $1/\tau_{1}^{\circ}$ has finite moments.\footnote{The moments of $1/\tau_{1}^{\circ}$ can be computed using \[P(1/\tau_1^{\circ}\geq t)=P(\tau^\circ_1\leq 1/t)=2P(W(1/t)\geq 1)=2\int_{\sqrt{t}}^\infty \frac{1}{\sqrt{2\pi}}e^{-z^2/2} dz,\]
and $E[(1/\tau_1^{\circ})^p]=\int_0^\infty p t^{p-1} P(1/\tau_1^{\circ}\geq t)dt$.
In particular, we have $E[1/\tau_1^{\circ}]=1$ and $E[(1/\tau_1^{\circ})^2]=3$.}
So $u_x$ and $u_{xx}$ are bounded on any region away from $x=0$. It follows that $a^*$ is bounded, and Lipschitz continuous in $x$ on any compact subset of $(0,\infty)$. Standard SDE theory then implies the existence of a unique strong solution up to $\tau^*$.

Step 3. Show that $\tau^*$ is non-atomic and has the desired p.d.f.\ \eqref{taupdf}.

We employ a change of probability measure argument. Let $B^\circ$ be a Brownian motion under some probability measure $\bbP^\circ$. Define $X_s:=x+\sigma (B^\circ_s-B^\circ_t)$ and $\tau:=\inf\{s\ge t: X_s=0\}$.
Let $u$ be the function given by \eqref{u}. Define
\[Z_s:=\frac{u(s\wedge \tau, X_{s\wedge \tau})}{u(t,x)}.\]
We have $Z_t=1$ and that the paths of $Z$ are $\bbP^\circ$-a.s.\ continuous. To see the latter, observe that if $R_\mu(s)$ is continuous at $s=\tau(\omega)$, then $\lim_{s\rightarrow \tau(\omega)}Z_s=\exp\left(\frac{R_\mu(\tau(\omega))}{2c\sigma^2}\right)/u(t,x)=Z_{\tau(\omega)}$. Since $R_\mu$ is decreasing, it has at most countably many points of discontinuity. Since $\tau$ is non-atomic under $\bbP^\circ$, with $\bbP^\circ$-probability one, $R_\mu(s)$ is continuous at $s=\tau$, and thus, the $s\mapsto Z_s$ is continuous at $s=\tau$. Path continuity before time $\tau$ is trivial.

By It\^o's lemma,
\[dZ_s=\begin{cases}
\frac{\sigma u_x(s,X_s)}{u(t,x)} dB^{\circ}_s, & t\le s<\tau,\\
0, & s\ge \tau,
\end{cases}\]
where the drift of $Z$ prior to time $\tau$ is killed using \eqref{heateq}. Since $Z$ is bounded, it is a $\bbP^\circ$-martingale. Hence we can define a probability measure $\bbP$ via $d\bbP/d\bbP^\circ=Z_\infty=Z_\tau$. Since $Z_\tau$ is strictly positive, $\bbP$ is equivalent to $\bbP^\circ$.

Next, we rewrite the dynamics of $Z$ as $dZ_s=Z_s dY_s$ where
\[dY_s=\begin{cases}
\frac{\sigma u_x}{u}(s,X_s) dB^{\circ}_s, & t\le s<\tau.\\
0, & s\ge \tau.
\end{cases}\]
We see that $Z$ is the stochastic exponential of $Y$. (We could set $Z_s\equiv 1$ and $dY_s\equiv 0$ for $s<t$ if necessary.) Girsanov Theorem implies that
\[B_s:=B^\circ_s-1_{\{s\ge t\}}\int_t^{s\wedge \tau} \frac{\sigma u_x}{u}(r,X_r) dr\]
is a $\bbP$-Brownian motion. Replacing $B^\circ$ by $B$ in the dynamics of $X$, we obtain
\[dX_s=\sigma^2 \frac{u_x}{u}(s,X_s)1_{\{s<\tau\}} ds+\sigma dB_s=-a^*(s,X_s)ds+\sigma dB_s, \quad s\ge t.\]
That is, $X$ has the same distribution under $\bbP$ as the process $X^\ast$ under the candidate  optimal effort $a^*$.
Under $\bbP$, we have
\[\bbP(\tau\in ds)=\frac{u(s,0)}{u(t,x)} \bbP^\circ(\tau\in ds), \quad s\ge t,\]
from which we conclude that $\tau$ (and hence $\tau^*$) is non-atomic and has the desired p.d.f.

Step 4. Verify that $v$ is indeed the value function and that $a^*$ is indeed the optimal effort.

The verification argument is standard except for an extra localization step to take care of the potential singularity of $v_x$ near $x=0$. Specifically, fixing an arbitrary admissible control $a$ and the associated state process $X$, we apply It\^{o}'s lemma to $v(s, X_s)$ from time $t$ to $\tau_\eps \wedge T$, where $\eps>0$ and $\tau_\eps:=\{s\ge t: X_s=\eps\}$. This leads to
\begin{align*}
Ev(\tau_\eps\wedge T, X_{\tau_\eps\wedge T})&=v(t,x)+E\int_t^{\tau_\eps\wedge T} \left(v_t(s,X_s)-a_sv_x(s,X_s)+\frac{1}{2}\sigma^2 v_{xx}(s, X_s)\right)ds\\
&\le v(t,x)+E\int_t^{\tau_\eps\wedge T} ca_s^2 ds,
\end{align*}
where the inequality holds because $v$ satisfies the HJB equation \eqref{HJB}. It follows that
\begin{equation}\label{eq:verification}
v(t,x)\ge E v(\tau_\eps\wedge T, X_{\tau_\eps\wedge T})-E\int_t^{\tau_\eps\wedge T} c a_s^2 ds.
\end{equation}
Now, as $\eps\searrow 0$, since $\tau_\eps$ is increasing and bounded by $\inf\{s\ge t: x+\sigma (B_s-B_t)=0\}$, it has a limit $\tau_0\le \tau :=\inf \{s\ge t: X_s=0\}$.
We claim that $\tau_0=\tau$. Indeed, if $X_{\tau_0}>0$, then by path continuity, $\eps=X_{\tau_\eps}>\delta$ for some $\delta>0$, for all $\eps$ sufficiently small, which is a contradiction.

By monotone convergence theorem, $E\int_t^{\tau_\eps\wedge T} ca_s^2 ds$ converges to $E\int_t^{\tau\wedge T} ca_s^2 ds$ as $\eps\searrow 0$. To obtain the limit of $Ev(\tau_\eps\wedge T, X_{\tau_\eps\wedge T})$, note that $\tau_\eps\wedge T\rightarrow \tau\wedge T$ and $X_{\tau_\eps\wedge T}\rightarrow X_{\tau\wedge T}$. If $\tau>T$, then $X_{\tau\wedge T}=X_T>0$, and $v(\tau_\eps\wedge T, X_{\tau_\eps\wedge T})\rightarrow v(T, X_T)=R_\infty=R_\mu(\tau)$ by the continuity of $v$ in $\bbR_+\times (0,\infty)$. If $\tau\le T$, then $X_{\tau\wedge T}=X_{\tau}=0$, and $v(\tau_\eps\wedge T, X_{\tau_\eps\wedge T})\rightarrow v(\tau, 0)=R_\mu(\tau)$
if $\tau$ is a continuity point of $R_\mu$. Since $R_\mu$ has at most countably many points of discontinuity (by monotonicity) and $\tau$ is non-atomic by admissibility, we have, after combining the two cases, that $v(\tau_\eps\wedge T, X_{\tau_\eps\wedge T})$ converges to $R_\mu(\tau)$ a.s.\footnote{If $R_\mu$ is lower semi-continuous, then we can directly obtain $\liminf_{\eps}v(\tau_\eps\wedge T, X_{\tau_\eps\wedge T})\ge R_\mu(\tau)$ without the non-atomic property of $\tau$.} Bounded convergence theorem then implies $Ev(\tau_\eps\wedge T, X_{\tau_\eps\wedge T})\rightarrow ER_\mu(\tau)$.
So, letting $\eps\searrow 0$ in \eqref{eq:verification} yields
\[v(t,x)\ge E\left[R_\mu(\tau)-\int_t^{\tau\wedge T} c a_s^2 ds\right].\]
Since $a$ is an arbitrary admissible control, taking supremum over $a$ leads to the conclusion that $v$ dominates the value function.

Replacing  $a$ by $a^*$ which is admissible by Steps 2 and 3, all inequalities above become equalities, which implies that $v(t,x)=E\left[R_\mu(\tau^*)-\int_t^{\tau^*\wedge T} c (a^*(s,X^*_s))^2 ds\right]$ is dominated by the value function. Putting everything  together, we conclude the verification proof.
\end{proof}

\begin{remark}
One can also obtain \eqref{taupdf} via an analytic argument. To see this, note that the Fokker-Planck equation
\[\rho_s(s,y)=[a^*(s,y)\rho_y(s,y)]_y+\frac{1}{2}\sigma^2 \rho_{yy}(s,y), \quad \rho(t,y)=\delta_x(y)\]
with absorbing boundary condition $\rho(s,0)=0$ has a semi-explicit solution
\[\rho(s,y)=\frac{u(s,y)}{u(t,x)}\varphi(s,y),\]
where $\varphi$ is the solution to
\[\varphi_s=\frac{1}{2}\sigma^2\varphi_{yy}, \quad \varphi(t,y)=\delta_x(y), \quad \varphi(s,0)=0.\]
From this, one can derive \eqref{taupdf} by repeated integration by parts:
\begin{align*}
\bbP(\tau\in ds)&=-\frac{d}{ds}\bbP(\tau> s)=-\frac{d}{ds} \int_0^\infty \rho(s,y)dy=-\int_0^\infty \rho_s(s,y)dy\\
&=-\frac{1}{u(t,x)}\int_0^\infty (u_s\varphi +u \varphi_s)(s,y) dy \\
&=-\frac{\sigma^2}{2u(t,x)}\int_0^\infty (-u_{yy}\varphi +u \varphi_{yy})(s,y) dy\\
&=\frac{\sigma^2}{2u(t,x)}u(s,0)\varphi_y(s,0)
\end{align*}
and
\[\bbP^\circ(\tau\in ds)=-\int_0^\infty \varphi_s(s,y)dy=-\frac{1}{2}\sigma^2\int_0^\infty \varphi_{yy}(s,y)dy=\frac{1}{2}\sigma^2\varphi_y(s,0).\]
\end{remark}

\begin{remark}\label{rmk:filtration}
From the the verification step of the proof of Proposition~\ref{prop:taupdf}, it is not hard to see that the feedback control $a^*$ remains optimal if admissible controls are progressively measurable with respect to a filtration larger than the one generated by the driving Brownian motion $B$, as long as $B$ remains a Brownian motion for that filtration. Later we will use this fact in the construction of an approximate Nash equilibrium for the $N$-player game in which each player observes not only her private state, but also the states of her competitors.
\end{remark}

{Having computed the best response of a single player to a given distribution $\mu$, we are now ready to study the Nash equilibrium of this infinite population game. Supposing every player uses the optimal feedback control $a^*=a^*(\cdot; \mu,c)$ given by Propsition~\ref{prop:taupdf}, we obtain a new population completion time distribution. If this new distribution is equal to $\mu$, then we say $\mu$ is an equilibrium (completion time distribution). We first consider a case in which equilibrium can be fully characterized 
in Section~\ref{sec:hom}, 
and then discuss the general case in Section~\ref{sec:general_theory}.}

\section{Homogeneous players}\label{sec:hom}

In this section, we assume a unit mass of homogeneous players with independent Brownian motions driving their state processes.
By homogeneous, we mean
the players all start playing at the same time $t=0$ and the same distance $x_0>0$ from the goal, and they all have the same cost parameter $c$.
Moreover,  we specialize to the purely rank-based reward functions of the form
\begin{equation}\label{eq:PRBR}
R(t,r)=1_{\{t\le T\}}H(r)+1_{\{t>T\}}R_\infty,
\end{equation}
where $H\ge R_\infty$ is a bounded decreasing function.
That is, the only time dependence is through the deadline $T$. We will show below that there always exists a unique equilibrium in semi-explicit form.
The proof is based on considering equation \eqref{taupdf} as a fixed point equation  for the p.d.f.\ of the population completion time.

\subsection{Explicit characterization and properties of the equilibrium}

Denote by $T^\mu_r$  the $r$-th quantile of a completion time distribution $\mu$, and by $N(x)$ the c.d.f.\ of the standard normal distribution. We have the following results.

\begin{thm}\label{thm:explicit_soln}
Suppose that $R\in\mathcal{R}$ is of the form \eqref{eq:PRBR}.

(i) For  $T<\infty$, the unique equilibrium completion time distribution $\mu$ has a quantile function given   by
\begin{equation}\label{eq:NEquantile}
T^\mu_r=F^{-1}_{\tau^\circ_{x_0/\sigma}}\left(\frac{1-F_{\tau^\circ_{x_0/\sigma}}(T)}{1-F_\mu(T)}\int_0^r \exp\left(\frac{R_\infty-H(z)}{2c\sigma^2}\right)dz\right), \quad r\in [0, F_\mu(T)],
\end{equation}
where
\begin{equation}\label{eq:BMFPTcdf}
F_{\tau^\circ_{x_0/\sigma}}(t)=2\left(1-N\left(\frac{x_0}{\sigma \sqrt{t}}\right)\right),
\end{equation}
and the equilibrium terminal completion rate $F_\mu(T)\in (0,1)$ is the unique solution of
\begin{equation}\label{eq:NErate}
F_{\tau^\circ_{x_0/\sigma}}(T)=\frac{1-F_{\tau^\circ_{x_0/\sigma}}(T)}{1-F_\mu(T)}\int_0^{F_\mu(T)} \exp\left(\frac{R_\infty-H(z)}{2c\sigma^2}\right)dz.
\end{equation}
Moreover, the value of the game is given by
\begin{equation}\label{eq:NEvalue}
V=v(0,x_0;\mu,c)=R_\infty+2c\sigma^2\ln \left(\frac{1-F_{\tau^\circ_{x_0/\sigma}}(T)}{1-F_\mu(T)}\right).
\end{equation}

(ii) For $T=\infty$, the unique equilibrium completion time distribution $\mu$ has a  quantile function given  by
\begin{equation}\label{eq:infNEquantile}
T^\mu_r=F^{-1}_{\tau^\circ_{x_0/\sigma}}\left(\frac{\int_0^r  \exp\left(-\frac{H(z)}{2c\sigma^2}\right) dz}{\int_0^1\exp\left(-\frac{H(z)}{2c\sigma^2}\right) dz}\right).
\end{equation}
Moreover, the value of the game is given by
\begin{equation}\label{eq:infNEvalue}
V_\infty=v(0,x_0;\mu,c)=-2c\sigma^2\ln \left(\int_0^1\exp\left(-\frac{H(z)}{2c\sigma^2}\right) dz\right).
\end{equation}
\end{thm}

\begin{proof}
(i) Suppose $T<\infty$.
By \eqref{taupdf},
the fixed point equation (in terms of the p.d.f.\ of the completion time distribution) is
\[f_\mu(t)=\frac{u(t,0;\mu,c)}{u(0,x_0;\mu,c)}f_{\tau^\circ_{x_0/\sigma}}(t), \quad t\in[0,T].\]
Denote $y(r)=F_{\tau^\circ_{x_0/\sigma}}(T^\mu_r)$.
Since any fixed point $\mu$ has positive density on $[0,T]$, $y(r)$ is differentiable on $[0, F_\mu(T)]$ with
\begin{align*}
y'(r)&=\frac{f_{\tau^\circ_{x_0/\sigma}}(T^\mu_r)}{f_\mu(T^\mu_r)}=\frac{u(0,x_0;\mu,c)}{u(T^\mu_r,0;\mu,c)}\\
&=\frac{\int_0^T \exp\left(\frac{H(F_\mu(s))}{2c\sigma^2}\right)dF_{\tau^\circ_{x_0/\sigma}}(s)+\exp\left(\frac{R_\infty}{2c\sigma^2}\right)(1-F_{\tau^\circ_{x_0/\sigma}}(T))}{\exp\left(\frac{H(F_\mu(T^\mu_r))}{2c\sigma^2}\right)}\\
&=\frac{\int_0^{F_\mu(T)} \exp\left(\frac{H(z)}{2c\sigma^2}\right)y'(z)dz+\exp\left(\frac{R_\infty}{2c\sigma^2}\right)(1-F_{\tau^\circ_{x_0/\sigma}}(T))}{\exp\left(\frac{H(r)}{2c\sigma^2}\right)}.
\end{align*}
Observe that $\exp\left(\frac{H(r)}{2c\sigma^2}\right) y'(r)$ is independent of $r$, hence is constant for $r\in [0, F_\mu(T)]$. Let $C$ be this constant. We have
\[C=CF_\mu(T)+\exp\left(\frac{R_\infty}{2c\sigma^2}\right)(1-F_{\tau^\circ_{x_0/\sigma}}(T)).\]
Since $F_{\tau^\circ_{x_0/\sigma}}(T)<1$, the above equation implies $F_\mu(T)<1$ and
\[\exp\left(\frac{H(r)}{2c\sigma^2}\right) y'(r)\equiv C=\frac{1-F_{\tau^\circ_{x_0/\sigma}}(T)}{1-F_\mu(T)}\exp\left(\frac{R_\infty}{2c\sigma^2}\right), \quad r\in[0, F_\mu(T)].\]
It follows that
\begin{equation*}
y(r)=\frac{1-F_{\tau^\circ_{x_0/\sigma}}(T)}{1-F_\mu(T)}\int_0^r \exp\left(\frac{R_\infty-H(z)}{2c\sigma^2}\right)dz, \quad r\in [0, F_\mu(T)].
\end{equation*}
Setting $r=F_\mu(T)$ yields  equation \eqref{eq:NErate} for $F_\mu(T)$. The existence and uniqueness of solution follows from the fact that
\begin{equation}\label{eq:phi}
\phi(r):=\frac{1}{1-r}\int_0^r  \exp\left(\frac{R_\infty-H(z)}{2c\sigma^2}\right)dz
\end{equation}
is a continuous, strictly increasing function on $[0,1)$ satisfying $\phi(0)=0$ and $\lim_{r\rightarrow 1}\phi(r)=\infty$. Finally, we have
$v(0,x_0;\mu,c)=2c\sigma^2 \ln u(0,x_0;\mu,c)$ and
\begin{align*}
u(0,x_0;\mu,c)
&=C F_\mu(T)+\exp\left(\frac{R_\infty}{2c\sigma^2}\right)(1-F_{\tau^\circ_{x_0/\sigma}}(T))=\frac{1-F_{\tau^\circ_{x_0/\sigma}}(T)}{1-F_\mu(T)}\exp\left(\frac{R_\infty}{2c\sigma^2}\right)
\end{align*}
from which \eqref{eq:NEvalue} follows.

(ii) Suppose $T=
\infty$.
First, note that any possible fixed point $\mu$ should have strictly increasing c.d.f.\ which satisfies $F_\mu(\infty)=1$ since the effort is non-negative and the first passage time of a Brownian motion is almost surely finite. Let $y(r)$ be defined as above. A similar calculation shows that
\[\exp\left(\frac{H(r)}{2c\sigma^2}\right) y'(r)=\int_0^{1} \exp\left(\frac{H(z)}{2c\sigma^2}\right)y'(z)dz\]
is constant (denoted by $C$) and
\[F_{\tau^\circ_{x_0/\sigma}}(T^\mu_r)=y(r)=C\int_0^r  \exp\left(-\frac{H(z)}{2c\sigma^2}\right) dz.\]
 We can then use $y(1)=1$ to find that $C=\left(\int_0^1\exp\left(-\frac{H(z)}{2c\sigma^2}\right) dz\right)^{-1}$. Having obtained $y(r)$, we have $T^\mu_r=F^{-1}_{\tau^\circ_{x_0/\sigma}}(y(r))$. The equilibrium value is
$v(0,x_0;\mu,c)=2c\sigma^2 \ln u(0,x_0;\mu,c)=2c\sigma^2 \ln C$.
\end{proof}

\begin{remark}\label{rmk:Tconvergence}
Observe that the finite horizon equilibrium converges to the infinite horizon equilibrium as $T\rightarrow \infty$. To see this, use \eqref{eq:NErate} to rewrite the finite horizon equilibrium quantile function and value of the game as
\[T^\mu_r(T)=F^{-1}_{\tau^\circ_{x_0/\sigma}}\left(\frac{\int_0^r \exp\left(\frac{R_\infty-H(z)}{2c\sigma^2}\right)dz}{\int_0^{F_\mu(T)} \exp\left(\frac{R_\infty-H(z)}{2c\sigma^2}\right)dz}F_{\tau^\circ_{x_0/\sigma}}(T)\right), \quad r\in[0, F_\mu(T)]\]
and
\[V(T)=R_\infty+2c\sigma^2 \ln\left(\frac{F_{\tau^\circ_{x_0/\sigma}}(T)}{\int_0^{F_\mu(T)} \exp\left(\frac{R_\infty-H(z)}{2c\sigma^2}\right)dz}\right)=2c\sigma^2 \ln\left(\frac{F_{\tau^\circ_{x_0/\sigma}}(T)}{\int_0^{F_\mu(T)} \exp\left(\frac{-H(z)}{2c\sigma^2}\right)dz}\right).\]
As $T\rightarrow\infty$, $F_{\tau^\circ_{x_0/\sigma}}(T)/(1-F_{\tau^\circ_{x_0/\sigma}}(T))\rightarrow \infty$. For \eqref{eq:NErate} to hold, we must have $\lim_{T\rightarrow \infty} F_\mu(T)=1$. It follows that the finite horizon equilibrium $T^\mu_r(T)$ and $V(T)$ converge to their infinite horizon counterparts as $T\rightarrow \infty$, where the convergence is pointwise for the quantile function.
\end{remark}

{\begin{remark}\label{3.2}
More generally, if $R\in\mathcal{R}$ is of the form
\[R(t,r)=\sum_{k=1}^n 1_{\{T_{k-1}<t\le T_k\}} \delta_k H(r)+1_{\{t>T\}}R_\infty,\]
where $0=T_0<T_1<\cdots<T_n=T<\infty$, then with $\alpha_k:=F_{\tau^\circ_{x_0/\sigma}}(T_k)$ and $\beta_k:=F_\mu(T_k)$, the equilibrium quantile function is given by
\[T^\mu_r=F^{-1}_{\tau^\circ_{x_0/\sigma}}\left(\frac{1-\alpha_n}{1-\beta_n} \sum_{k=1}^n \int_{\beta_{k-1}\wedge r}^{\beta_k\wedge r} \exp\left(\frac{R_\infty- \delta_k H(z)}{2c\sigma^2}\right)dz\right), \quad r\in [0, \beta_n],\]
where $\beta_1, \ldots, \beta_n$ can be found by solving the following nonlinear system of equations:
\[\alpha_k-\alpha_{k-1}=\frac{1-\alpha_n}{1-\beta_n}\int_{\beta_{k-1}}^{\beta_k} \exp\left(\frac{R_\infty-\delta_k H(z)}{2c\sigma^2}\right)dz, \quad k=1, \ldots, n.\]
The game value in this case is $V=R_\infty+2c\sigma^2\ln \left(\frac{1-\alpha_n}{1-\beta_n}\right).$
Similarly, a semi-explicit formula can be derived for the case $T=\infty$.
\end{remark}

The  semi-explicit solution allows us to obtain some comparative statics analytically. The proof is based on elementary calculus and is provided in the appendix. 

\begin{prop}\label{prop:homcs}
Suppose that $R\in\mathcal{R}$ is of the form \eqref{eq:PRBR} with $H$ continuous. Then,  the terminal equilibrium completion rate $\beta:=F_\mu(T)$ and the game value $V$ (written as $V_\infty$ when $T=\infty$) have the following properties.
\begin{itemize}
\item[(i)] When $T<\infty$, $\beta$ is increasing in $T$ and decreasing in $x_0$ and $c$, where the monotonicity in $T$ and $x_0$ are strict. Moreover,
\[
\lim_{T\rightarrow 0} \beta =0, \ \lim_{T\rightarrow \infty} \beta =1, \ \lim_{x_0\rightarrow 0}\beta=1,\ \lim_{x_0\rightarrow \infty}\beta=0, \ \lim_{c\rightarrow 0} \beta=1, \ \lim_{c\rightarrow \infty} \beta=F_{\tau^\circ_{x_0/\sigma}}(T).
\]
\item[(ii)] When $T<\infty$, $V$ is increasing in $T$ and decreasing in $x_0$. Moreover,
\[\lim_{T\rightarrow 0} V =R_\infty, \quad \lim_{T\rightarrow \infty} V =V_\infty, \quad \lim_{x_0\rightarrow 0}V=V_\infty, \quad \lim_{x_0\rightarrow \infty}V=R_\infty.\]
\item[(iii)] $V_\infty$ is independent of $x_0$, and increasing in $c$ with
\[\lim_{c\rightarrow 0} V_\infty=H(1-),\footnote{$H(r-)$ denotes the left limit of $H$ at $r$. The one-sided limit exists because $H$ is bounded and monotone.} \quad \lim_{c\rightarrow \infty} V_\infty=\int_0^1 H(r)dr.\]
\item[(iv)] If $H_1\ge H_2$, then $V_\infty(H_1)\ge V_\infty(H_2)$. When $T<\infty$ and $R_i=1_{\{t\le T\}} H_i+1_{\{t>T\}} R_{i,\infty}$, if $H_1-R_{1,\infty} \ge H_2-R_{2,\infty}$, then
$\beta(R_1)\ge \beta(R_2)$ and $V(R_1)-R_{1,\infty}\ge V(R_2)-R_{2,\infty}$.
\end{itemize}
\end{prop}

The results in items (i), (ii) and (iv) are intuitive and not surprising. Somewhat surprising, at first sight, might be the fact stated in (iii) that
the value $V_\infty$ of the infinite horizon game is increasing in the cost parameter $c$. Let us  consider the limiting cases.
When the cost goes to zero,  all the players will apply very  high effort and the value will converge  to the  value  $H(1)$ (if $H$ is left-continuous at $1$) of the lowest ranked player.
On the other hand, when the cost goes to infinity, they will apply very small effort and the project values be driven by pure noise,
which results in an aggregate gain of $\int_0^\infty H(F_\mu(t)) dF_\mu(t)=\int_0^1 H(r)dr$, where the equilibrium measure $\mu$ is distributed as $\tau^\circ_{x_0/\sigma}$ in the limit.
Since $H$ is decreasing, the averaged value is higher than  $H(1)$.

In effect, low cost incentivizes   players to apply too much effort in competing
with each other, without resulting in good ranking. It's  a rat race without winners. For the aggregate welfare it does not matter which of the players finish early,
and competing too hard with each other to reach the goal sooner  decreases the welfare. Thus,  in this game it is beneficial if the players are discouraged  from working too hard by having a high cost of effort.
For example, if building a start-up was too easy, too many entrepreneurs may apply too high effort.
 However,with a finite deadline,  as we will see in the next subsection, the effect of the cost on the value may be increasing or decreasing. Moreover,  if the pie was not fixed, but, for example, if it grew with the population completion rate and the speed of completion, then  lower cost may lead to higher value.

We end the theoretical analysis with a result on the expected total effort, needed later below for further comparative statics.
\begin{prop}\label{prop:expected-effort}
Let $T<\infty$. Suppose that $R\in\mathcal{R}$ is of the form \eqref{eq:PRBR}.
Let $\mu$ be the unique equilibrium completion time distribution given by Theorem~\ref{thm:explicit_soln}(i), and $a^*_t:=a^*(t,X_t;\mu,c)$, $X$, $\tau$ be the associated equilibrium action, state process and completion time, respectively.
Then, the expected total effort in equilibrium is given by
\begin{equation}\label{eq:expected-effort}
E \int_0^{\tau\wedge T} a^*_t dt=\frac{x_0\left(F_\mu(T)-F_{\tau^\circ_{x_0/\sigma}}(T)\right)}{1-F_{\tau^\circ_{x_0/\sigma}}(T)}.
\end{equation}
\end{prop}
\begin{proof}
From $X_{\tau\wedge T}=x_0-\int_0^{\tau\wedge T} a^*_t dt+\sigma B_{\tau\wedge T}$, we get
\[E \int_0^{\tau\wedge T} a^*_t dt=x_0-EX_{\tau\wedge T}.\]
To compute $EX_{\tau\wedge T}$, we make use of the change of measure introduced in the proof of Proposition~\ref{prop:taupdf}. Let $\bbP^\circ$ and $Z_\infty$ as in step 3 of the proof of Proposition~\ref{prop:taupdf} with $(t,x)=(0,x_0)$. We have
\begin{align*}
EX_{\tau\wedge T}&=E[1_{\{\tau>T\}}X_T]=E^{\bbP^\circ}[Z_\infty 1_{\{\tau>T\}}X_T]\\
&=E^{\bbP^\circ}\left[\frac{u(\tau,X_\tau;\mu,c)}{u(0,x_0;\mu,c)} 1_{\{\tau>T\}}X_T\right]=E^{\bbP^\circ}\left[\frac{\exp\left(\frac{R_\infty}{2c\sigma^2}\right)}{u(0,x_0;\mu,c)} 1_{\{\tau>T\}}X_T\right]\\
&=\frac{\exp\left(\frac{R_\infty}{2c\sigma^2}\right)}{u(0,x_0;\mu,c)}E^{\bbP^\circ}X_{\tau\wedge T}
\end{align*}
Since $X$ is a martingale under $\bbP^\circ$, we have $E^{\bbP^\circ}X_{\tau\wedge T}=x_0$. By Theorem~\ref{thm:explicit_soln}(i),
\[u(0,x_0;\mu,c)=\exp\left(\frac{v(0,x_0;\mu,c)}{2c\sigma^2}\right)=\frac{1-F_{\tau^\circ_{x_0/\sigma}}(T)}{1-F_\mu(T)}\exp\left(\frac{R_\infty}{2c\sigma^2}\right).\]
Putting all pieces together, we have
\[E \int_0^{\tau\wedge T} a^*_t dt=x_0-\frac{1-F_\mu(T)}{1-F_{\tau^\circ_{x_0/\sigma}}(T)}x_0=\frac{x_0\left(F_\mu(T)-F_{\tau^\circ_{x_0/\sigma}}(T)\right)}{1-F_{\tau^\circ_{x_0/\sigma}}(T)}.\]
\end{proof}

We see that the expected total effort increases with the completion rate $F_\mu(T)$ when $x_0, \sigma, T$ are fixed. Since $F_\mu(T)$ decreases with $c$, so does the expected total effort.

\subsection{The finite deadline: numerical comparative statics}

We now perform a numerical study of the equilibrium in games with a finite deadline. The benchmark choice of model inputs are $T=1$, $\sigma=0.25$, $x_0=1$, $c=1$ and $R_\infty=0$. Whenever we vary one parameter, we keep the other parameters fixed.

Figures~\ref{fig:hom1} and \ref{fig:hom2} illustrate some features of the equilibrium with a smooth reward function $H(r)=6(1-r)^2$, and a step reward function $H(r)=5\cdot1_{[0,0.25)}(r)+2\cdot1_{[0.25,0.5)}(r)+1_{[0.5,1]}(r)$, respectively. We can see the following:

\begin{itemize}
\item For a large subset of time-location pairs $(t,x)$, the effort is  low.
 The effort is high initially when there are a lot of players around the same level of progress, and
close to the cutoff dates that distinguish between different rank rewards (the only such cutoff date being the deadline in the case of a smooth reward).
Thus in Figure~\ref{fig:hom2}, there is, for example, a lot of completion in the short interval before $F_\mu$ hits 0.25 because 0.25 is the cutoff point for a higher reward. Once that point is crossed, the players apply low effort until getting close to the next cutoff point, so that the completion rate increases very  slowly in the period between a cutoff point and close to the next cutoff point.

\item Moreover, the effort is very low  (but not zero) when it is hard to complete the task due to distant location or high cost,  or when the reward too small. In those cases (figures not shown) we find that the players maintain a  low effort, relying on randomness of the project to bring them closer to completion, which happens with low probability. 

\item The c.d.f\ of the completion time for early times is close to, but not exactly equal to zero,  because some players finish early by sheer randomness, even if they apply low effort. After those, there is a whole bunch of players who apply the optimal strategy and,  when the reward is piecewise constant,  
finish about the same time. This is where the first sudden increase in the c.d.f\ shows up. 

\end{itemize}

\begin{figure}[t]
\centering
\includegraphics[height=5.5cm]{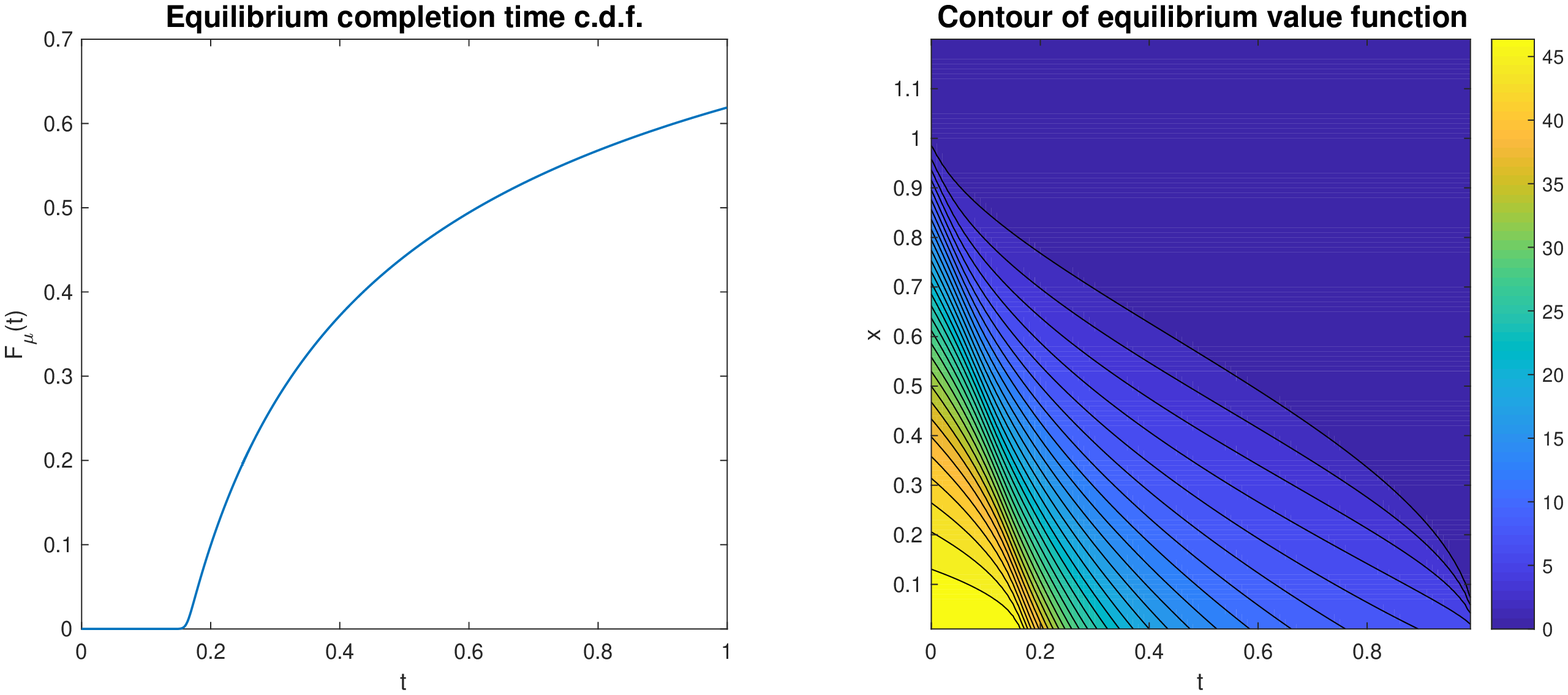}
\includegraphics[height=5.5cm]{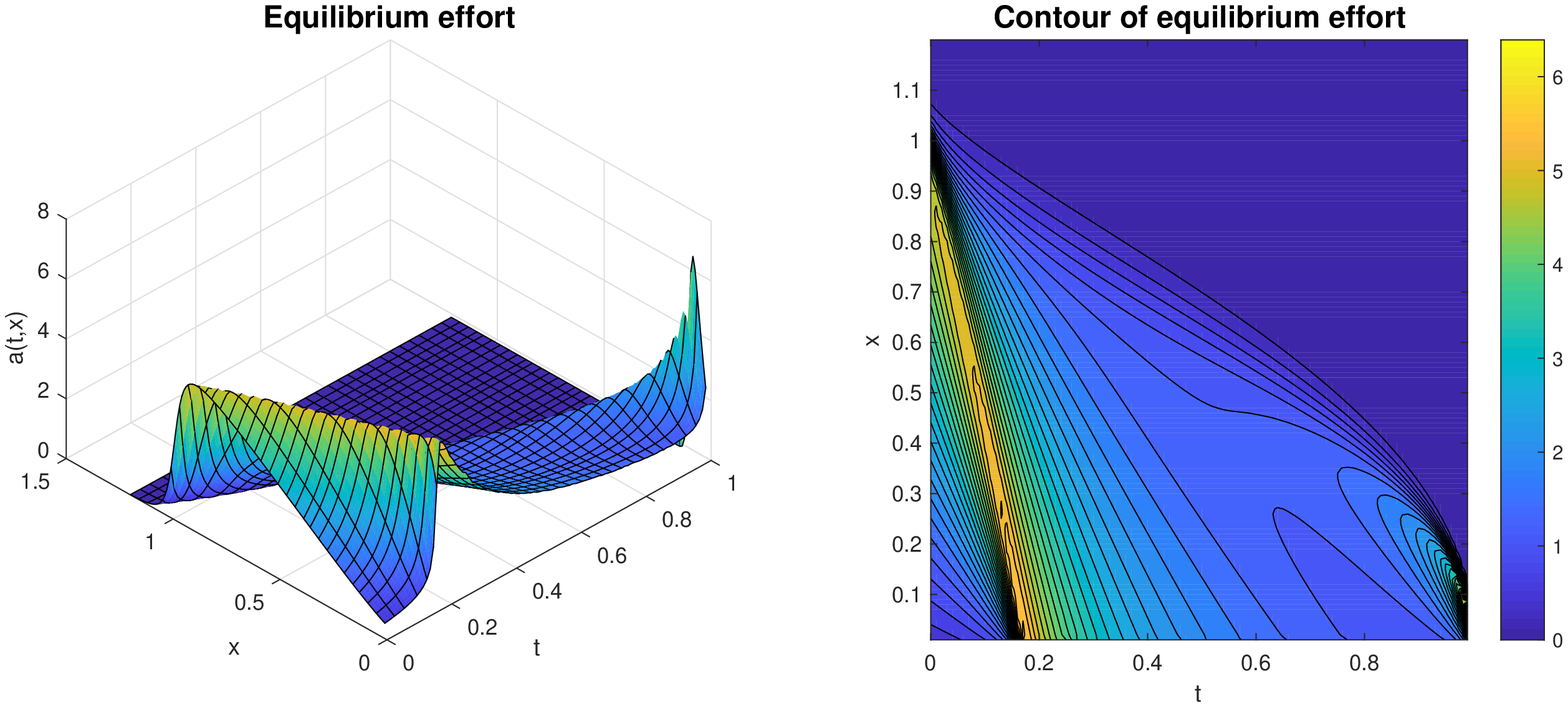}
\caption{Equilibrium completion time distribution, value function and effort function with $H(r)=6(1-r)^2$.}
\label{fig:hom1}
\end{figure}

\begin{figure}[t]
\centering
\includegraphics[height=5.5cm]{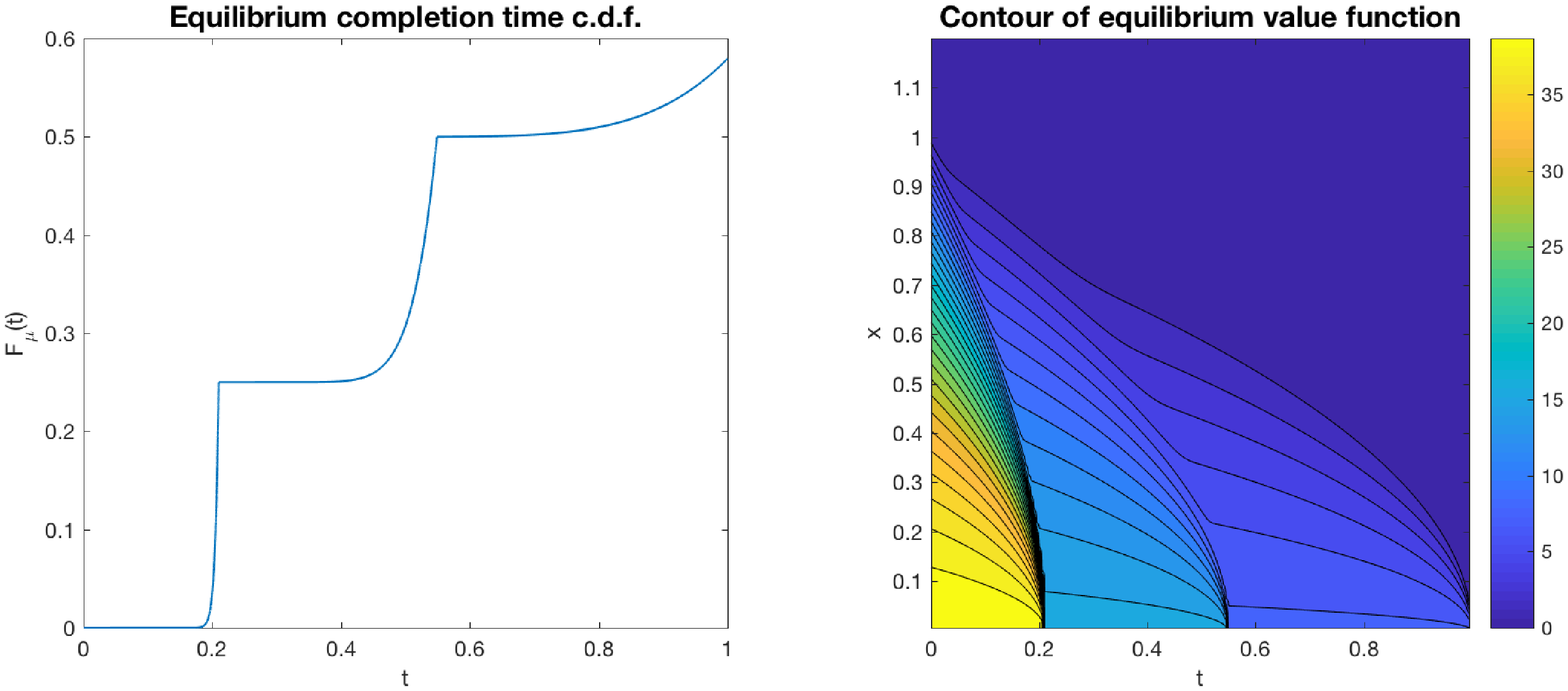}
\includegraphics[height=5.5cm]{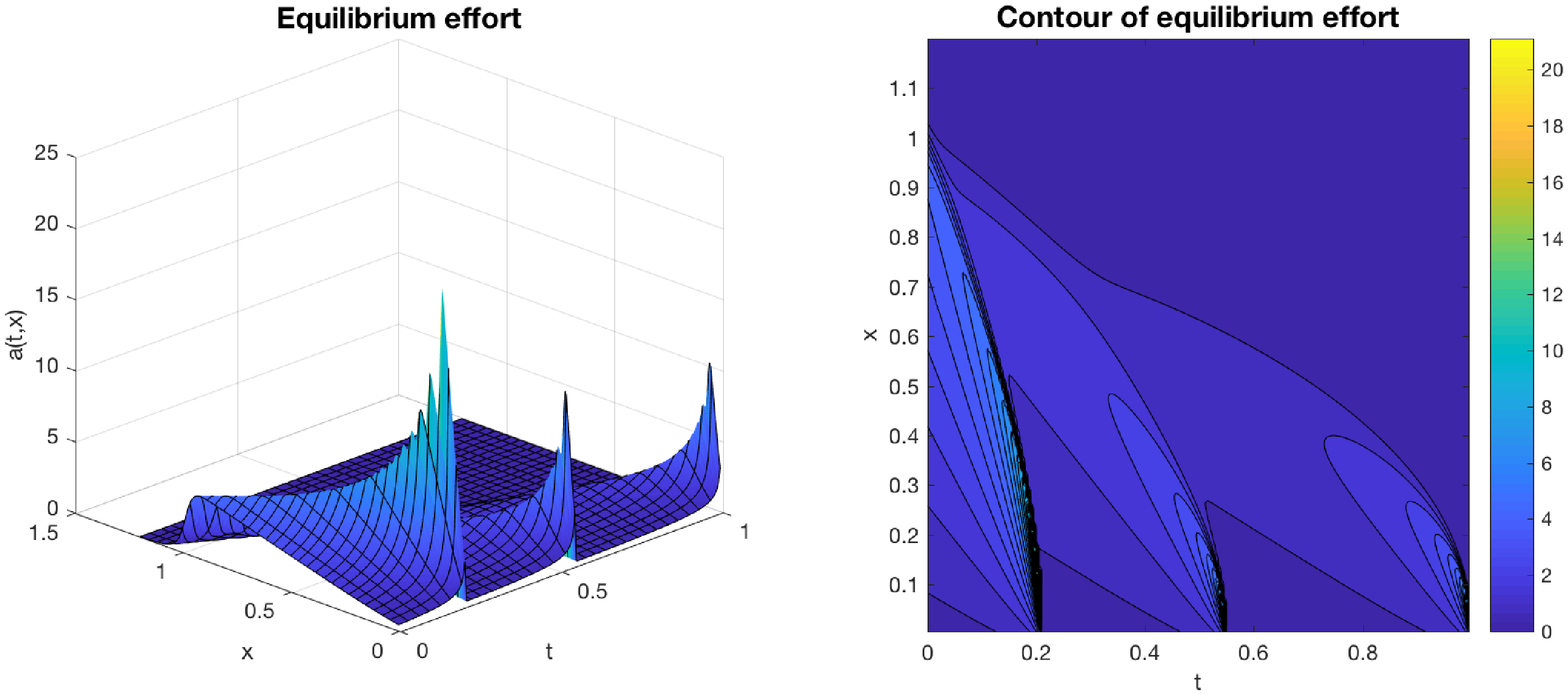}
\caption{Equilibrium completion time distribution, value function and effort function with $H(r)=5\cdot1_{[0,0.25)}(r)+2\cdot1_{[0.25,0.5)}(r)+1_{[0.5,1]}(r)$.}
\label{fig:hom2}
\end{figure}

In the remaining part of this section, we fix $H(r)=6(1-r)^2$ except in Section~\ref{depH} where we analyze the dependence on $H$.

\subsubsection{Dependence on the tournament horizon}

Table~\ref{tab:varyT} shows that as  the deadline increases, both the equilibrium terminal completion rate  and the game value increase (as proved in Proposition~\ref{prop:homcs}).
However,   with a longer deadline
  it takes longer to reach a fixed percentage of completion.
On the other hand, we also see that the quantile functions are not too sensitive to changes in the deadline, except near the discontinuity $r=F_\mu(T)$. This suggests that competition alone is often good enough to drive the progress; the external deadline provides a little extra, but not significant motivation.
\begin{table}[h]
\centering
\scalebox{0.9}{
\begin{tabular}{|c|c|c|c|c|c|}\hline
$T$
&\makebox[5em]{1st quartile}&\makebox[5em]{median}&\makebox[5em]{3rd quartile}
&\makebox[5em]{$F_\mu(T)$}&\makebox[5em]{$V$}\\\hline\hline
0.5 & 0.281 & - & - & 44.9\% & 0.074\\\hline
1 & 0.285 & 0.613 & - & 61.9\% & 0.121\\\hline
2 & 0.289 & 0.630 & - & 73.6\% & 0.166\\\hline
5 & 0.293 & 0.649 & 2.424 & 83.4\% & 0.215\\\hline
10 &0.295 & 0.658 & 2.545 &  88.1\% &  0.237 \\\hline
100 &0.296 & 0.666 & 2.661 &  96.0\% &  0.256 \\\hline
$\infty$ & 0.296  & 0.667  & 2.667  & 100\%  & 0.257 \\\hline
\end{tabular}}
\medskip
\caption{Equilibrium quartiles, completion rates and game values under varying deadline.}
\label{tab:varyT}
\end{table}

\subsubsection{Dependence on the reward function}\label{depH}

Next, we focus on the two-parameter family:
\[H(r)=K(1+p)(1-r)^p\]
where $K=\int_0^1 H(r)dr$ represents the total reward budget, and $p$ determines the convexity of the reward function. A large $p$ means that most of the reward is given to highly ranked players. In many examples, including population income in US,  the prize money decreases in a very convex manner, with high ``earners" earning a very large chunk of the pie.

We see from the top left panel of Figure~\ref{fig:homvaryp} that, with a moderate value $K=2$, as $p$ increases, the peak of $f_\mu=F'_{\mu}$ shifts to the left, meaning most of the players finish earlier. On the other hand, this is at the expense of a lower population completion rate, since the laggards, knowing that the reward drops quickly once the leaders have occupied the high ranks, put in less effort and give up more easily.
  The lower completion rate also leads to a lower game value; see the bottom left panel of Figure~\ref{fig:homvaryp}.
  Thus,  when the pie is not too small,  the higher the convexity of the reward function, the lower the welfare. That is, shifting the rewards more to the highly ranked players
decreases the welfare.
However, the monotonicity of the welfare  in $p$ is only true when there is sufficient benefit for finishing early. The top right panel of Figure~\ref{fig:homvaryp} shows that when the reward $K$ is small (similar behavior can be observed when  $x_0$ or $c$ is large), both the population terminal completion rate and the game value are no longer decreasing in $p$. A more complete picture is shown in the bottom right panel of Figure~\ref{fig:homvaryp}. We only plot the game value, since it moves in the same direction as the completion rate.

We focus now on the following implications of our results:  since the expected total effort increases with the (terminal) completion rate when we fix $x_0, \sigma, T$ (see Proposition~\ref{prop:expected-effort}), the expected effort will be lower as we increase $p$ in the top left panel of Figure~\ref{fig:homvaryp}; that is, the expected total effort decreases with the level of competition (more unequal rewards). Thus, when the total pie is large enough,  the  competitiveness resulting from inequality in rewards has a demoralizing effect, as in \cite{FNS18}.  However, the top right panel of Figure~\ref{fig:homvaryp} shows that when the total pie $K$ is small, the completion rate and hence the effort level may  go up with more unequal payment. This is because there is   another effect that decreases
competitiveness -- a higher percentage of players gives up, and if the rewards are more uneven it is still worthwhile for the players who do not give up to exert higher effort,  bringing up the aggregate completion rate. That is, the players who do not give up are less discouraged by the prize inequality, because they compete within a smaller group.\footnote{Note that in \cite{FNS18} the size of the pie is directly linked to the aggregate effort, and not fixed exogenously as in the present example.}
If, for example, we consider the competition of building internet search engines, or social media sites, or computer operating systems
as a tournament race, assuming that our infinite players game with a fixed pie is still a decent approximation, since the total pie is large and the rewards are uneven, our results suggest that there is loss of efficiency in that many players get discouraged from applying effort.
They also suggest that  government subsidies for  R\&D (e.g., for renewable energy)  that make the rewards more even among different players are efficiency improving,
 but only in areas in which  the profits to be made are likely to be substantial.

\begin{figure}[t]
\centering
\includegraphics[height=5.5cm]{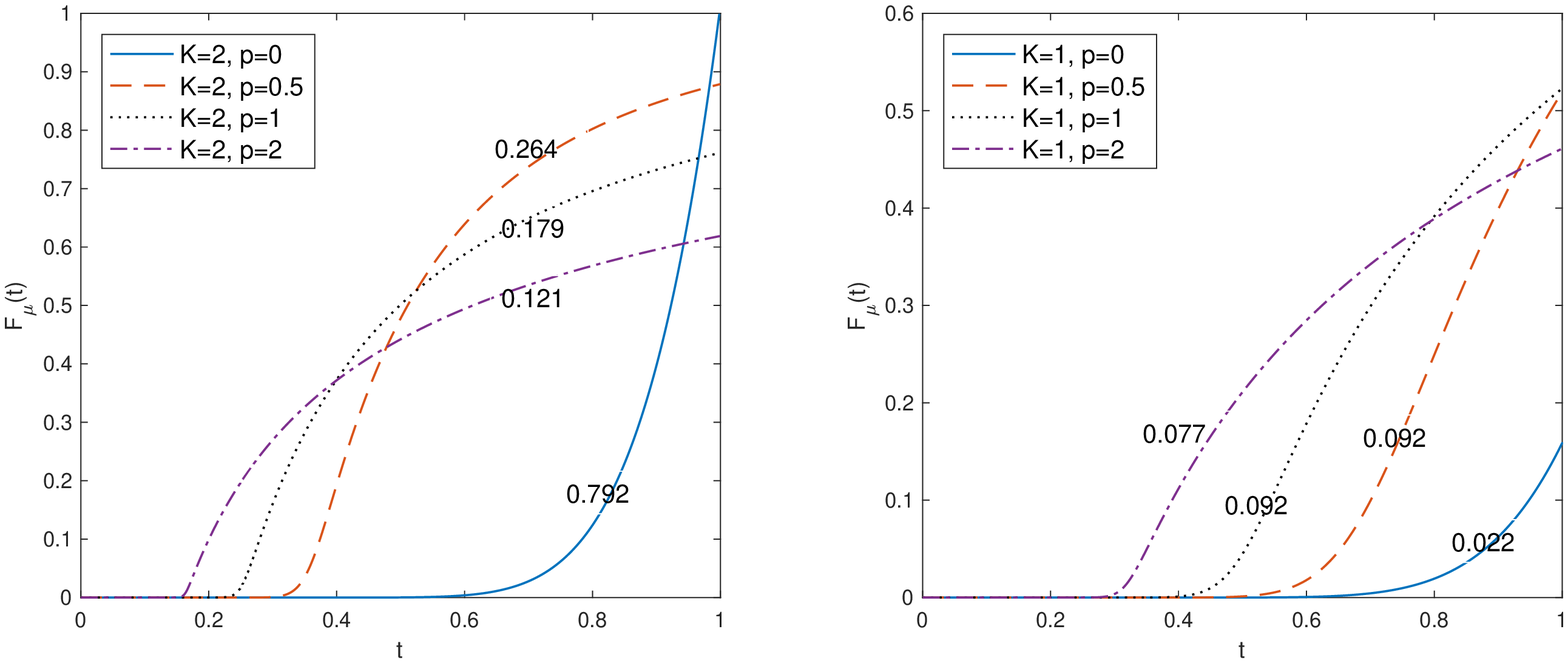}
\includegraphics[height=5.5cm]{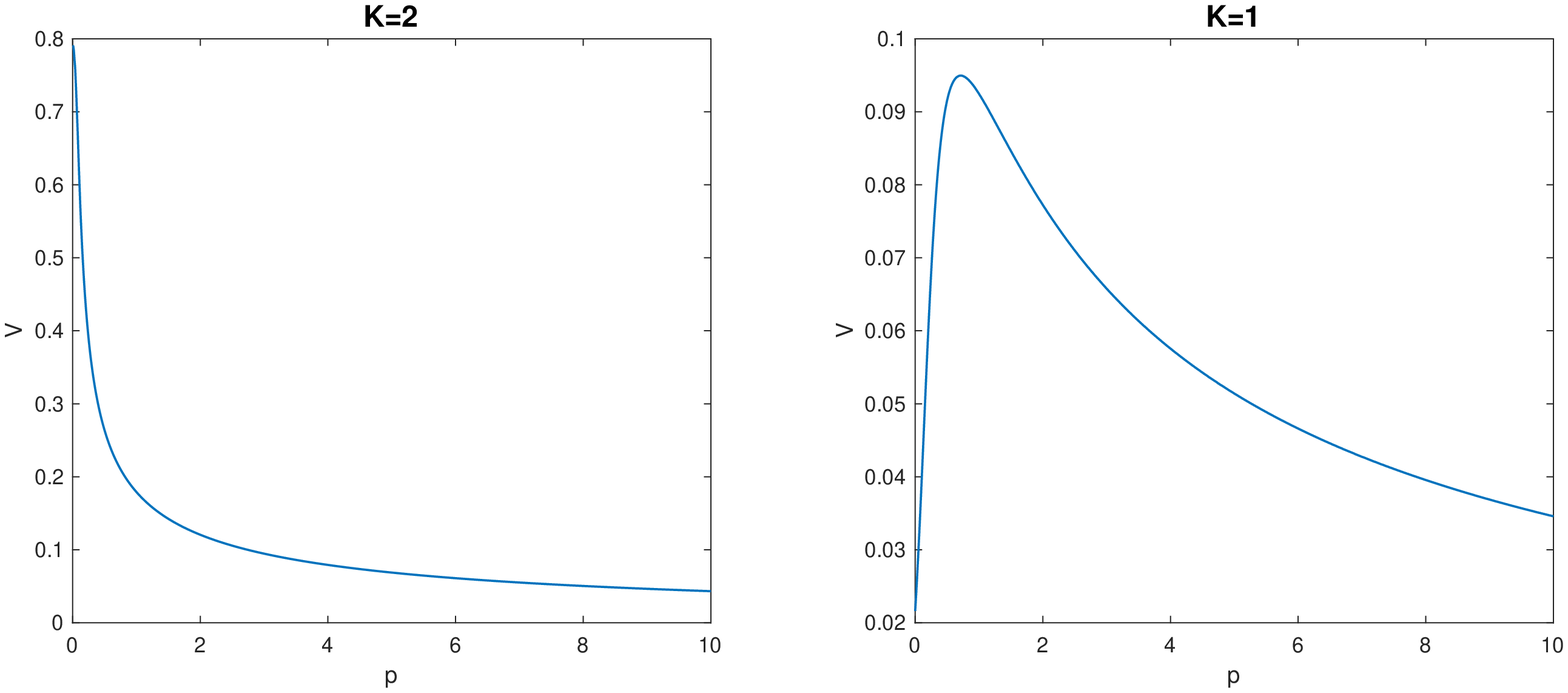}
\caption{Equilibrium completion time distribution and game value under different total reward $K$ and convexity index $p$. The numbers on the curves in the top panels represent their corresponding game values.}
\label{fig:homvaryp}
\end{figure}


\subsubsection{Dependence on the cost parameter}

Figure~\ref{fig:homvaryc} illustrates the dependence on the cost parameter $c$. 
As $c$ increases, the welfare value  first experiences an increase before it starts to decrease}.  The intuition is the same as in the case of $T=\infty$ (in which case, as we proved  in Proposition~\ref{prop:homcs}(iii), the value is always increasing in cost): the aggregate welfare does not depend on the relative ranks of the players, while, with low cost, the players compete
for those ranks ``too hard" against each other, raising the realized cost of effort and  bringing down the welfare value.
 Thus, higher cost $c$ can be beneficial,
by discouraging the players to work too hard. However,
 as $c$ continues to increase, the possibility of failing to finish by  the deadline starts to offset the benefit of a reduced effort, and the value starts to decrease.\footnote{
  We note that with Poisson type uncertainty (sudden breakthroughs), in the infinite horizon model of  \cite{MZ17},  the game value is independent of $c$.} 

\begin{figure}[t]
\centering
\includegraphics[height=5.5cm]{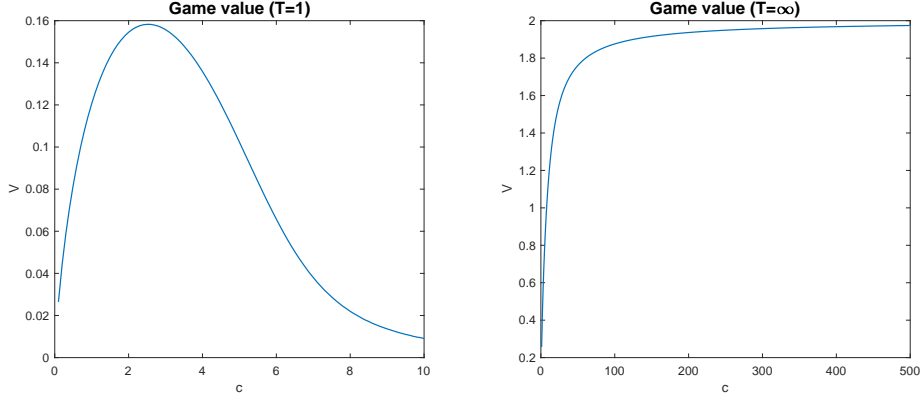}
\caption{Game value against the cost parameter. }
\label{fig:homvaryc}
\end{figure}

\subsection{Reverse engineering: realizing a target distribution}\label{subsec:attainable_mu}

Theorem~\ref{thm:explicit_soln} allows us to obtain, for a given reward function, the corresponding equilibrium distribution. The opposite problem is also important: if we want to achieve a certain equilibrium distribution, how do we go about it? The following two theorems, for $T=\infty$ and $T<\infty$ respectively, do the following:

\begin{itemize}
\item[(i)] identify which  distributions $\mu$ are realized in equilibrium and by which reward function;

\item[(ii)] identify conditions under which a distribution $\mu$ can be realized by an expected budget not higher than a given budget $K$.
\end{itemize}
These results are also helpful in allowing us to convert an optimization over reward functions to an equivalent one over the set of feasible equilibrium distributions. The latter problem is easier in some cases; see Section~\ref{sec:profit} for an example.

Let $\mathcal{H}_0$ be the set of bounded, decreasing reward functions from $[0,1]$ to $[R_\infty,\infty)$, and
\[\mathcal{E}:\mathcal{H}_0\rightarrow \mathcal{P}(\bbR_+), \quad \hbox{for } T=\infty,\]
\[\mathcal{E}_T:\mathcal{H}_0\rightarrow \mathcal{P}(\bbT_0), \quad \hbox{for } T<\infty,\]
be the mappings from $H$ to the equilibrium completion time distribution $\mu$, where $\bbT_0=[0,T]\cup\{\infty\}$. Observe that $\mathcal{E}$ is translation invariant, i.e., \ $\mathcal{E}(H+C)=\mathcal{E}(H)$ for any $C\in\bbR$ such that $H+C\ge R_\infty$.
We first show that $\mathcal{E}$ and $\mathcal{E}_T$ are one-to-one mappings up to a.e.\ equivalence, and in the infinite horizon case, also up to an additive constant.

\begin{lemma}\label{lemma:1to1} Let $H_1, H_2\in\mathcal{H}_0$.
\begin{itemize}
\item[(i)] Suppose $\mathcal{E}(H_1)=\mathcal{E}(H_2)=\mu$, then $H_1=H_2+C$ a.e.\ for some constant $C$.
\item[(ii)] Suppose $\mathcal{E}_T(H_1)=\mathcal{E}_T(H_2)=\mu$, then $H_1=H_2$ a.e.\ on $[0, F_\mu(T)]$.
\end{itemize}
\end{lemma}
 \begin{proof}
We only prove (ii); case (i) is similar. By \eqref{eq:NEquantile},
\[F_{\tau^\circ_{x_0/\sigma}}(t)=\frac{1-F_{\tau^\circ_{x_0/\sigma}}(T)}{1-F_\mu(T)}\int_0^{F_\mu(t)} \exp\left(\frac{R_\infty-H_i(z)}{2c\sigma^2}\right)dz, \quad t\in[0,T], \ i=1, 2.\]
Differentiating both sides and rearranging terms yields
\[H_i(F_\mu(t))=2c\sigma^2 \ln\left(\frac{f_\mu(t)}{f_{\tau^\circ_{x_0/\sigma}}(t)}\right)+2c\sigma^2 \ln\left(\frac{1-F_{\tau^\circ_{x_0/\sigma}}(T)}{1-F_\mu(T)}\right)+R_\infty\quad a.e.\]
It follows that
\[\int_0^{F_\mu(T)} |H_1(r)-H_2(r)|dr=\int_0^T \left|H_1(F_\mu(t))-H_2(F_\mu(t))\right|f_\mu(t) dt=0,\]
and consequently, $H_1=H_2$ a.e.\ on $[0, F_\mu(T)]$.
\end{proof}

Denote by  $\mathcal{P}^+(\bbR_+)$ (reps.\ $\mathcal{P}^+(\bbT_0)$) the set of probability distributions on $\bbR_+$ (resp.\ $\bbT_0$) that have strictly positive density on $\bbR_+$ (resp.\ $[0,T]$). 
The key quantity is the normalized density
 $$\zeta_\mu:=f_\mu/f_{\tau^\circ_{x_0/\sigma}},$$
 and if $T<\infty$, also the normalized incompletion rate
 \[\delta_\mu:=\frac{1-F_\mu(T)}{1-F_{\tau^\circ_{x_0/\sigma}}(T)}.\]

\begin{thm}\label{thm:Erange}
Fix $K\ge R_\infty$. For $\mu\in \mathcal{P}^+(\bbR_+)$, define
\[H_\mu(r):=
2c\sigma^2 \ln \zeta_\mu(F^{-1}_\mu(r)).\]
We have
\begin{itemize}
\item[(i)] $\mathcal{E}(\mathcal{H}_0)=\{\mu\in \mathcal{P}^+(\bbR_+): \ \ln \zeta_\mu \text{ is bounded and decreasing}\}$ and
\[\mathcal{E}^{-1}(\mu)=\{H_\mu+C: C\ge R_\infty-2c\sigma^2\ln \left(\inf \zeta_\mu \right)\},\footnote{Here and in the sequel, we identify a reward function $H$ with its equivalence class under a.e.\ relation.}\quad \mu\in \mathcal{E}(\mathcal{H}_0).\]
\item[(ii)] For $\mu\in \mathcal{E}(\mathcal{H}_0)$,
$\int_0^1 (H_\mu(r)+C) dr\le K$ if and only if $C\le K-2c\sigma^2 \int_0^\infty \ln\zeta_\mu(t)d\mu(t)$.
\end{itemize}
\end{thm}
\begin{proof}
(i) By Proposition~\ref{prop:taupdf}, $\mathcal{E}(\mathcal{H}_0)\subseteq \{\mu\in \mathcal{P}^+(\bbR_+): \ln \zeta_\mu \text{ is bounded and decreasing}\}$. Conversely, given any $\mu\in \mathcal{P}^+(\bbR_+)$ such that $\ln \zeta_\mu$ is bounded and decreasing, $H_\mu$ is also bounded and decreasing. Since $F^{-1}_\mu:[0,1]\mapsto \bbR_+$ is bijective, $H_\mu+C\ge R_\infty$ if and only $C\ge R_\infty-2c\sigma^2 \ln(\inf \zeta_\mu)$. It is straightforward to check that $\mu$ satisfies \eqref{eq:infNEquantile} with $H=H_\mu$:
\begin{align*}
\frac{\int_0^{r}  \exp\left(-\frac{H_\mu(z)}{2c\sigma^2}\right) dz}{\int_0^1\exp\left(-\frac{H_\mu(z)}{2c\sigma^2}\right) dz}&=\frac{\int_0^{r}  \frac{f_{\tau^\circ_{x_0/\sigma}}(F_\mu^{-1}(z))}{f_\mu(F_\mu^{-1}(z))} dz}{\int_0^1 \frac{f_{\tau^\circ_{x_0/\sigma}}(F_\mu^{-1}(z))}{f_\mu(F_\mu^{-1}(z))}dz}=\frac{\int_0^{F^{-1}_\mu(r)} \frac{f_{\tau^\circ_{x_0/\sigma}}(t)}{f_\mu(t)} dF_\mu(t)}{\int_0^\infty \frac{f_{\tau^\circ_{x_0/\sigma}}(t)}{f_\mu(t)} dF_\mu(t)}=F_{\tau^\circ_{x_0/\sigma}}\left(F_\mu^{-1}(r)\right).
\end{align*}
Hence, $\mathcal{E}(H_\mu+C)=\mathcal{E}(H_\mu)=\mu$ by the translation invariance of $\mathcal{E}$. By Lemma~\ref{lemma:1to1}(i), any admissible reward scheme realizing $\mu$ differs from $H_\mu$ by a constant. (ii) follows from straightforward calculation.
\end{proof}

\begin{thm}\label{thm:ETrange}
Fix $T<\infty$ and $K\ge R_\infty$ . For $\mu\in \mathcal{P}^+(\bbT_0)$, define
\[\widetilde H_\mu(r):=R_\infty+2c\sigma^2 1_{[0,F_\mu(T)]}(r)\left\{\ln \zeta_\mu(F^{-1}_\mu(r))-\ln\delta_\mu\right\}.\]
We have
\begin{itemize}
\item[(i)]
$\mathcal{E}_T(\mathcal{H}_0)=\left\{\mu\in \mathcal{P}^+(\bbT_0):  \ \ln \zeta_\mu \text{ is bounded and decreasing}, \displaystyle{\inf_{t\in[0,T]}}\zeta_\mu(t)\ge\delta_\mu\right\}$,
and
\[\mathcal{E}_T^{-1}(\mu)=\{H\in\mathcal{H}_0: H|_{[0,F_\mu(T)]}=\widetilde H_\mu|_{[0,F_\mu(T)]}\},\quad \mu\in \mathcal{E}_T(\mathcal{H}_0).\]
\item[(ii)]
For $\mu\in\mathcal{E}_T(\mathcal{H}_0)$, $\int_0^1 \widetilde H_\mu(r)dr=\int_0^{F_\mu(T)} \widetilde H_\mu(r)dr+(1-F_\mu(T))R_\infty\le K$ if and only if
\[\int_0^T\ln \zeta_\mu(t)d\mu(t)\le \frac{K-R_\infty}{2c\sigma^2}+F_\mu(T) \ln\delta_\mu.\]
\end{itemize}
\end{thm}
\begin{proof}
(i) Let $\mu\in \mathcal{E}_T(\mathcal{H}_0)$. By Proposition~\ref{prop:taupdf}, we have that $\mu\in \mathcal{P}^+(\bbT_0)$ and $\ln \zeta_\mu$ is bounded and decreasing. Moreover, since $H\ge R_\infty$, we have that for $t\in[0,T]$,
\[\zeta_\mu(t)=\frac{f_\mu(t)}{f_{\tau^\circ_{x_0/\sigma}}(t)}=\frac{u(t,0;\mu,c)}{u(0,x_0;\mu,c)}
=\frac{1-F_\mu(T)}{1-F_{\tau^\circ_{x_0/\sigma}}(T)}\exp\left(\frac{H(F_\mu(t))-R_\infty}{2c\sigma^2}\right)\ge \delta_\mu.\]
Conversely, given any $\mu\in \mathcal{P}^+(\bbT)$ such that $\ln \zeta_\mu$ is bounded and decreasing, and
$\inf_{t\in[0,T]}\zeta_\mu(t)\ge \delta_\mu$, we have $\widetilde H_\mu\in\mathcal{H}_0$. It is straightforward to check that $\mu$ satisfies \eqref{eq:NEquantile} with $H=\widetilde H_\mu$: for $r\in[0,F_\mu(T)]$,
\begin{align*}
&\frac{1-F_{\tau^\circ_{x_0/\sigma}}(T)}{1-F_\mu(T)}\int_0^r \exp\left(\frac{R_\infty-\widetilde H_\mu(z)}{2c\sigma^2}\right)dz=\int_0^r \frac{f_{\tau^\circ_{x_0/\sigma}}(F^{-1}_\mu(z))}{f_\mu(F^{-1}_\mu(z))}dz\\
&=\int_0^{F^{-1}_\mu(r)}\frac{f_{\tau^\circ_{x_0/\sigma}}(t)}{f_\mu(t)}dF_\mu(t)=F_{\tau^\circ_{x_0/\sigma}}(F^{-1}_\mu(r)).
\end{align*}
Hence, $\mathcal{E}_T(\widetilde H_\mu)=\mu$. By Lemma~\ref{lemma:1to1}(ii), any admissible reward scheme realizing $\mu$ must agree with $\widetilde H_\mu$ on $[0, F_\mu(T)]$. Conversely, since the reward after rank $F_\mu(T)$ is irrelevant in determining the individual's best response, we have $\mathcal{E}_T(H)=\mathcal{E}_T(\widetilde H_\mu)=\mu$ for any $H\in\mathcal{H}_0$ which agrees with $\widetilde H_\mu$ on $[0, F_\mu(T)]$. (ii) follows from straightforward calculation.
\end{proof}

\subsection{Optimal reward design}

The semi-explicit characterization of the equilibrium allows us to further study the optimal reward design problem for a principal or social planner. 
We continue to consider only the reward functions of the form  \eqref{eq:PRBR} with $H\in\mathcal{H}_0$, in which case
there exists a unique equilibrium with completion time distribution denoted by $\mathcal{E}(H)$ or $\mathcal{E}_T(H)$. 

We consider three different optimization criteria: minimizing time to achieve a given population completion rate (Section~\ref{sec:quantile}), maximizing welfare (Section~\ref{sec:welfare}) and maximizing net profit (Section~\ref{sec:profit}).
To preview the results: for the first two criteria, the optimal reward is a two-step function -- the same reward for all sufficiently high ranks, and the same minimum guaranteed payment $R_\infty$ for the low ranks. This is different from the one-stage Poisson game of \cite{MZ17} where the quantile-minimizing reward scheme is concave for high ranks.
For the third problem where for a given profit function $g$, we maximize the expected profit $E g(\tau)$ minus the cost of reward, with $\tau$ drawn from the infinite horizon equilibrium distribution, the optimal reward for finishing at time $t$ is a linear transformation of $g(t)$ for $t$ lower than a  bonus deadline $t^*_b$, and $R_\infty$ otherwise.

{Here we only highlight the proof of the third problem where we rely on the result from reverse engineering. All other proofs are provided in the appendix.}

\subsubsection{Minimizing the time to achieve a given completion rate}\label{sec:quantile}

We fix a deadline $T\in(0,\infty]$, a total reward budget $K$, a target completion rate $\alpha\in(0,1)$ and a minimum participation reward $R_\infty\le K$, and  look for  reward function $H(r)\ge R_\infty$ that minimizes the time it takes $\alpha$ fraction of the population to complete their projects in equilibrium.
More precisely, the  feasible set of reward functions is
\[\mathcal{H}:=\left\{H\in\mathcal{H}_0: \int_0^1 H(r) dr\le K\right\}, \text{ for }  T=\infty,\]
and, with $\beta(H):=F_{\mathcal{E}_T(H)}(T)$,
\[\mathcal{H}^\alpha_T:=\left\{H\in\mathcal{H}_0: \beta(H)\ge \alpha, \int_0^{\beta(H)} H(r) dr+(1-\beta(H))R_\infty\le K\right\}, \text{ for } T<\infty.\]

For $H\in\mathcal{H}$, let $T_\alpha(H):=T^{\mathcal{E}(H)}_\alpha$ be the $\alpha$-quantile of $\mathcal{E}(H)$. We wish to find
$T^*_\alpha=\inf_{H\in\mathcal{H}} T_\alpha (H),$
and identify the minimizer $H^*$, if it exists.
 Similarly, for $H\in\mathcal{H}_T$, let $T_\alpha(H;T):=T^{\mathcal{E}_T(H)}_\alpha$ be the $\alpha$-quantile of $\mathcal{E}_T(H)$.
We will look for the optimizer of $T^*_\alpha(T)=\inf_{H\in\mathcal{H}^\alpha_T}T_\alpha (H;T).$

\begin{remark}
For the finite horizon problem, there are two different ways to impose the budget constraint: (i) to require $\int_0^1 H(r)dr\le K$, as in the definition of $\mathcal{H}$; in this case, the budget may not be fully utilized due to a portion of the players failing to complete by time $T$,
but the advantage of such a constraint is that the total reward does not go over $K$ even out of equilibrium; (ii) to bound  the total  reward only in equilibrium: $\int_0^{\beta(H)} H(r) dr+(1-\beta(H))R_\infty$, as in the definition of $\mathcal{H}^\alpha_T$; such a constraint is weaker, but it might be violated if the population does not end up in  equilibrium.
It turns out, as we will see in Theorem~\ref{thm:optRT},  that the two constraints result in the same optimal value and optimal reward function.
\end{remark}

The following two theorems present the optimal reward functions and the corresponding equilibria.
\begin{thm}\label{thm:optR}
Let $T=\infty.$ Then,
$\inf_{H\in\mathcal{H}}T_\alpha(H)$ is uniquely attained (up to a.e.\ equivalence) by the uniform scheme with cutoff rank $\alpha$:
\[H^*(r)=R_\infty+\frac{K-R_\infty}{\alpha}1_{[0,\alpha]}(r).\]
The minimal time is
\[T^*_\alpha=T_\alpha(H^*)=F^{-1}_{\tau^\circ_{x_0/\sigma}}\left(\frac{\alpha}{\alpha+(1-\alpha)\exp\left(\frac{K-R_\infty}{2\alpha c\sigma^2}\right)}\right).\]
The c.d.f.\ of $\mu=\mathcal{E}(H^*)$ is given by
\[F_\mu(t)=
\begin{cases}
\left(\alpha+(1-\alpha)\exp\left(\frac{K-R_\infty}{2\alpha c\sigma^2}\right)\right)F_{\tau^\circ_{x_0/\sigma}}(t), & \text{if } t\le T^*_\alpha,\\
F_{\tau^\circ_{x_0/\sigma}}(t)+\alpha\left(1-F_{\tau^\circ_{x_0/\sigma}}(t)\right)\left(1-\exp\left(\frac{R_\infty-K}{2\alpha c\sigma^2}\right)\right), & \text{if } t> T^*_\alpha.
\end{cases}
\]
The equilibrium value attained by each player is
\[V=R_\infty-2c\sigma^2 \ln\left(\alpha \exp\left(\frac{R_\infty-K}{2\alpha c\sigma^2}\right)+1-\alpha\right),\]
and the equilibrium effort, in feedback form, is
\[a(t,x)=\frac{2\left[\exp\left(\frac{K-R_\infty}{2\alpha c\sigma^2}\right)-1\right]N' \left(\frac{x}{\sigma\sqrt{T^*_\alpha-t}}\right) \frac{\sigma}{\sqrt{T^*_\alpha-t}}}{1+2\left[\exp\left(\frac{K-R_\infty}{2\alpha c\sigma^2}\right)-1\right]\left[1-N\left(\frac{x}{\sigma\sqrt{T^*_\alpha-t}}\right)\right]},\]
where $N$ and $N'$ are the c.d.f.\ and p.d.f.\ of the standard normal distribution.
\end{thm}

\begin{thm}\label{thm:optRT}
Let the deadline $T<\infty$, the minimum participation reward $R_\infty$, the total reward budget $K\ge R_\infty$ and the target completion rate $\alpha\in(0,1)$ be given. We have $T^{*}_\alpha(T)=\inf_{H\in\mathcal{H}^\alpha_T}T_\alpha(H;T)=\inf_{H\in\mathcal{H}}T_\alpha(H;T)$.
Let $T^{*}_\alpha$ be the optimal time given by Theorem~\ref{thm:optR}.
\begin{itemize}
\item If $T<T^{*}_\alpha$, then $\mathcal{H}^\alpha_T=\emptyset$ and $T^{*}_\alpha(T)=\Delta$.
\item If $T\ge T^{*}_\alpha$, then
$T^{*}_\alpha(T)=T^{*}_\alpha$ is uniquely attained (up to a.e.\ equivalence) by the uniform scheme with cutoff rank $\alpha$:
\[H^*(r)=R_\infty+\frac{K-R_\infty}{\alpha}1_{[0,\alpha]}(r).\]
The c.d.f.\ of $\mu=\mathcal{E}_T(H^*)$ on $[0,T]$, the equilibrium value $V$ attained by each player and the equilibrium effort $a(t,x)$, $t\in[0,T]$ in feedback form have the same expression as those of Theorem~\ref{thm:optR}.
\item The minimum budget needed to ensure that $\alpha$ fraction of players finish by time $T$ is
\[K_{\min}=R_\infty+2\alpha c\sigma^2 \ln\left(\frac{\alpha}{1-\alpha}\cdot\frac{1-F_{\tau^\circ_{x_0/\sigma}}(T)}{F_{\tau^\circ_{x_0/\sigma}}(T)}\right).\]
The reward scheme which achieves the goal with budget $K_{\min}$ is given by
\[H(r)=R_\infty+\frac{K_{\min}-R_\infty}{\alpha}1_{[0,\alpha]}(r).\]

\item Under the given budget $K$, the maximum equilibrium completion rate $\alpha_{\max}$ attainable at time $T$ is the unique solution of
\[C_T:=\frac{F_{\tau^\circ_{x_0/\sigma}}(T)}{1-F_{\tau^\circ_{x_0/\sigma}}(T)}=\frac{\alpha_{\max}}{1-\alpha_{\max}}\exp\left(\frac{R_\infty-K}{2\alpha_{\max} c\sigma^2}\right). \footnote{$\alpha_{\max}$ can also be expressed as $\left(1+\frac{2c\sigma^2}{K-R_\infty}W\left(C^{-1}_T\right)\right)^{-1}$, where $W$ is the Lambert-W function.}
\]
The reward scheme which yields $\alpha_{\max}$ is given by
\[H(r)=R_\infty+\frac{K-R_\infty}{\alpha_{\max}}1_{[0,\alpha_{\max}]}(r).\]
\end{itemize}
\end{thm}


\subsubsection{Maximizing welfare}\label{sec:welfare}

We now find the reward scheme that maximizes the aggregate game value of all the players, again in the homogeneous case. 


\begin{thm}\label{thm:welfare}
Fix the participation reward $R_\infty$ and the reward budget $K\ge R_\infty$. Then,
\begin{itemize}
\item[(i)] When $T=\infty$, the maximum welfare $\sup_{H\in\mathcal{H}}V_\infty(H) =K$ is uniquely attained (up to a.e.\ equivalence) by the uniform scheme $H^*(r)\equiv K$
(thus, with zero effort by everyone).
\item[(ii)] When $T<\infty$, the maximum welfare
\begin{align*}
\sup_{H\in\mathcal{H}^0_T}V(H) &=R_\infty+2c\sigma^2 \ln \left(\frac{1-F_{\tau^\circ_{x_0/\sigma}}(T)}{1-\alpha}\right)
\end{align*}
is uniquely attained (up to a.e.\ equivalence) by the uniform scheme with cutoff rank $\alpha$:
\[H^*(r)=R_\infty+\frac{K-R_\infty}{\alpha}1_{[0,\alpha]}(r),\]
where $\alpha$ is the maximum attainable completion rate by time $T$, given in Theorem \ref{thm:optRT}. Moreover, $H^*$ also maximizes the expected total effort given by Proposition~\ref{prop:expected-effort}.
\end{itemize}
\end{thm}

\subsubsection{Maximizing net profit}\label{sec:profit}

We now suppose that each project completed at time $t$ generates a profit of $g(t)=g(t;x_0)$ for a principal.
 We assume that $g$ is continuous, bounded and decreasing, and that $g\not\equiv g(\infty)$. The principal wants to maximize the expected  net profit $E\left[g(\tau)-R_\mu(\tau)\right]$, $ \tau\sim \mu,$
subject to the participation constraint $R\ge R_\infty$. We only consider the case $T=\infty$.

\begin{thm}\label{thm:net}
Suppose $T=\infty$. A reward scheme $H^*\in\mathcal{H}_0$ is optimal if and only if
\[H^*(r)=R_\infty+g(F^{-1}_{\mu^*}(r)\wedge t_b^*)-g(t_b^*),\]
where the ``bonus" deadline $t_b^*$ is given by $t^*_b=\inf\{z\ge 0: g(z)=g(z^*)\}$
for some
\[z^*=\argmax_{z\in[0,\infty)}  \frac{\int_0^\infty g(t\vee z) f_{\tau^\circ_{x_0/\sigma}}(t) \exp\left(\frac{g(t\wedge z)}{2c\sigma^2}\right)dt}{\int_0^\infty f_{\tau^\circ_{x_0/\sigma}}(s) \exp\left(\frac{g(s\wedge z)}{2c\sigma^2}\right)ds},\]
and where the associated equilibrium distribution $\mu^*$ has p.d.f.
\[f_{\mu^*}(t)=\frac{f_{\tau^\circ_{x_0/\sigma}}(t) \exp\left(\frac{g(t\wedge t^*_b)}{2c\sigma^2}\right)}{\int_0^\infty f_{\tau^\circ_{x_0/\sigma}}(s) \exp\left(\frac{g(s\wedge t^*_b)}{2c\sigma^2}\right) ds}.\]
The maximum net profit is
\[U=  \frac{\int_0^\infty g(t\vee t^*_b) f_{\tau^\circ_{x_0/\sigma}}(t) \exp\left(\frac{g(t\wedge t^*_b)}{2c\sigma^2}\right) dt}{\int_0^\infty f_{\tau^\circ_{x_0/\sigma}}(s) \exp\left(\frac{g(s\wedge t^*_b)}{2c\sigma^2}\right)ds}-R_\infty.\]
\end{thm}

Note that in equilibrium the players receive $H^*(F_{\mu^*}(t))=R_\infty+g(t\wedge t^*_b)-g(t^*_b)$ for finishing at time $t$. That is, the optimal reward for finishing at time $t$ is a linear transformation of $g(t)$ for those who finish before the bonus deadline  $t^*_b$; otherwise it is  the minimum guaranteed payment. Thus, it is optimal for the principal to align the players' interest  with her own.

\noindent\textbf{Proof of Theorem \ref{thm:net}.}
By Theorem~\ref{thm:Erange}, we only need to optimize over $\mu\in\mathcal{E}(\mathcal{H}_0)$, which can be realized by the reward $H_\mu+C$ for any constant $C\ge R_\infty -2c\sigma^2 \ln(\inf f_{\mu}/f_{\tau^\circ_{x_0/\sigma}})$. It is clear that among the translations of $H_\mu$, the principal should choose the smallest $C$ which meets the reservation reward constraint, namely, $C_\mu:=R_\infty -2c\sigma^2 \ln(\inf f_{\mu}/f_{\tau^\circ_{x_0/\sigma}})$. Hence we can rewrite the optimization problem as
\[U=\sup_{f_\mu} \int_0^\infty \left[g(t)-2c\sigma^2 \ln\left(\frac{f_\mu(t)}{f_{\tau^\circ_{x_0/\sigma}}(t)}\right)-R_\infty+2c\sigma^2  \ln \left(\inf\frac{f_\mu}{f_{\tau^\circ_{x_0/\sigma}}}\right) \right]f_\mu(t)dt.\]
The proof consists of  two steps.
\begin{description}
\item[Step 1] Fix $b=\inf f_\mu/f_{\tau^\circ_{x_0/\sigma}}\in(0,1]$ and solve the constraint optimization problem:
\begin{align*}
U(b):
&=\sup_{f_\mu\ge bf_{\tau^\circ_{x_0/\sigma}}} \int_0^\infty \left[g(t)-2c\sigma^2 \ln\left(\frac{f_\mu(t)}{b f_{\tau^\circ_{x_0/\sigma}}(t)}\right)\right]f_\mu(t)dt-R_\infty
\end{align*}
subject to the additional integral constraint
\[\int_0^\infty f_\mu(t)dt=1.\]
In addition, we also need to $f_\mu/f_{\tau^\circ_{x_0/\sigma}}$ to be bounded and decreasing in order to obtain a bounded, decreasing reward function $H_\mu$. This can be verified after we find the optimizer.
\item[Step 2] Maximize $U(b)$ over $b\in(0,1]$.
\end{description}

We introduce a Lagrange multiplier $\lambda$ to handle the integral constraint. Define
\begin{align*}
L(f_\mu,\lambda):&=\int_0^\infty \left[g(t)-2c\sigma^2 \ln\left(\frac{f_\mu(t)}{b f_{\tau^\circ_{x_0/\sigma}}(t)}\right)\right]f_\mu(t)dt -\lambda\left(\int_0^\infty f_\mu(t)dt-1\right).
\end{align*}
For each fixed $\lambda$, the integrand, being a concave function of $f_\mu$, attains a pointwise maximum at
\begin{equation}\label{eq:fstar0}
f_{\mu(b)}(t)=bf_{\tau^\circ_{x_0/\sigma}}(t) \exp\left(\left[\frac{g(t)-\lambda}{2c\sigma^2}-1\right]^+\right)
\end{equation}
on $[bf_{\tau^\circ_{x_0/\sigma}}(t),\infty)$.
We then find $\lambda$ via
\begin{equation}\label{eq:lambda}
1=\int_0^\infty f_{\mu(b)}(s)ds=\int_0^\infty bf_{\tau^\circ_{x_0/\sigma}}(s) \exp\left(\left[\frac{g(s)-\lambda}{2c\sigma^2}-1\right]^+\right)ds=:\phi(\lambda).
\end{equation}
Note that the above equation has a unique solution in $(-\infty, g(0)-2c\sigma^2]$ since $\phi$ is continuous, strictly decreasing and satisfies $\phi(g(0)-2c\sigma^2)=b\le 1$ and $\lim_{\lambda\rightarrow -\infty}\phi(\lambda)=\infty$.
Since any $\lambda\ge g(0)-2c\sigma^2$ (which is only possible when $b=1$) leads to the same $f_{\mu(b)}$, we may assume without loss of generality that $\lambda\le g(0)-2c\sigma^2$. It is clear that for each $b\in(0,1]$, $f_{\mu(b)}/f_{\tau^\circ_{x_0/\sigma}}$ is bounded and decreasing. So the associated reward scheme $H_{\mu(b)}+C_{\mu(b)}\in\mathcal{H}_0$.

Define
\[b_0:= \left(\int_0^\infty f_{\tau^\circ_{x_0/\sigma}}(t) \exp\left(\frac{g(t)-g(\infty)}{2c\sigma^2}\right)dt\right)^{-1}\in(0,1).\]
It is not hard to see from  \eqref{eq:lambda} that $b\ge b_0$ if and only if $g(\infty)\le 2c\sigma^2+\lambda$. We consider two cases for $b$.

\begin{itemize}
\item[Case 1.] If $0<b< b_0$, or equivalently, $g(\infty)> 2c\sigma^2+\lambda$, the positive part in \eqref{eq:fstar0} and \eqref{eq:lambda} can be removed, leading to
\[f_{\mu(b)}(t)=\frac{f_{\tau^\circ_{x_0/\sigma}}(t) \exp\left(\frac{g(t)-\lambda}{2c\sigma^2}-1\right)}{\int_0^\infty f_{\tau^\circ_{x_0/\sigma}}(s) \exp\left(\frac{g(s)-\lambda}{2c\sigma^2}-1\right)ds}=b_0 f_{\tau^\circ_{x_0/\sigma}}(t)
\exp\left(\frac{g(t)-g(\infty)}{2c\sigma^2}\right).\]

In this case,
\begin{align*}
U(b)&=\int_0^\infty \left[g(t)-2c\sigma^2 \ln\left(\frac{f_{\mu(b)}(t)}{b f_{\tau^\circ_{x_0/\sigma}}(t)}\right)\right]f_{\mu(b)}(t)dt-R_\infty\\
&=\int_0^\infty \left[g(\infty)+2c\sigma^2 \ln \frac{b}{b_0}\right]b_0 f_{\tau^\circ_{x_0/\sigma}}(t)
\exp\left(\frac{g(t)-g(\infty)}{2c\sigma^2}\right)dt-R_\infty\\
&<\int_0^\infty g(\infty)b_0 f_{\tau^\circ_{x_0/\sigma}}(t)\exp\left(\frac{g(t)-g(\infty)}{2c\sigma^2}\right)dt-R_\infty\\
&=\frac{\int_0^\infty g(\infty) f_{\tau^\circ_{x_0/\sigma}}(t)\exp\left(\frac{g(t)}{2c\sigma^2}\right)dt}{\int_0^\infty f_{\tau^\circ_{x_0/\sigma}}(s) \exp\left(\frac{g(s)}{2c\sigma^2}\right)ds}-R_\infty=:U_0.
\end{align*}

\item[Case 2.] If $b_0\le b\le 1$, or equivalently, $g(\infty)\le 2c\sigma^2+\lambda$, let
\begin{equation*}
t_b:=\inf\{t\ge 0: g(t)\le 2c\sigma^2+\lambda\}\in[0,\infty]
\end{equation*}
be the point after which the constraint $f_\mu\ge bf_{\tau^\circ_{x_0/\sigma}}$ will be binding. By the continuity of $g$ and that $g(0)-\lambda\ge 2c\sigma^2$, we must have $g(t_b)=2c\sigma^2+\lambda$. It follows that
\begin{equation*}
\left[\frac{g(s)-\lambda}{2c\sigma^2}-1\right]^+= \left[\frac{g(s)-g(t_b)}{2c\sigma^2}\right]^+ =\frac{g(s\wedge t_b)-g(t_b)}{2c\sigma^2}.
\end{equation*}
Thus, we are again able to get rid of the positive part in \eqref{eq:fstar0} and \eqref{eq:lambda}, and get
\begin{equation}\label{eq:fstar}
f_{\mu(b)}(t)
=bf_{\tau^\circ_{x_0/\sigma}}(t) \exp\left(\frac{g(t\wedge t_b)-g(t_b)}{2c\sigma^2}\right),
\end{equation}
and 
\begin{equation}\label{eq:tb}
b=\left(\int_0^\infty f_{\tau^\circ_{x_0/\sigma}}(s) \exp\left(\frac{g(s\wedge t_b)-g(t_b)}{2c\sigma^2}\right)ds\right)^{-1}=:\psi(t_b).
\end{equation}
It can be shown that the $[b_0,1]$-valued function $\psi$ is decreasing on $[0,\infty]$ and strictly decreasing on any interval where $g$ is strictly decreasing. This implies $\psi(z_1)=\psi(z_2)$ if and only if $g(z_1)=g(z_2)$. So if $\psi(z)=b=\psi(t_b)$, then $g(z)=g(t_b)$, and since $t_b$ is the first time $g$ hits $2c\sigma^2+\lambda$, we must have $z\ge t_b$. In other words, $t_b$ can be characterized as the smallest solution of $\psi(z)=b$. Alternatively, if $b=\psi(z_0)$, then $t_b=\inf\{z\ge 0: g(z)=g(z_0)\}$, independent of the choice of $z_0$.

Using \eqref{eq:fstar}, we obtain
\begin{align*}
U(b)&=\int_0^\infty \left[g(t)-2c\sigma^2 \ln\left(\frac{f_{\mu(b)}(t)}{b f_{\tau^\circ_{x_0/\sigma}}(t)}\right)\right]f_{\mu(b)}(t)dt-R_\infty\\
&=\int_0^\infty \left[g(t)-g(t\wedge t_b)+g(t_b)\right] b f_{\tau^\circ_{x_0/\sigma}}(t) \exp\left(\frac{g(t\wedge t_b)-g(t_b)}{2c\sigma^2}\right) dt-R_\infty\\
&=\frac{\int_0^\infty g(t\vee t_b)  f_{\tau^\circ_{x_0/\sigma}}(t) \exp\left(\frac{g(t\wedge t_b)}{2c\sigma^2}\right)dt}{\int_0^\infty f_{\tau^\circ_{x_0/\sigma}}(s) \exp\left(\frac{g(s\wedge t_b)}{2c\sigma^2}\right)ds}-R_\infty, \quad b\in[b_0,1].
\end{align*}
Observe that setting $t_b=\infty$ in the above expression yields $U_0$.
\end{itemize}

Combining the two cases, we see that $U(b)<\widetilde U(\infty)$ if $b\in(0,b_0)$ and $U(b)=\widetilde U(t_b)$ if $b\in[b_0,1]$, where we introduce the auxiliary objective function
\[\widetilde U(z):=\frac{\int_0^\infty g(t\vee z)  f_{\tau^\circ_{x_0/\sigma}}(t) \exp\left(\frac{g(t\wedge z)}{2c\sigma^2}\right)dt}{\int_0^\infty f_{\tau^\circ_{x_0/\sigma}}(s) \exp\left(\frac{g(s\wedge z)}{2c\sigma^2}\right)ds}-R_\infty.\]
Clearly, $\sup_{b\in(0,1]}U(b)\le \sup_{z\in [0,\infty]}\widetilde U(z)$. Since $\widetilde U(z)$ is continuous and $\widetilde U(0)> g(\infty)-R_\infty=\widetilde U(\infty)$, $\sup_{z\in [0,\infty]} \widetilde U(z)$ is attained  by some $z^*<\infty$ which can be found numerically. 

\textbf{CLAIM}: Maximizing $U(b)$ is equivalent to maximizing $\widetilde U(z)$ in the following sense:
\begin{itemize}
\item[(i)] $\sup_{b\in(0,1]}U(b)= \sup_{z\in [0,\infty]}\widetilde U(z)$;
\item[(ii)] $b^*\in\argmax_{b\in(0,1]} U(b)$ if and only if $b^*=\psi(z^*)$ for some $z^*\in \argmax_{z\in [0,\infty]}\widetilde U(z)$.
\end{itemize}

The proof of the claim relies on the observation that if $b=\psi(z)$, then $U(b)=\widetilde U(t_{b})\ge \widetilde U(z)$. Indeed, $b=\psi(z)$ implies that
$z\ge t_b$ and $g(z)=g(t_b)$ as we have argued before, which further yields $g(t\wedge z)=g(t\wedge t_{b})$. Consequently,
\begin{align*}
U(b)=\widetilde U(t_{b})
&=\frac{\int_0^\infty g(t\vee t_{b})  f_{\tau^\circ_{x_0/\sigma}}(t) \exp\left(\frac{g(t\wedge t_b)}{2c\sigma^2}\right)dt}{\int_0^\infty f_{\tau^\circ_{x_0/\sigma}}(s) \exp\left(\frac{g(s\wedge t_b)}{2c\sigma^2}\right)ds}-R_\infty\\
&=\frac{\int_0^\infty g(t\vee t_{b})  f_{\tau^\circ_{x_0/\sigma}}(t) \exp\left(\frac{g(t\wedge z)}{2c\sigma^2}\right)dt}{\int_0^\infty f_{\tau^\circ_{x_0/\sigma}}(s) \exp\left(\frac{g(s\wedge z)}{2c\sigma^2}\right)ds}-R_\infty\\
&\ge \frac{\int_0^\infty g(t\vee z)  f_{\tau^\circ_{x_0/\sigma}}(t) \exp\left(\frac{g(t\wedge z)}{2c\sigma^2}\right)dt}{\int_0^\infty f_{\tau^\circ_{x_0/\sigma}}(s) \exp\left(\frac{g(s\wedge z)}{2c\sigma^2}\right)ds}-R_\infty=\widetilde U(z).
\end{align*}
The above observation implies that if $z^*$ is optimal for $\widetilde U$, then we have
$\sup_{b\in(0,1]}U(b)\ge U(\psi(z^*))\ge \widetilde U(z^*)=\sup_{z\in [0,\infty]}\widetilde U(z)\ge \sup_{b\in(0,1]}U(b).$
This proves the first claim and the ``if" part of the second claim. For the ``only if" part of the second claim, suppose $b^*$ is optimal for $U$, then
$U(b^*)=\sup_{b\in(0,1]}U(b)= \sup_{z\in [0,\infty]}\widetilde U(z)\ge \widetilde U(\infty)$. Since $U(b)<\widetilde U(\infty)$ for all $b\in(0,b_0)$, we must have $b^*\in [b_0,1]$. It follows that $b^*=\psi(t_{b^*})$ and
$\widetilde U(t_{b^*})=U(b^*)=\sup_{b\in(0,1]}U(b)= \sup_{z\in [0,\infty]}\widetilde U(z).$
The latter implies $t_{b^*}$ is optimal for $\widetilde U$.

In view of the claim, we can find the optimal value and all optimizers of $\sup_{b\in(0,1]}U(b)$ by solving the auxiliary problem. For each optimal $z^*$ for the auxiliary problem, we get an optimal $b^*=\psi(z^*)$. We then recover $f_{\mu^*}=f_{\mu(b^*)}$ by \eqref{eq:fstar} and $H_{\mu^*}$ by Theorem~\ref{thm:Erange}.
The optimal reward is given by
\[H^*(r)=H_{\mu^*}(r)+R_\infty-2c\sigma^2 \ln b^* =R_\infty+g(F^{-1}_{\mu^*}(r)\wedge t_{b^*})-g(t_{b^*}).\]
\qed

\subsection{Extension: completion rate-dependent pie and multiple equilibria}\label{sec:un-fixed pie}

Consider a reward pie that depends on the aggregate completion rate by time $T<\infty$. (That rate is 100\% when $T=\infty$.) Specifically, we assume the reward for finishing at time $t$ rank $r$, and when the population  completion rate is $\beta$, is given by
\begin{equation}\label{eq:rwd-unfixed-pie}
R(t,r,\beta)=1_{\{t\le T\}} H(r,\beta)+1_{\{t> T\}} R_\infty(\beta),
\end{equation}
where $H:[0,1]\times [F_{\tau^\circ_{x_0/\sigma}}(T),1]\mapsto \bbR$ and $R_\infty:  [F_{\tau^\circ_{x_0/\sigma}}(T),1]\mapsto \bbR$ are continuous in $\beta$.
What we have in mind is an organization or a society in which the overall wealth coming from individual projects is higher when more individuals complete their goals. Similar to the case of a fixed pie, the equilibrium also admits a semi-explicit characterization:
\begin{thm}\label{thm:explicit_solnB}
Let $T<\infty$. Suppose the reward function is of the form \eqref{eq:rwd-unfixed-pie}. Then, there exist at least one
equilibrium, and a distribution  $\mu$ is  an equilibrium completion time distribution if and only if
\begin{equation}\label{eq:NEquantileB}
T^\mu_r=F^{-1}_{\tau^\circ_{x_0/\sigma}}\left(\frac{1-F_{\tau^\circ_{x_0/\sigma}}(T)}{1-F_\mu(T)}\int_0^r \exp\left(\frac{R_\infty(F_\mu(T))-H(z,F_\mu(T))}{2c\sigma^2}\right)dz\right), \quad r\in [0, F_\mu(T)],
\end{equation}
where
\begin{equation*}
F_{\tau^\circ_{x_0/\sigma}}(t)=2\left(1-N\left(\frac{x_0}{\sigma \sqrt{t}}\right)\right).
\end{equation*}
Moreover, the equilibrium terminal completion rate $F_\mu(T)\in (0,1)$ is a solution of
\begin{equation}\label{eq:NErateB}
\frac{F_{\tau^\circ_{x_0/\sigma}}(T)}{1-F_{\tau^\circ_{x_0/\sigma}}(T)}=\frac{1 }{1-F_\mu(T)} \int_0^{F_\mu(T)} \exp\left(\frac{R_\infty(F_\mu(T))-H(z,F_\mu(T))}{2c\sigma^2}\right) dz.
\end{equation}
The associated value and expected total effort of the game are given by
\begin{equation}\label{eq:NEvalueB}
V(\mu)=R_\infty(F_\mu(T))+2c\sigma^2\ln \left(\frac{1-F_{\tau^\circ_{x_0/\sigma}}(T)}{1-F_\mu(T)}\right)
\end{equation}
and \eqref{eq:expected-effort}, respectively.
\end{thm}
\begin{proof}
The proof is similar to the case of a fixed pie. So we shall be brief.
The fixed point equation (in terms of the p.d.f.\ of the completion time distribution) for a freshly started game is
\[f_\mu(t)=\frac{u(t,0;\mu,c)}{u(0,x_0;\mu,c)}f_{\tau^\circ_{x_0/\sigma}}(t), \quad t\in[0,T].\]
Let $y(r):=F_{\tau^\circ_{x_0/\sigma}}(T^\mu_{r})$. As before, it can be shown that $\exp\left(\frac{H(r, F_\mu(T))}{2c\sigma^2}\right)y'(r)$ is independent of $r$ on $[0, F_\mu(T)]$, from which we obtain \eqref{eq:NEquantileB}.
Setting $r=F_\mu(T)$ in \eqref{eq:NEquantileB} leads to an equation for $F_\mu(T)$, namely \eqref{eq:NErateB}. The existence of a solution follows from the fact that
\begin{equation}\label{eq:phiB}
\phi(\beta):=\frac{1}{1-\beta}\int_0^\beta  \exp\left(\frac{R_\infty(\beta)-H(z,\beta)}{2c\sigma^2}\right)dz
\end{equation}
is a continuous function on $[0,1)$ satisfying $\phi(0)=0$ and $\lim_{\beta\rightarrow 1}\phi(\beta)=\infty$. Each solution of \eqref{eq:NErateB} yields an equilibrium completion time distribution $\mu$. The associated game value $V(\mu)=v(0,x_0;\mu,c)=2c\sigma^2 \ln u(0,x_0;\mu,c)$ follows from straightforward computation as before. The derivation of the expected total effort is the same as that in Proposition~\ref{prop:expected-effort}.
\end{proof}

Note that we do not necessarily have uniqueness under conditions of Theorem~\ref{thm:explicit_solnB}.  Equation \eqref{eq:NErateB} may have multiple solutions, leading to multiple equilibria.
For example, when $H(r,\beta)=R_\infty(\beta)+\beta$, one can check that equation \eqref{eq:NErateB} may have one, two or three solutions, depending on the value of $T, x_0, c, \sigma$ (see Figure~\ref{fig:unfixed-pie}).
Among the equilibria, the dominant one (i.e., the one with the highest game value) corresponds to the largest solution of \eqref{eq:NErateB}.
In this example, we see that by limiting the size of the projects or imposing  a long enough  deadline, the ``bad" equilibria can be avoided.

\begin{figure}[h]
\centering
\includegraphics[height=5.5cm]{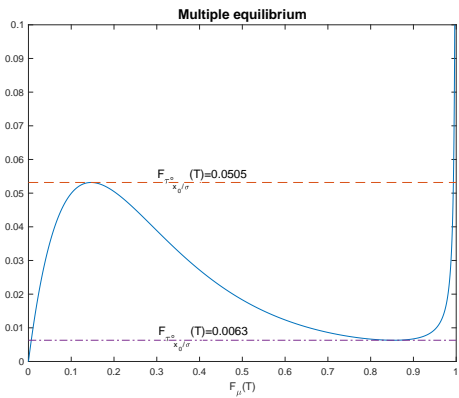}
\caption{The blue curve plots the right hand side of \eqref{eq:NErateB} as a function of $F_\mu(T)$ with $c=1$, $\sigma=0.25$ and $H(r,\beta)=R_\infty(\beta)+\beta$. The two horizontal lines plot the left hand side of \eqref{eq:NErateB} in the two critical cases. We see that there is a unique equilibrium when $F_{\tau^\circ_{x_0/\sigma}}<0.0063$ or $F_{\tau^\circ_{x_0/\sigma}}>0.0505$,  two equilibria when $F_{\tau^\circ_{x_0/\sigma}}=0.0063$ or $0.0505$, and three equilibria when $0.0063<F_{\tau^\circ_{x_0/\sigma}}<0.0505$. }
\label{fig:unfixed-pie}
\end{figure}

When reward $R(t,r,\beta)$ is independent of rank  $r$ and increasing in completion rate $\beta$, we have a contribution game where the interaction is not through the aggregate effort towards a single project as in, for example,  \cite{Georgiadis15}, but through the completion rate of many parallel projects. By gradually injecting rank dependence into the game, one can analyze the effect of competition on the completion rate and the game value.

Consider, for example,
\begin{align*}
R(t,r,\beta)&=\Pi(\beta) \left[\gamma+1_{\{t\le T\}}(1-\gamma) H_\eps (r)\right]
\end{align*}
where $\Pi(\beta):=K(1+\beta)$ represents the size of the pie, $\gamma\in[0,1]$ is the fraction of the pie that are used to reward participation, and $H_\epsilon(r)=1+\epsilon(1-2r)$ specifies how the remaining $(1-\gamma)$ fraction of the pie is shared among players who finish by the deadline and are ranked. A larger $\epsilon$ corresponds to a larger degree of competition (or inequality) among ranked players. When $\epsilon=0$, we have a pure contribution game. 
Figure~\ref{fig:contribution-competition} plots the equilibrium completion rate and game value against $\epsilon$ for $\gamma=0.5$. We see that the degree of inequality in the rewards has a positive effect on the equilibrium completion rate and the game value when the pie (that is, $K$) is small, an adverse effect when the pie is large, and a mixed effect when the size of the pie is moderate. The intuition is the same as  in the case of the fixed pie:
when the pie is large, the  demoralizing effect of competitiveness prevails. However, when the pie is small,
 a higher percentage of players gives up,  and   higher prize inequality makes the players who do not give up  exert higher effort,  as they,  competing within a smaller group,  are less discouraged by the prize inequality.

\begin{figure}[h]
\centering
\includegraphics[height=5.5cm]{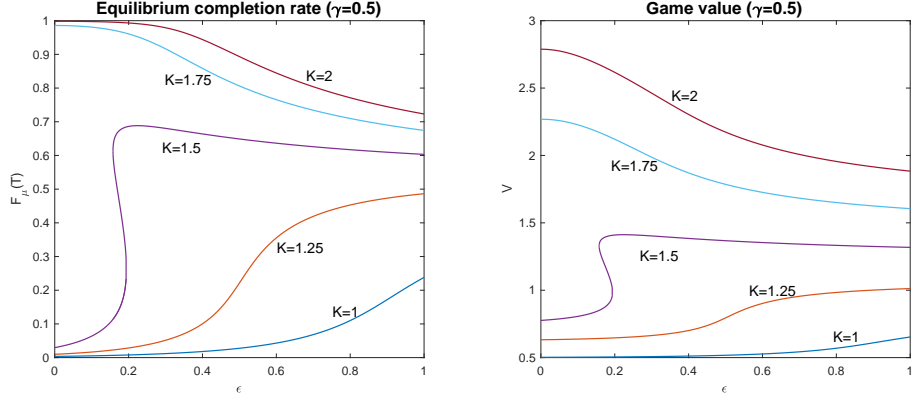}
\caption{$F_\mu(T)$ and $V$ against $\epsilon$. Parameters are set to $x_0=c=T=1$, $\sigma=0.25$ and $\gamma=0.5$. When $K=1.5$, there is an interval of $\epsilon$ where the mapping $\epsilon\mapsto F_\mu(T)$ is multi-valued; meaning the game has multiple equilibria.}
\label{fig:contribution-competition}
\end{figure}



\section{General case: existence, uniqueness, stability and $\epsilon$-Nash equilibrium}\label{sec:general_theory}

In this section, we work with general reward functions $R\in\mathcal{R}$, and no longer assume players are homogeneous. We introduce heterogeneity by assuming that a tournament is characterized by the initial condition $(t_0, \pi, m)$, where $t_0<T$ is the starting time, $\pi$ is the proportion of players that have already finished by time $t_0$, and $m\in\mathcal{P}(\bbR_{++}\times [\underline{c},\overline{c}])\subseteq\bbR^2_{++}$  is the joint distribution of the time-$t_0$ location, denoted by $\xi$, and the  cost of effort $c$ of the  population still playing.\footnote{Inhomogeneity in $\sigma\in[\underline{\sigma},\overline{\sigma}]$ could be easily incorporated as well, by taking $m$ to be the joint distribution of $(\xi,c,\sigma)$.} We assume that we are in the non-degenerate case $\pi<1$.

\subsection{Existence of a fixed point}\label{subsec:existence}


Given an initial condition $(t_0, \pi, m)$, for a representative player
we introduce a binary random variable $\theta\in\{0,1\}$ with $P(\theta=1)=\pi$ and $P(\theta=0)=1-\pi$. We interpret $\theta$ as the player's game completion status at time $t_0$; $\theta=1$ means game is completed and $\theta=0$ means game is in progress.
If $\theta=0$, we further pick a time-$t_0$ location $\xi$ and cost of effort coefficient $c$ randomly from the joint distribution $m$. The initial randomizations $(\theta,\xi,c)$ are assumed to be independent of the Brownian motions driving the state process. 

For a given starting time $t_0$, fix a $[t_0,\infty)$-supported completion time distribution $\mu$ of the population with $\mu(\{t_0\})=\pi$, i.e., $\mu$ is a modification of the completion time distribution by moving all the mass before time $t_0$ to time $t_0$. Since the probability for finishing exactly at time $t_0$ is zero, such a modification does not affect the player's optimization problem.

 The existence proof will be based on a fixed-point argument.
For that purpose,
we now define the best response mapping $\Phi^{t_0,\pi,m}$ which maps $\mu$ to the distribution $\mathcal{L}(\tau_\mu)$ of the optimal completion time $\tau_\mu$. 
{Assume for the moment that $T<\infty$.}
If $\theta=1$, we set $\tau_\mu:=t_0$. If $\theta=0$, we solve the player's optimization problem with initial randomization $(\xi,c)$ and using payoff function $R_\mu(t)=R(t,F_\mu(t))$. We set $\tau_\mu:=\inf\{t\in [t_0,T]: X^{\mu}_t=0\}\wedge \Delta$,\footnote{We use the convention that $\inf \emptyset =\infty$.} where $\Delta$ is a fixed number in $(T,\infty)$ corresponding to incompletion, and
\[dX^{\mu}_t=-a^*(t,X^\mu_t;\mu,c)dt+\sigma dB_t, \quad X^\mu_{t_0}=\xi.\]

By Proposition~\ref{prop:taupdf}, we know that for $t\in (t_0,T]$, the density of the optimal completion time is given by
\begin{equation*}
\bbP(\tau_\mu\in dt|\theta=1,\xi,c)=\frac{u(t,0;\mu,c)}{u(t_0,\xi;\mu, c)}f_{\tau^\circ_{\xi/\sigma}}(t-t_0).
\end{equation*}
It follows that
\begin{equation}\label{taupdf1}
\bbP(\tau_\mu=t_0)=\pi, \quad \bbP(\tau_\mu=\Delta)=1-\bbP(\tau_\mu\le T),
\end{equation}
\begin{equation}\label{taucdf}
\bbP(\tau_\mu\le t)
=\pi+(1-\pi)E\left[\int_{t_0}^t \frac{u(s,0;\mu,c)}{u(t_0,\xi;\mu,c)}f_{\tau^\circ_{\xi/\sigma}}(s-t_0)ds\right], \ t\in [t_0,T].
\end{equation}
In terms of the p.d.f., we have
\begin{equation}\label{taupdf2}
\bbP(\tau_\mu\in  dt)
=(1-\pi)E\left[\frac{u(t,0;\mu,c)}{u(t_0,\xi;\mu, c)}f_{\tau^\circ_{\xi/\sigma}}(t-t_0)\right], \ t\in(t_0,T].
\end{equation}
Let $\bbT:=[t_0,T]\cup\{\Delta\}$. Our goal is to find a fixed point of
$\Phi^{t_0,\pi,m}: \mu \mapsto \mathcal{L}(\tau_\mu)$
in the space $\mathcal{P}(\bbT)$ of probability measures on  $\bbT$,  which is compact in the topology of weak convergence. Since $\mathbb{T}$ is compact, weak convergence on $\mathcal{P}(\mathbb{T})$ can be metrized by the 1-Wasserstein metric (see \cite[Corollary 6.13]{OptimalTransport2009}):
\begin{align*}
W_1(\mu, \mu')&:=\inf\left\{\int_{\bbT^2} |x-y| d\pi(x,y): \pi\in \mathcal{P}(\bbT^2) \text{ with marginals } \mu \text{ and } \mu'\right\}.
\end{align*}
When $T=\infty$, we can define $\Phi^{t_0, \pi, m}$ in a similar fashion with $\Delta:=\infty$. In this case, additional care is need to ensure the compactness of $\mathcal{P}([t_0, \infty])$; see the proof of Theorem~\ref{thm:existence} for details.

The main ingredient of fixed point theorems is the continuity of the best response mapping which, in view of \eqref{taucdf}, amounts to the continuity of $u(t,x; \mu,c)$ with respect to $\mu$. We can guarantee this by restricting ourselves to reward functions $R\in\mathcal{R}$ that are continuous in the rank variable, or more generally, of the form
$R(t,r)=G(t,H(r))$
for some functions $G:\bbR_+\times\bbR\mapsto \bbR$ and $H:[0,1]\mapsto \bbR$ such that $y\mapsto G(t,y)$ is continuous for each $t$, and $H$ is monotone. In this way, the (potential)  discontinuity in rank is decoupled from time variable.
To fix notation, let
\[\mathcal{R}_D:=\{R\in\mathcal{R}: R(t,r)=G(t,H(r)), y\mapsto G(t,y) \text{ is continuous for each $t$}, H \text{ is monotone}\},\]
where
$D$ stands for ``decoupled" or ``decomposable".

There is another case that is sufficient for existence of a fixed point. Consider piecewise constant $H$ of the form
\begin{equation}\label{eq:H}
H(r)=\sum_{k=1}^{d} R_k 1_{[r_{k-1},r_k)}(r)+R_{d+1} 1_{[r_d,1]}, \ 0=r_0<r_1<\cdots<r_{d}<1,
\end{equation}
and
 define
\[\mathcal{R}_{S}:=\{R\in\mathcal{R}: R(t,r)=G(t,H(r)), H\text{ is of the form \eqref{eq:H}}\},\]
where $S$ stands for ``step function". 
The nice thing about $R\in\mathcal{R}_S$ is that $R_\mu$ depends on $\mu$ only through its quantiles $(T^\mu_1, ..., T^\mu_{d})$; the shape of $\mu$ between each neighboring $T^\mu_k$'s becomes irrelevant. In other words, each $\mu$ can be encoded to $d$ numbers. As a consequence, we have a fixed point problem in a $d$-dimensional space instead of the infinite dimensional space of measures.

We now state our general existence theorem.

\begin{thm}\label{thm:existence}
Assume $R\in\mathcal{R}_D\cup \mathcal{R}_{S}$. Then $\Phi^{t_0,\pi,m}$ has a fixed point. Thus, an equilibrium exists.
\end{thm}
\begin{proof}
 Assume first  $T<\infty$. We will look for fixed points in the compact (Hausdorff) metric space $(\mathcal{P}(\bbT),W_1)$.

 Assume $R\in\mathcal{R}_D$.
In order for $F_\mu$ not to hit the (potential) discontinuity of $H$ on a non-negligible subset of $[t_0,T]$, we need to make sure that it is not flat there. Observe from \eqref{taupdf2} that any fixed point $\mu$ of $\Phi^{t_0,\pi,m}$, if it exists, must have a strictly increasing c.d.f.\ on $[t_0,T]$. This motivates us to define the following non-empty, convex subset of $\mathcal{P}(\bbT)$:
\begin{equation}
\mathcal{D}_{\underline{f}}:=\left\{\mu\in\mathcal{P}(\bbT): \mu([a,b])\ge \int_a^b \underline{f}(t)dt \quad \forall a,b\in [t_0,T]\right\},
\end{equation}
where
\begin{equation}\label{eq:flbdd}
\underline{f}(t):=(1-\pi)\exp\left(\frac{R_\infty-R(0,0)}{2\overline{c}\sigma^2}\right)Ef_{\tau^\circ_{\xi/\sigma}}(t-t_0)>0.
\end{equation}
We claim then, that

(i) $\Phi^{t_0,\pi,m}(\mathcal{P}(\bbT))\subseteq \mathcal{D}_{\underline{f}}$.

 (ii) $\mathcal{D}_{\underline{f}}$ is closed under weak convergence.

(iii) Any $\mu\in\mathcal{D}_{\underline{f}}$ has strictly increasing c.d.f.\ on $[t_0,T]$.

 (i) follows from \eqref{taupdf2} and \eqref{u-bdd}. (ii) holds because $\mu_n$ converges to $\mu$ weakly if and only if $\limsup_n \mu_n(C)\le \mu(C)$ for any closed sets $C\subseteq\bbT$.
(iii) is obvious.

Note that weak convergence is equivalent to $W_1$ convergence on $\mathcal{P}(\bbT)$, so that (ii) implies that $\mathcal{D}_{\underline{f}}$  is a closed subset of a compact space, hence also compact. To apply Schauder's fixed point theorem, it remains to show that $\Phi^{t_0,\pi,m}$ is continuous on $\mathcal{D}_{\underline{f}}$.

Let $\{\mu_k\}\subseteq \mathcal{D}_{\underline{f}}$ be such that $W_1(\mu_k,\mu)\rightarrow 0$. We wish to show
\[W_1(\Phi^{t_0,\pi,m}(\mu_k), \Phi^{t,\pi_0,m}(\mu))=W_1(\mathcal{L}(\tau_{\mu_k}),\mathcal{L}(\tau_{\mu}))\rightarrow 0,\]
or, equivalently, $\mathcal{L}(\tau_{\mu_k})$ converges to $\mathcal{L}(\tau_{\mu})$ weakly. It suffices to show $F_{\tau_{\mu_k}}$ converges to $F_{\tau_\mu}$ pointwise on $[t_0,T]$.
Recall \eqref{taucdf}:
\[F_{\tau_{\mu_k}}(t)=\pi+(1-\pi)E\left[\int_{t_0}^t\frac{u(s,0;\mu_k,c)}{u(t_0,\xi;\mu_k,c)}f_{\tau^\circ_{\sigma/\xi}}(s-t_0)ds\right], \ t\in [t_0,T].\]
Since $u$ is bounded by a constant independent of $\mu_k$, to show the expected values converge, we only need to show that for every $x\ge 0$ and realization of $c$, $u(t,x;\mu_k,c)$ converges to $u(t,x;\mu,c)$ for a.e.\ $t\in(t_0,T]$. Also, recall that
\[u(t,x;\mu_k,c)=E\left[\exp\left(\frac{R_{\mu_k}(t+\frac{x^2}{\sigma^2}\tau^\circ_1)}{2c\sigma^2}\right)\right],\]
where $R_{\mu_k}(t)=R(t,F_{\mu_k}(t))$. Since $W_1(\mu_k,\mu)\rightarrow 0$, $\mu_k$ also converges to $\mu$ weakly, which further implies $F_{\mu_k}(t)$ converges to $F_{\mu}(t)$ at every point of continuity of $F_{\mu}$.
Since $F_\mu(t)$ and $H(r)$ are monotone, they have at most countably many points of discontinuity, denoted by $t_1, t_2, \ldots$ and $r_1, r_2, \ldots,$ respectively. Now, since $\mu\in\mathcal{D}_{\underline{f}}$, $F_\mu$ is strictly increasing on $[t_0,T]$ by item (iii) above. It follows that each $r_{i}$ can only be attained by $F_\mu(t)$ for at most one $t\in [t_0,T]$, denoted by $\tilde t_{i}$. For any $t$ not belonging to the countable set $\{t_i, \tilde{t}_i\}_{ i=1,2,\ldots}$, we have that $F_{\mu_k}(t)$ converges to $F_\mu(t)$, that $H(F_{\mu_k}(t))$ converges to $H(F_{\mu}(t))$ and that $R_{\mu_k}(t)=G(t,H(F_{\mu_k}(t)))$ converges to $R_{\mu}(t)=G(t,H(F_{\mu}(t)))$. Bounded convergence theorem then implies that $u(t,x;\mu_k,c)$ converges to $u(t,x;\mu,c)$ pointwise.

(b)  Assume $R\in\mathcal{R}_S$.
Let $T^\mu_k=T^\mu_{r_k}$ be the $(r_k)$-quantile of $\mu$ for $k=1, \ldots, d$, and set $T^
\mu_0=0$, $T^\mu_{d+1}=\Delta$.
Exploiting the piecewise constant structure of $H$, we have
\begin{equation}\label{eq:g}
R_{\mu}(t)=\sum_{k=1}^{d+1} G(t,R_k)1_{[T^\mu_{k-1},T^\mu_k)}(t), \quad t\in [t_0,T].
\end{equation}
Since $R_\mu$ depends on $\mu$ only through its quantiles $(T^\mu_1, ..., T^\mu_{d})$, as pointed out earlier, the fixed point problem can be reduced to a finite dimensional one in the space
\[\mathcal{Q}:=\{(T_1,\ldots, T_d)\in \widetilde\bbT^d: T_1\le T_2\le \cdots \le T_d\},\]
where $\widetilde \bbT:=[t_0,\Delta]$ is the convex hull of $\bbT$.

We will use the notation $R_{\bq}(t)$, $u(t,x;\bq,c)$ and $\tau_\bq$ instead of $R_{\mu}(t)$, $u(t,x;\mu,c)$ and $\tau_\mu$, where $\bq=(T_1, \ldots, T_d)$ is a vector of quantiles. The best response mapping is reformulated as
\[\Theta: \bq\mapsto  (T^{\mathcal{L}(\tau_{{\bq}})}_{1}, \ldots, T^{\mathcal{L}(\tau_{{\bq}})}_{d}).\]
Any fixed point $\bq$ of $\Theta$ induces a fixed point $\mathcal{L}(\tau_{\bq})$ of $\Phi^{t_0,\pi,m}$. Note that $\mathcal{Q}$ is a convex, compact set which is mapped into itself under $\Theta$.
Once we show $\Theta$ is continuous on $\mathcal{Q}$, we can use Brouwer's fixed point theorem to conclude that $\Theta$ has a fixed point. Continuity is proved as follows:

Let
$\bq_n=(T^n_1, \ldots, T^n_d)\rightarrow \bq=(T_1, \ldots, T_d) $ in $\mathcal{Q}$. It is easy to see from \eqref{eq:g} that $R_{\bq_n}(t)$ converges to $R_{\bq}(t)$ if $t\notin\{T_1, \ldots, T_d\}$.
Since $\tau^\circ_{x/\sigma}$ is non-atomic, bounded convergence theorem implies for any $(t,x)\in [0, T]\times\bbR_{++}$ and any (positive) realization of $c$,
\begin{equation*}
u(t,x;\bq_n,c)\rightarrow u(t,x;\bq,c)  \ \text{ as } n\rightarrow \infty.
\end{equation*}
Once we have the pointwise convergence of $u$, we can then use \eqref{taupdf2} to obtain weak convergence of
$\mathcal{L}(\tau_{\bq_n})$ to $\mathcal{L}(\tau_\bq)$. To show $T^{\mathcal{L}(\tau_{\bq_n})}_k\rightarrow T^{\mathcal{L}(\tau_{\bq})}_k$, we consider three cases of $k\in\{1,\ldots, d\}$.

Case (i). If $r_k\le \pi$, then by our construction, $T^{\mathcal{L}(\tau_{\bq_n})}_k=T^{\mathcal{L}(\tau_{\bq})}_k=t_0$.

Case (ii). If $\pi<r_k\le F_{\tau_\bq}(T)$, then we observe from \eqref{taupdf2} that $\tau_{\bq}$ has strictly positive density in $(t_0,T)$. So the quantile function of $\mathcal{L}(\tau_\bq)$ is continuous in $(\pi,F_{\tau_\bq}(T)]$. Weak convergence of $\mathcal{L}(\tau_{\bq_n})$ to $\mathcal{L}(\tau_\bq)$ then implies $T^{\mathcal{L}(\tau_{\bq_n})}_k\rightarrow T^{\mathcal{L}(\tau_{\bq})}_k$.

Case (iii). If $r_k>F_{\tau_\bq}(T)$, then $T^{\mathcal{L}(\tau_{\bq})}_k=\Delta$. By weak convergence of $\mathcal{L}(\tau_{\bq_n})$ to $\mathcal{L}(\tau_\bq)$ and the continuity of $F_{\tau_\bq}$ at time $T$, we have that $F_{\tau_{\bq_n}}(T)\rightarrow F_{\tau_\bq}(T)$. This implies $r_k>F_{\tau_{\bq_n}}(T)$ and hence $T^{\mathcal{L}(\tau_{\bq_n})}_k=\Delta$ for $n$ sufficiently large.

Finally, we consider the case $T=\infty$.
Assume first $R\in \mathcal{R}_D$.
In this case, it is not clear whether $\tau_\mu$ has finite moments. In the worst case when the reward is constant after a certain rank, $\tau_\mu$ inherits the tail property of a Brownian motion first passage time, and thus, has infinite mean. Hence, we cannot directly work with the space $[t_0,\infty)$ using its standard topology.
Instead, we equip $[t_0,\infty]$ with the order topology under which it is homeomorphic to $[0,1]$.  For example, one can take the homeomorphism to be $\bar F:=E[F_{t_0+\tau^\circ_{\xi/\sigma}}]$ with the convention that any c.d.f.\ maps $\infty$ to 1. With this choice of topology, $[t_0,\infty]$ becomes compact and metrizable.
It follows that the space $\mathcal{P}([t_0,\infty])$ is also compact for the topology of weak convergence
which can be assigned the corresponding 1-Wasserstein metric: 
\begin{align*}
W_1(\mu, \mu')&:=\inf\left\{\int_{[t_0,\infty]^2} |\bar F(x)-\bar F(y)| d\pi(x,y): \pi\in \mathcal{P}([t_0,\infty]^2) \text{ with marginals } \mu \text{ and } \mu'\right\}.
\end{align*}

By Portmanteau theorem for Polish spaces, weak convergence of $\mu_k$ to $\mu$ in $\mathcal{P}([t_0,\infty])$ still implies the convergence of $F_{\mu_k}$ to $F_\mu$ at the continuity points of $F_\mu$, although the topology involved is slightly non-standard. Indeed, for any $t\in[t_0,\infty]$, since $[t_0,t)$ is open and $[t_0,t]$ is closed, we have $\liminf_k \mu_k([t_0,t))\ge \mu([t_0,t))$ and $\limsup_k \mu_k([t_0,t])\le \mu([t_0,t])$. If $F_\mu$ is continuous at $t$, then $\limsup_k F_{\mu_k}(t)\le F_{\mu}(t)=F_{\mu}(t-)\le \liminf_k F_{\mu_k}(t-)\le \liminf_k F_{\mu_k}(t)$. When $\mu\in\mathcal{P}([t_0,\infty))$, which holds for any measure in the range of $\Phi^{t_0,\pi,m}$, weak convergence of $\mu_n$ to $\mu$ in the usual sense implies weak convergence in $\mathcal{P}([t_0,\infty])$, since any bounded continuous function on $[t_0,\infty]$ (equipped with the order topology) is also bounded continuous on $[t_0,\infty)$ (equipped with the standard topology). This means that, to show $W_1(\Phi^{t_0,\pi,m}(\mu_k), \Phi^{t,\pi_0,m}(\mu))\rightarrow 0$, pointwise convergence of $F_{\mu_k}$ to $F_\mu$ on $[t_0,\infty)$ is still sufficient as before. Thus, the proof from the case $T<\infty$ remains valid.

For the case $R\in \mathcal{R}_S$, we similarly identify
$\mathcal{Q}=\{(T_1,\ldots, T_d)\in [t_0,\infty]^d: T_1\le T_2\le \cdots \le T_d\}$
with the convex, compact set
$\mathcal{Y}:=\{(y_1,\ldots, y_d)\in [0,1]^d: y_1\le y_2\le\cdots\le y_d\}$
under the homeomorphism $\mathbf{\bar F}:(T_1,\ldots, T_d)\mapsto (\bar F(T_1), \ldots, \bar F(T_d))$. Since the extended real line is not a vector space, Schauder's fixed point theorem does not directly apply to $\mathcal{Q}$. However, we can look for a fixed point of $\Xi:=\mathbf{\bar F}\circ\Theta\circ\mathbf{\bar F^{-1}}$ in $\mathcal{Y}$ instead, where $\Theta$ is the original best response mapping. Any fixed point $\by$ of $\Xi$ induces a fixed point $\bq=\mathbf{\bar F^{-1}}(\by)$ of $\Theta$. Since $\mathbf{\bar F}$ is a homeomorphism, $\Xi$ is continuous if and only if $\Theta$ is continuous. The proof then proceeds similarly as before.
\end{proof}

\subsection{Uniqueness of the fixed point}\label{subsec:uniqueness}


Denote by $\mathcal{AC}^T_{t_0}$ the set of measures $\mu\in \mathcal{P}(\bbR_+)$ with CDF $F_\mu$ that is absolutely continuous on $[t_0,T]\cap[t_0,\infty)$. The following monotonicity condition, which we will show is sufficient for uniqueness, is in the same spirit as \cite{LasryLions.07} (also see \cite[Section 3.2]{BZ16}).

\begin{assumption}\label{assumption-uniqueness}
For any $\mu, \mu'\in\mathcal{AC}^T_{t_0}$ such that $F_{\mu}(t_0)=F_{\mu'}(t_0)$, we have
\[\int_{t_0}^T (R_{\mu}-R_{\mu'})(t) d(\mu-\mu')(t)\leq 0.\]
\end{assumption}

\begin{prop}\label{prop:check-uniqueness}
Assumption~\ref{assumption-uniqueness} is satisfied if the function
\begin{equation*}
h(t,x,y):=\begin{cases}
\frac{R(t,x)-R(t,y)}{x-y} , & t\in [t_0,T]\cap[t_0,\infty), 0\le x\ne y\le 1\\
\frac{\partial}{\partial x} R(t,x) & t\in [t_0,T]\cap[t_0,\infty), 0\le x=y\le 1
\end{cases}
\end{equation*}
is well-defined, non-positive, increasing in $t,x,y$, and the functions $R(t,F_\mu(t))$, $h(t, F_\mu(t), F_{\mu'}(t))$ are absolutely continuous on $[t_0,T]\cap[t_0,\infty)$ for any $\mu, \mu\in \mathcal{AC}^T_{t_0}$.
In particular, Assumption~\ref{assumption-uniqueness} is satisfied for reward functions of the form
$R(t,r)=1_{\{t\le T\}}H(r)+1_{\{t>T\}}R_\infty,$
where $H$ is a convex, decreasing, $C^2[0,1]$ function, such that $H(r)\ge R_\infty$ for all $r$.
\end{prop}
\begin{proof}
Using integration by parts for absolutely continuous functions, we have
 \begin{align*}
&\int_{t_0}^T (R_{\mu}-R_{\mu'})(t) d(\mu-\mu')(t)\\
&=(R_\mu(T)-R_{\mu'}(T))(F_\mu(T)-F_{\mu'}(T))-\int_{t_0}^T (F_{\mu}-F_{\mu'})(t) d(R_{\mu}-R_{\mu'})(t)\\
 &=(F_{\mu}(T)-F_{\mu'}(T))^2h(T, F_{\mu}(T),F_{\mu'}(T))-\int_{t_0}^T(F_{\mu}-F_{\mu'})(t) d\left[(F_{\mu}-F_{\mu'})(t)h(t,F_{\mu}(t),F_{\mu'}(t))\right]\\
 &\le -\int_{t_0}^T(F_{\mu}-F_{\mu'})(t) \left[(F_{\mu}-F_{\mu'})(t)dh(t,F_{\mu}(t),F_{\mu'}(t))+h(t,F_{\mu}(t),F_{\mu'}(t))d(F_{\mu}-F_{\mu'})(t)\right]\\
 &= -\int_{t_0}^T(F_{\mu}-F_{\mu'})^2(t) dh(t,F_{\mu}(t),F_{\mu'}(t)) -\int_{t_0}^T (R_{\mu}-R_{\mu'})(t) d(\mu-\mu')(t)
 \end{align*}
where we have used $F_\mu(t_0)=F_{\mu'}(t_0)$ and $h\leq 0$ for the boundary terms. Re-arranging terms and using the monotonicity of $h (t,F_\mu(t),F_{\mu'}(t))$, we get
 \begin{align*}
 &2\int_{t_0}^T (R_{\mu}-R_{\mu'})(t) d(\mu-\mu')(t)\le -\int_{t_0}^T(F_{\mu}-F_{\mu'})^2(t) dh(t,F_{\mu}(t),F_{\mu'}(t)) \leq 0.
 \end{align*}
\end{proof}


 \begin{thm}\label{uniq}
Under Assumption~\ref{assumption-uniqueness}, $\Phi^{t_0,\pi,m}$ has at most one fixed point.
\end{thm}
\begin{proof}
The proof is similar to that of \cite[Proposition 3.1]{BZ16}. Suppose $\mu$ and $\mu'$ are two fixed points of $\Phi^{t_0,\pi, m}$. Then we obviously have $F_{\mu}(t_0)=F_{\mu'}(t_0)=\pi$. Moreover, $\mu, \mu'\in\mathcal{AC}^T_{t_0}$ by \eqref{taucdf}.
Assumption~\ref{assumption-uniqueness} then implies that
\begin{equation}\label{uniqueness-eq-0}
\int_{t_0}^T [R_{\mu}(t)-R_{\mu'}(t)] d(\mu-\mu')(t)\le 0.
\end{equation}

Write $v(t,x;c):=v(t,x;\mu,c)$ and $v'(t,x;c):=v(t,x;\mu',c)$.
Let $(\theta, \xi, c)$ be a randomization according to the initial condition $(t_0,\pi,m)$, and let $X^\mu$, $X^{\mu'}$, $\tau_\mu$, $\tau_{\mu'}$ be constructed as in Section~\ref{subsec:existence}. Also define $\tau^\eps_\mu$ and $\tau^\eps_{\mu'}$ to be the first hitting time of level $\eps>0$.
By It\^o's lemma and the PDE satisfied by $v$ and $v'$, we have

\begin{equation*}
\begin{aligned}
&E\left[R_{\mu}(\tau_\mu)-R_{\mu'}(\tau_\mu)\right]\\
&= E\left[1_{\{\theta=1\}} (R_{\mu}(t_0)-R_{\mu'}(t_0))\right]+E\left[1_{\{\theta=0\}} (v-v')(\tau_\mu\wedge T, X^\mu_{\tau_\mu\wedge T};c)\right]\\
&=\pi [R_\mu(t_0)-R_{\mu'}(t_0)]+\lim_{\eps\rightarrow 0} E\left[1_{\{\theta=0\}} (v-v')(\tau^\eps_\mu\wedge T, X^\mu_{\tau^\eps_\mu\wedge T};c) \right]\\
&=\pi [R_\mu(t_0)-R_{\mu'}(t_0)]+\lim_{\eps\rightarrow 0}E \left[1_{\{\theta=0\}}\left\{ (v-v')(t_0,\xi;c)+\int_{t_0}^{\tau^\eps_\mu\wedge T} \frac{1}{4c}(v_x-v'_x)^2(t,X^\mu_t;c)dt\right\}\right]\\
&=\pi [R_\mu(t_0)-R_{\mu'}(t_0)]+E \left[1_{\{\theta=0\}}\left\{ (v-v')(t_0,\xi;c)+\int_{t_0}^{\tau_\mu\wedge T} \frac{1}{4c}(v_x-v'_x)^2(t,X^\mu_t;c)dt\right\}\right].
\end{aligned}
\end{equation*}
Using  $\mu=\Phi^{t_0,\pi,m}(\mu)=\mathcal{L}(\tau_\mu)$ and  $R_{\mu}(t_0)=R(t_0,F_{\mu}(t_0))=R(t_0,\pi)=R(t_0,F_{\mu'}(t_0))=R_{\mu'}(t_0)$, we get
\begin{equation}\label{uniqueness-eq-1}
\begin{aligned}
&\int_{t_0}^\infty [R_{\mu}(t)-R_{\mu'}(t)] d\mu(t)\\
&=E\left[1_{\{\theta=0\}}(v-v')(t,\xi;c)\right]+E\left[1_{\{\theta=0\}}\int_{t_0}^{\tau_\mu\wedge T} \frac{1}{4c}(v_x-v'_x)^2(t,X^\mu_t;c)dt\right]
\end{aligned}
\end{equation}
Exchanging the roles of $\mu$ and $\mu'$ in \eqref{uniqueness-eq-1} and adding the resulting equation to  \eqref{uniqueness-eq-1} yields
\begin{align*}
&\int_{t_0}^\infty [R_{\mu}(t)-R_{\mu'}(t)] d(\mu-\mu')(t)\\
&= E\left[\frac{1_{\{\theta=0\}}}{4c}\left\{\int_{t_0}^{\tau_\mu\wedge T}(v_x-v'_x)^2(t,X^\mu_t;c)dt+\int_{t_0}^{\tau_{\mu'}\wedge T}(v'_x-v_x)^2(t,X^{\mu'}_t;c)dt\right\}\right]\geq 0.
\end{align*}
Since $R_\mu(t)=R_{\mu'}(t)=R_\infty$ for any $t> T$, the left hand side can be replaced by
\begin{equation*}
\int_{t_0}^T [R_{\mu}(t)-R_{\mu'}(t)] d(\mu-\mu')(t)
\end{equation*}
which is non-positive by \eqref{uniqueness-eq-0}. So on the set $\{\theta=0\}$, we have
\[(v_x-v'_x)^2(t,X^\mu_t;c)1_{\{t_0\le t< \tau_{\mu}\wedge T\}}=(v'_x-v_x)^2(t,X^{\mu'}_t;c)1_{\{t_0\le t< \tau_{\mu'}\wedge T\}}=0 \quad P\times dt\text{-a.e.}\]
The first term being zero implies
\[v_x(t,X^{\mu}_t;c)1_{\{t_0\le t< \tau_{\mu}\wedge T\}}=v'_x(t,X^{\mu}_t;c)1_{\{t_0\le t< \tau_{\mu}\wedge T\}} ~P\times dt\text{-a.e.}\]
By the uniqueness of the solution of the SDE \eqref{SDE}, we have $X^{\mu}=X^{\mu'}$ a.s.\ up to time $\tau_\mu\wedge T$. Thus, on the set $\{\tau_\mu\le T\}$, we have $\tau_{\mu}= \tau_{\mu'}$ a.s.\ On the set $\{\tau_\mu> T\}$, we also have $\tau_{\mu'}>T$ since $X^{\mu'}$ coincides with $X^\mu$ up to time $T$, and consequently, $\tau_\mu=\tau_{\mu'}=\Delta$ by our construction.
In summary, we have proved that $\tau_\mu=\tau_{\mu'}$ a.s.\ and thus, $\mu=\mathcal{L}(\tau_\mu)=\mathcal{L}(\tau_{\mu'})=\mu'$.
\end{proof}

\subsection{Stability of the fixed point} How stable or sensitive is the equilibrium with respect to changes in the reward function? This question is not only of interest in its own, but also provides the basis for numerical computation of the mean field equilibrium (see Section~\ref{subsec:numm}).
Specifically, the stability result guarantees that the sequence of equilibria of discretized (mean field) games,
if converging, will converge weakly to an equilibrium of the original game.

Introduce a Lipschitz condition:
\[\mathcal{R}_{DL}:=\{R\in\mathcal{R}_D: \exists L_G\in\bbR \text{ s.t.\ } |G(t,y_1)-G(t,y_2)|\le L_G |y_1-y_2| \ \forall t\in\bbR_+, y_1, y_2\in \bbR\}.\]
We then have the following stability (continuity) result.

\begin{thm}\label{thm:stability}
Suppose $\{R_n(t,r)=G(t,H_n(r))\}\subseteq\mathcal{R}_{DL}$ is a sequence of reward functions which converges to $R(t,r)=G(t,H(r))\in\mathcal{R}_{DL}$ in the following sense:
\[\| H_n-H\|_{L^1[0,1]}\rightarrow 0 \text{ as } n\rightarrow \infty.\]
Let $\mu_n$ be a fixed point of the best response mapping $\Phi^{t_0,\pi,m}_n$ associated with $R_n$ (given by Theorem~\ref{thm:existence}). Then $\mu_n$ has a subsequence which converges weakly
to a fixed point $\mu$ of the best response mapping $\Phi^{t_0,\pi,m}$ associated with $R$.
\end{thm}
\begin{proof}
Assume first $T<\infty$.

First, we construct a weak limit. Since each $\mu_n$ belongs to the compact space $\mathcal{P}(\bbT)$, we can extract a subsequence of $\mu_n$ which converges weakly to some $\mu$. With a slight abuse of notation, we still denote the convergent subsequence by $\mu_{{n}}$.

Next, from part (a) of the the proof of Theorem~\ref{thm:existence}, we know $\Phi^{t_0,\pi,m}$ is continuous on $(\mathcal{D}_{\underline{f}},W_1)$. It follows that $\tilde\mu_n:=\Phi^{t_0,\pi,m}(\mu_{n})$ converges in $W_1$-metric and also weakly to $\Phi^{t_0,\pi,m}(\mu)$.

The main task is to show the the  c.d.f.s of $\mu_n=\Phi^{t_0,\pi,m}_{n}(\mu_n)$ and of $\tilde\mu_n=\Phi^{t_0,\pi,m}(\mu_n)$ are uniformly close for $n$ large. More precisely, we claim that
\begin{equation}\label{eq:est0}
\|F_{\mu_n}-F_{\tilde\mu_n}\|_\infty\le C\|H_n-H\|_{L^1[0,1]}
\end{equation}
for some constant $C$ independent of $n$. Denote by $u_n$ the transformed value function associated with the reward function $R_n$. The proof relies on the following two estimates:
\begin{equation}\label{eq:est1}
E|u_n(t_0,\xi;\mu_n,c)-u(t_0,\xi;\mu_n,c)|\le C_1 \|H_n-H\|_{L^1[0,1]},
\end{equation}
\begin{equation}\label{eq:est2}
E\int_{t_0}^T |u_n(t,0;\mu_n,c)-u(t,0;\mu_n,c)|f_{\tau^\circ_{\xi/\sigma}}(t-t_0)dt\le C_2 \|H_n-H\|_{L^1[0,1]},
\end{equation}
where the constants $C_1$ and $C_2$ do not depend on $n$.

By \eqref{taupdf2} and \eqref{u-bdd}, we know the p.d.f.\ of $\mu_n$, denoted by $f_{\mu_n}$, is bounded from below in $(t_0,T)$ by $\underline{f}$ (defined in \eqref{eq:flbdd}).
By \eqref{u}, the Lipschitz continuity of $G(t,\cdot)$, the assumption that $G(t,x)=R_\infty$ for all $t>T$, and \eqref{eq:flbdd},
we obtain
\begin{align*}
&E\left[|u_n(t_0,\xi;\mu_n,c)-u(t_0,\xi;\mu_n,c)|\right]\\
&\le E\left[E\left[\left|\exp\left(\frac{(R_n)_{\mu_n}(t_0+\tau^\circ_{\xi/\sigma})}{2c\sigma^2}\right)-\exp\left(\frac{R_{\mu_n}(t_0+\tau^\circ_{\xi/\sigma})}{2c\sigma^2}\right)\right|\bigg| \xi, c\right]\right]\\
&\le E\left[\frac{L_G}{2c\sigma^2}\exp\left(\frac{R(0,0)}{2c\sigma^2}\right)E\left[1_{\{t_0+\tau^\circ_{\xi/\sigma}\le T\}}\left|H_n(F_{\mu_n}(t_0+\tau^\circ_{\xi/\sigma}))-H(F_{\mu_n}(t_0+\tau^\circ_{\xi/\sigma}))\right|\bigg| \xi, c\right]\right]\\
&= E\left[\frac{L_G}{2c\sigma^2}\exp\left(\frac{R(0,0)}{2c\sigma^2}\right)\int_{t_0}^T\left|H_n(F_{\mu_n}(t))-H(F_{\mu_n}(t))\right|\frac{f_{\tau^\circ_{\xi/\sigma}}(t-t_0)}{f_{\mu_n}(t)}dF_{\mu_n}(t)\right]\\
&\le \frac{L_G}{2\underline{c}\sigma^2}\exp\left(\frac{R(0,0)}{2\underline{c}\sigma^2}\right) \int_{t_0}^T\left|H_n(F_{\mu_n}(t))-H(F_{\mu_n}(t))\right|\frac{E f_{\tau^\circ_{\xi/\sigma}}(t-t_0)}{\underline{f}(t)}dF_{\mu_n}(t)\\
&\le \frac{L_G}{2\underline{c}\sigma^2(1-\pi)}\exp\left(\frac{R(0,0)}{2\underline{c}\sigma^2}+\frac{R(0,0)-R_\infty}{2\overline{c}\sigma^2}\right)\int_{t_0}^T \left|H_n(F_{\mu_n}(t))-H(F_{\mu_n}(t))\right|dF_{\mu_n}(t)\\
&=\frac{L_G}{2\underline{c}\sigma^2(1-\pi)}\exp\left(\frac{R(0,0)}{2\underline{c}\sigma^2}+\frac{R(0,0)-R_\infty}{2\overline{c}\sigma^2}\right)\int_{F_{\mu_n}(t_0)}^{F_{\mu_n}(T)} \left|H_n(r)-H(r)\right|dr\\
&\le \frac{L_G}{2\underline{c}\sigma^2(1-\pi)}\exp\left(\frac{R(0,0)}{2\underline{c}\sigma^2}+\frac{R(0,0)-R_\infty}{2\overline{c}\sigma^2}\right) \|H_n-H\|_{L^1[0,1]}
\end{align*}
This proves \eqref{eq:est1}. Similarly,
\begin{align*}
&E\int_{t_0}^T|u_n(t,0;\mu_n,c)-u(t,0;\mu_n,c)|f_{\tau^\circ_{\xi/\sigma}}(t-t_0)dt\\
&= E\int_{t_0}^T \left|\exp\left(\frac{(R_n)_{\mu_n}(t)}{2c\sigma^2}\right)-\exp\left(\frac{R_{\mu_n}(t)}{2c\sigma^2}\right)\right|f_{\tau^\circ_{\xi/\sigma}}(t-t_0)dt\\
&\le \frac{L_G}{2\underline{c}\sigma^2}\exp\left(\frac{R(0,0)}{2\underline{c}\sigma^2}\right)\int_{t_0}^T\left|H_n(F_{\mu_n}(t))-H(F_{\mu_n}(t))\right|\frac{E f_{\tau^\circ_{\xi/\sigma}}(t-t_0)}{f_{\mu_n}(t)}dF_{\mu_n}(t)\\
&\le \frac{L_G}{2\underline{c}\sigma^2(1-\pi)}\exp\left(\frac{R(0,0)}{2\underline{c}\sigma^2}+\frac{R(0,0)-R_\infty}{2\overline{c}\sigma^2}\right)\int_{F_{\mu_n}(t_0)}^{F_{\mu_n}(T)} \left|H_n(r)-H(r)\right|dr\\
&\le \frac{L_G}{2\underline{c}\sigma^2(1-\pi)}\exp\left(\frac{R(0,0)}{2\underline{c}\sigma^2}+\frac{R(0,0)-R_\infty}{2\overline{c}\sigma^2}\right) \|H_n-H\|_{L^1[0,1]},
\end{align*}
which verifies \eqref{eq:est2}. We are now ready to show \eqref{eq:est0}. For $t\le t_0$, $F_{\mu_n}(t)=F_{\tilde\mu_n}(t)$ trivially.
For $t\in(t_0,T]$, using \eqref{taucdf}, \eqref{u-bdd}, \eqref{eq:est1} and \eqref{eq:est2}, we deduce that
\begin{align*}
& \left|F_{\mu_n}(t)-F_{\tilde\mu_n}(t)\right|\\
&\le (1-\pi)E\left[\int_{t_0}^t \left|\frac{u_n(s,0;\mu_n,c)}{u_n(t_0,\xi;\mu_n,c)}-\frac{u(s,0;\mu_n,c)}{u(t_0,\xi;\mu_n,c)}\right|f_{\tau^\circ_{\xi/\sigma}}(s-t_0)ds\right]\\
&= (1-\pi)E\left[\int_{t_0}^t \left|\frac{u_n(s,0;\mu_n,c)u(t_0,\xi;\mu_n,c)-u(s,0;\mu_n,c)u_n(t_0,\xi;\mu_n,c)}{u_n(t_0,\xi;\mu_n,c)u(t_0,\xi;\mu_n,c)}\right|f_{\tau^\circ_{\xi/\sigma}}(s-t_0)ds\right]\\
&\le (1-\pi) \exp\left(\frac{R(0,0)}{2\underline{c}\sigma^2}-\frac{R_\infty}{\overline{c}\sigma^2}\right) E\left[\left|u(t_0,\xi;\mu_n,c)-u_n(t_0,\xi;\mu_n,c)\right|\int_{t_0}^t  f_{\tau^\circ_{\xi/\sigma}}(s-t_0)ds\right]\\
&\quad +(1-\pi)\exp\left(\frac{-R_\infty}{2\overline{c}\sigma^2}\right) E\left[\int_{t_0}^t \left|u_n(s,0;\mu_n,c)-u(s,0;\mu_n,c)\right|f_{\tau^\circ_{\xi/\sigma}}(s-t_0)ds\right]\\
&\le (1-\pi)\left\{C_1\exp\left(\frac{R(0,0)}{2\underline{c}\sigma^2}-\frac{R_\infty}{\overline{c}\sigma^2}\right)+C_2 \exp\left(\frac{-R_\infty}{2\overline{c}\sigma^2}\right)\right\} \|H_n-H\|_{L^1[0,1]}.
\end{align*}
When $T<\infty$, both $F_{\mu_n}$ and $F_{\tilde\mu_n}$ stay constant after time $T$ until time $\Delta$ where they jump to one. So the maximum of $|F_{\mu_n}-F_{\tilde\mu_n}|$ occurs on $[t_0,T]$ and an upper bound is given by \eqref{eq:est0}.

Finally, we have all the pieces to show $\mu$ is a fixed point of $\Phi^{t_0,\pi,m}$. Since $\tilde \mu_n$ converges weakly to $\Phi^{t_0,\pi,m}(\mu)$ and $|F_{\mu_n}-F_{\tilde \mu_n}|$ converges to zero uniformly, we know $\mu_n$ also converges to $\Phi^{t_0,\pi,m}(\mu)$ weakly. By uniqueness of the weak limit, we must have $\Phi^{t_0,\pi,m}(\mu)=\mu$.

For $T=\infty$, we  use the compactification of $\mathcal{P}([t_0,\infty])$ as in the proof of Theorem \ref{thm:existence}.
\end{proof}

\subsection{Approximate equilibrium of the $N$-player game }
We want to see when we can approximate a game with finitely many players with the  game with infinitely many players.
Our results will be derived for the reward functions that are H\"older continuous in the rank variable. For $\alpha\in(0,1]$, define
\[\mathcal{R}_{C^\alpha}:=\{R\in\mathcal{R}: \exists L_\alpha \in\bbR \text{ s.t.\ } |R(t,r_1)-R(t,r_2)|\leq L_\alpha |r_1-r_2|^\alpha \ \forall t\in\bbR_+, r_1, r_2\in \bbR\}\subseteq\mathcal{R}_D.\]
For simplicity of the presentation, we consider only a freshly started tournament, that is, $t_0=0$, and $\pi=0$.

The $N$-player system is given by
\[X^{a_i}_{i,t}=x_{i}-\int_0^t a_{i,s} ds+\sigma B_{i,t}, \quad i=1,\ldots, N,\]
where $B_1, \ldots, B_N$ are independent Brownian motions, and $\mathbf{a}=(a_{1}, \ldots, a_N)$ is a vector of admissible actions, meaning that each $a_i$ is non-negative and progressively measurable with respect to the filtration of $B_1, \ldots, B_N$, and yields a unique strong solution $X^{a_i}_i$ up to $\tau^{a_i}_i:=\inf\{t\geq 0: X^{a_i}_{i,t}=0\}$, and if $\tau^{a_i}_i$ is non-atomic.

Let $\bar\mu^{N,\mathbf{a}}=\frac{1}{N}\sum_{i=1}^N \delta_{\tau_i^{a_i}}$ be the empirical distribution of the completion time.
\begin{defn}
An admissible action vector $\mathbf{a}=(a_{1}, \ldots, a_N)$ is called an $\epsilon$-Nash equilibrium action of the $N$-player game if for any $i\in\{1, \ldots, N\}$ and any admissible control $\beta$, we have
\[{E}\left[R_{\bar{\mu}^{N,\mathbf{a}}}(\tau^{a_i}_{i})-\int_0^{\tau^{a_i}_i\wedge T} c_i a_{i,t}^2dt\right]+\epsilon\ge {E}\left[R_{\bar{\mu}^{N,\mathbf{a}^i_\beta}}(\tau^{\beta}_{i})-\int_0^{\tau^\beta_i\wedge T} c_i \beta_t^2dt\right]\]
where
$\mathbf{a}^{i}_\beta=(a_1, \ldots, a_{i-1}, \beta, a_{i+1}, \ldots, a_N)$.
\end{defn}

We have the following approximation result.

\begin{thm}\label{thm:epsNE}
Assume $R\in\mathcal{R}_{C^\alpha}$ for some $\alpha\in(0,1]$ and $(x_i, c_i)$, $i=1, \ldots, N$ are i.i.d.\ samples from $m$.
For any fixed point $\mu$ of $\Phi^{0,0,m}$,  the vector $\mathbf{\bar a}=(\bar a_1, \ldots, \bar a_N)$ with
\[\bar{a}_{i,t}:=-(2c_i)^{-1}v_x(t,X^{\bar{a}_i}_{i,t};\mu,c_i)1_{\{t< \tau^{\bar a_i}_i\}}, \ i=1,\ldots, N\]
 is an $O(N^{-\alpha/2})$-Nash equilibrium action  of the $N$-player game, as $N\rightarrow \infty$.
\end{thm}
\begin{proof}
Let $\mu$ and $\bar{a}_{i,t}$ be defined as in the theorem statement. To simplify notation, we omit the superscript of any state process $X_i$ and its first passage time $\tau_i$ if it is controlled by the optimal Markovian feedback strategy $-(2c_i)^{-1}v_x(t,x;\mu,c_i)$. Fix an initial position $x$ and cost  $c$. Let
\[V(x,c):=v(0,x;\mu,c)={E}\left[R_{\mu}(\tau^{x,c})-\int_0^{\tau^{x,c}\wedge T} \frac{1}{4c}v_x^2(t,X^{x,c}_t;\mu,c)dt\right]\]
be the value of the  game with infinitely many players, where
\[d X^{x,c}_t=-\frac{v_x(t,X^{x,c}_t;\mu,c)}{2c}1_{\{t< \tau^{x,c}\}}dt+\sigma dB_t, \quad X^{x,c}_0=x\]
and $\tau^{x,c}$ is the first passage time of $X^{x,c}$ to level zero. Since $X_i$ is an  identical copy of $X^{x_i, c_i}$, we have
\begin{align*}
V(x_i,c_i)&={E}\left[R_{\mu}(\tau_i)-\int_0^{\tau_i\wedge T} \frac{1}{4c_i}v_x^2(t,X_{i,t};\mu,c)dt\right]={E}\left[R_{\mu}(\tau_i)-\int_0^{\tau_i\wedge T} c_i \bar{a}^2_{i,t} dt\right].
\end{align*}
Let
\[\bar{\mu}^{N}:=\frac{1}{N}\sum_{i=1}^N \delta_{\tau_{i}}\]
be the empirical distribution of the completion time, and
\begin{align*}
J^{N}_i:&=E\left[R_{\bar{\mu}^N}(\tau_{i})-\int_0^{\tau_i\wedge T} c_i\bar{a}^2_{i,t}dt\right]
\end{align*}
be the net gain of player $i$ in an $N$-player game when everybody uses the candidate approximate Nash equilibrium $\mathbf{\bar a}$. We first show that $J^N_i$ and $V(x_i,c_i)$ are close. By the $\alpha$-H\"older continuity of $R$,
we have
\begin{align*}
V(x_i,c_i)-J^N_i= {E}\left[(R_{\mu}(\tau_i)-R_{\bar{\mu}^N}(\tau_{i}))\right]\le L_\alpha\left[\|F_{\bar \mu^N}-F_{\mu^\eps}\|^\alpha_\infty\right]
\end{align*}

Let $q(\cdot \, ;x,c)$ be the distribution of $\tau^{x,c}$. Since $\tau_i$ has $q(\cdot \, ;x_i,c_i)$ as its distribution and   $(x_i, c_i)$ has $m$ as its distribution, we can treat $\bar\mu^N$ as the empirical distribution of i.i.d.\ samples from $q\otimes m=\mu$.\footnote{Being a fixed point of $\Phi^{0,0,m}$, $\mu$ is supported on $[0,\Delta]$. Here by a slight abuse of notation, we identify $\mu$ with $q\otimes m\in\mathcal{P}[0,\infty)$ without changing the action $\mathbf{\bar a}$.} By Dvoretzky-Kiefer-Wolfowitz inequality, we have
\[P\left(\|F_{\bar \mu^N}-F_{\mu}\|_\infty>z\right) \leq 2e^{-2Nz^2} \quad \forall\, z>0.\]
It follows that
\begin{align*}
V(x_i, c_i)-J^N_i&\leq L_\alpha{E}[\|F_{\bar \mu^N}-F_\mu\|^\alpha_\infty]=L_\alpha\int_0^\infty{P}\left(\|F_{\bar \mu^N}-F_{\mu}\|^\alpha_\infty>z\right) dz \\
&\leq L_\alpha\int_0^\infty 2e^{-2Nz^{2/\alpha}} dz= \frac{2L_\alpha}{N^{\alpha/2}}\int_0^\infty e^{-2y^{2/\alpha}} dy=O(N^{-\alpha/2})~\text{ as } N\rightarrow \infty.
\end{align*}

Next, consider the system where player $i$ makes a unilateral deviation from the candidate approximate Nash equilibrium $\mathbf{\bar a}$; say, he chooses an admissible control $\beta$. Denote his controlled state process by $X^\beta_i$, and the state processes of all other players by $X_j$ as before for $j\neq i$. Let
\[\bar{\nu}^N:=\frac{1}{N}(\delta_{\tau^\beta_{i}}+\sum_{j\neq i} \delta_{\tau_{j}})\]
be the corresponding empirical measure of the completion times, and
\[J^{N,\beta}_i:=E\left[R_{\bar{\nu}^N}(\tau^\beta_{i})-\int_0^{\tau^\beta_i\wedge T} c_i\beta^2_{t}dt\right]\]
be the corresponding net gain for player $i$.
We have
\begin{align*}
&J^{N,\beta}_i-V(x_i, c_i)\\
&=E\left[R_{\bar{\nu}^N}(\tau^\beta_{i})-\int_0^{\tau^\beta_i\wedge T} c_i\beta^2_{t}dt\right]-{E}\left[ R_{\mu}(\tau_i)-\int_0^{\tau_i\wedge T} c_i\bar{a}^2_{i,t} dt\right]\\
&=E\left[R_{\bar{\nu}^N}(\tau^\beta_{i})-R_{\mu}(\tau^\beta_i)\right]+E\left[R_{\mu}(\tau^\beta_i)-\int_0^{\tau^\beta_i\wedge T} c_i\beta^2_{t}dt\right]-E\left[R_{\mu}(\tau_i)-\int_0^{\tau_i\wedge T} c_i\bar{a}^2_{i,t} dt\right]\\
&\leq E\left[R_{\bar{\nu}^N}(\tau^\beta_{i})-R_{\mu}(\tau^\beta_i)\right],
\end{align*}
where we have used the optimality of $\bar{a}_i$ for the $i$-th player's problem (in response to $\mu$) in the last step.\footnote{By Remark~\ref{rmk:filtration}, $\bar a_i$ is optimal in the filtration of $B_1, \ldots, B_N$.} Let
\[\bar \nu^{N}_{-i}:=\frac{1}{N-1}\sum_{j\neq i} \delta_{\tau_j}.\]
$\bar \nu^{N}_{-i}$ is the empirical distribution of $(N-1)$ i.i.d.\ samples from the distribution $\mu$.
Similar to how we estimated $V(x_i, c_i)-J^N_i$, we have
\begin{align*}
J^{N,\beta}_i-V(x_i, c_i)&\leq L_\alpha{E}[|F_{\bar{\nu}^N}(\tau^\beta_{i})-F_{\mu}(\tau^\beta_i)|^\alpha]\\
&=L_\alpha E\left[\left|\frac{1}{N}\left(1-F_{\mu}(\tau^\beta_{i})\right)+\frac{N-1}{N}\left(F_{\bar \nu^{N}_{-i}}(\tau^\beta_{i})-F_{\mu}(\tau^\beta_{i})\right)\right|^\alpha\right] \\
&\leq L_\alpha E\left[\left(\frac{1}{N}+\frac{N-1}{N}\|F_{\bar \nu^{N}_{-i}}-F_{\mu}\|_\infty\right)^\alpha\right]\\
&\leq L_\alpha\left(\frac{1}{N}+\frac{N-1}{N}E\|F_{\bar \nu^{N}_{-i}}-F_{\mu}\|_\infty\right)^\alpha=O(N^{-\alpha/2})~\text{ as } N\rightarrow \infty.
\end{align*}
Combining the two estimates, we obtain
\[J^{N,\beta}_i-J^N_i= J^{N,\beta}_i-V(x_i, c_i)+(V(x_i, c_i)-J^N_i)\le O(N^{-\alpha/2})~\text{ as } N\rightarrow \infty.\]
This shows $\mathbf{\bar a}$ is an $O(N^{-\alpha/2})$-Nash equilibrium action.
\end{proof}

\begin{remark}
For reward functions that are merely continuous in the rank variable, we still have convergence of the approximation, but we do not know the convergence rate.
\end{remark}

\section{A case with heterogeneous players}\label{sec:het}

We end our discussion with a numerically computed example with heterogeneous agents, assuming $R\in\mathcal{R}$ is of the form \eqref{eq:PRBR}. The numerical method is decribed in Appendix.
We consider players that differ in the initial location and cost, with  $x_0\in\{1,2\}$ and $c\in\{1,4\}$. Other model inputs are fixed to be $T=1$, $\sigma=0.25$ and $R(t,r)=1_{\{t\le T\}}15(1-r)^2$. Denote by $\beta$  the population terminal completion rate and by $\beta_{\mathbf{x}}$  the terminal completion rate of the players of type  $\mathbf{x}$ relative to its initial weight. Similarly, $V_{\mathbf{x}}$ refers to the game value for a player of type $\mathbf{x}$. We define the total welfare as
\[\text{Total welfare}=\int V(x_0,c) dm(x_0,c).\]
Let us  call   ``disadvantaged" (DA) the players  who start from a larger distance, or who have a higher cost of effort.
Analogously for ``advantaged" (AD) players.

\begin{table}[h]
\centering
\setlength\extrarowheight{3pt}
\scalebox{0.9}{
\begin{tabular}{|c|c|c|c|c|c|c|c|c|c|c|}\hline
Case & $m$ &\makebox[3.6em]{$\beta$}&\makebox[3.6em]{$\beta_{\text{AD}}$}&\makebox[3.6em]{$\beta_{\text{DA}}$}
&\makebox[3.6em]{$V_{\text{AD}}$}&\makebox[3.6em]{$V_{\text{DA}}$} & Total welfare \\\hline\hline
0&$\delta_{(1,1)}$ & 75.9\% & 75.9\% & - & 0.178 & - & 0.178 \\\hline\hline
1&$\frac{4}{5}\delta_{(1,1)}+\frac{1}{5}\delta_{(2,1)}$ & 73.8\% & 92.2\% & 0\% & 0.319 & 0 & 0.255 \\\hline
2&$\frac{3}{5}\delta_{(1,1)}+\frac{2}{5}\delta_{(2,1)}$ & 60\% & 100\% & 0\% & 1.701  & 0 & 1.020  \\\hline
3&$\frac{2}{5}\delta_{(1,1)}+\frac{3}{5}\delta_{(2,1)}$ & 49.8\%  & 100\%  & 16.4\%  & 4.276  & 0.022 & 1.724 \\\hline
4&$\frac{1}{5}\delta_{(1,1)}+\frac{4}{5}\delta_{(2,1)}$ & 49.8\%  & 100\%  & 37.3\%  & 7.338  & 0.058 & 1.514 \\\hline
5&$\delta_{(2,1)}$ & 49.8\% & - & 49.8\% &  - &  0.086 & 0.086 \\ \hline \hline
6&$\frac{4}{5}\delta_{(1,1)}+\frac{1}{5}\delta_{(1,4)}$ & 73.8\% & 92.2\% & 0.1\% & 0.319 & 0 & 0.255 \\\hline
7&$\frac{3}{5}\delta_{(1,1)}+\frac{2}{5}\delta_{(1,4)}$ & 60.4\% & 100\% & 0.9\% & 1.675  & 0.005 & 1.007 \\\hline
8&$\frac{2}{5}\delta_{(1,1)}+\frac{3}{5}\delta_{(1,4)}$ & 51.9\%  & 100\% & 19.8\% & 4.030  & 0.110  &  1.678  \\\hline
9&$\frac{1}{5}\delta_{(1,1)}+\frac{4}{5}\delta_{(1,4)}$ & 51.8\%  & 100\%  & 39.8\%  & 7.091 & 0.253 & 1.621 \\\hline
10&$\delta_{(1,4)}$ & 51.8\% & - & 51.8\% &  - & 0.365 & 0.365 \\\hline
\end{tabular}}
\medskip
\caption{Equilibrium completion rates and game values under different population composition.}
\label{tab:het}
\end{table}

Table~\ref{tab:het} shows that for those parameters, the following  happens:

\begin{itemize}
\item
  As the fraction of  DA  players in the population increases, the population completion rate decreases, but the completion rates of
the  DA group and of the AD group both increase. That is, the aggregate percentage of completions is lower because more of the population
is disadvantaged, and not because of working less hard. In fact, both the value of DA players and the value of AD players go up with the population fraction of DA players.
This is because DA players find it more advantageous to put in higher effort when the
fraction of AD players is lower,  due to weaker competition, so that  a higher percentage of DA players completes the task and enjoys a higher value.

 \item The aggregate welfare is not monotone in the percentage of DA players: it goes up with the fraction of DA players only  in the lower and mid  range of the DA fraction, unlike the groups' welfare that always go up. When the fraction of DA players becomes very high, the  total population welfare  starts decreasing.
This is the  result of two conflicting forces: quality (higher individual value) versus quantity (fewer AD individuals).

\end{itemize}

These results support two (not very surprising) empirical predictions: (i)
in low-growth industries, or in non-profit institutions we should not see tournament-based compensation, especially
when they have a high percentage of high-skilled employees;
   (ii) in undeveloped countries with large differences in access to good education, those with less access would give up early.


To recap, a more efficient society (more AD workers) has  a higher productivity (larger $\beta$) and may have a higher total welfare, but it  makes, in this example,  the individuals worse off,  because the efficient workers tend to  work too hard and the inefficient workers tend to work too little.

Figure~\ref{fig:het3} shows the equilibrium  effort corresponding to Cases 6 and  9 in Table~\ref{tab:het}, varying the fraction of low cost vs high cost players.
(Other cases lead to similar observations.)
Comparing the effort level between the left two panels and the right two panels, we see that the low-cost players exert higher effort than the high-cost players. Their effort region is also larger. On the other hand, players also tend to play to their opponents' level. A comparison of the top left panel with the bottom left panel demonstrates that the effort peak of AD players is significantly lower when they face weaker competition. Similarly, DA players tend to raise their effort when they face stronger competition until the competition gets too strong for them to keep up.

\begin{figure}[t]
\centering
\includegraphics[height=5.5cm]{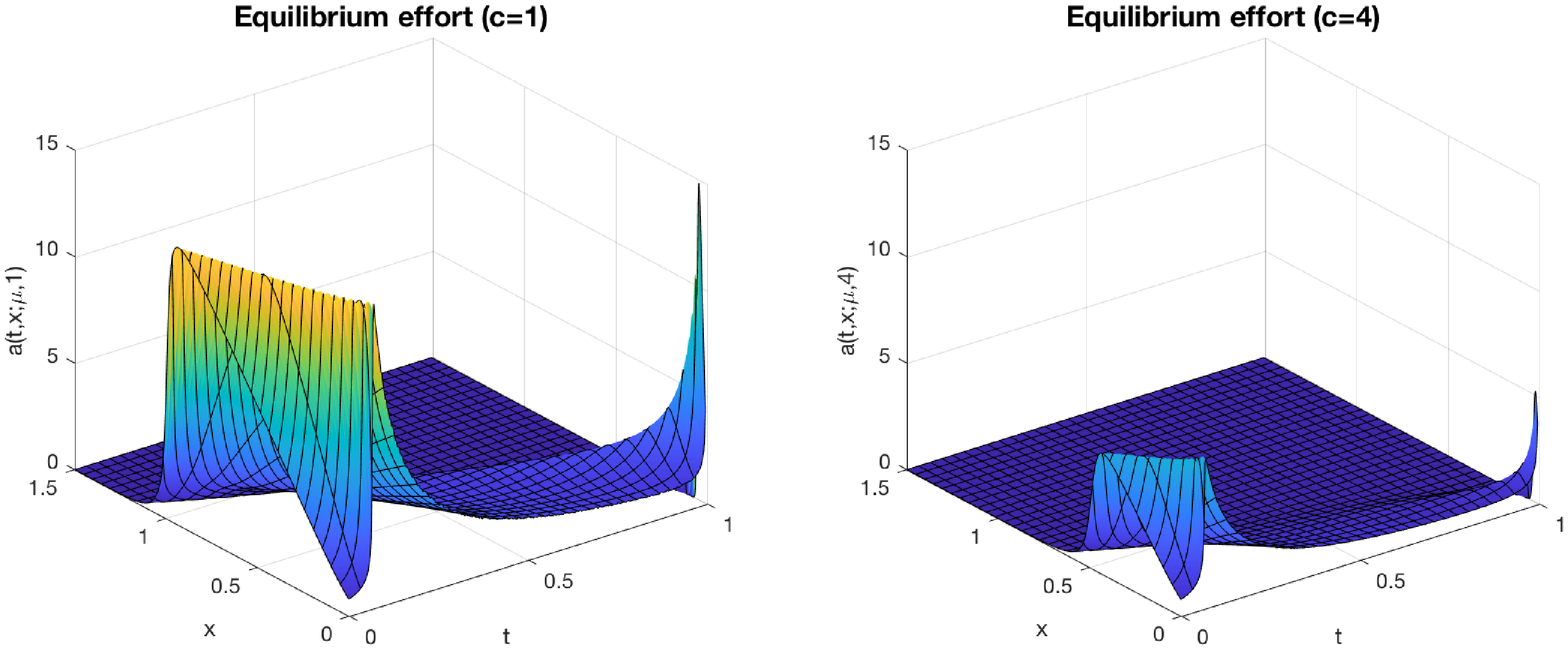}
\includegraphics[height=5.5cm]{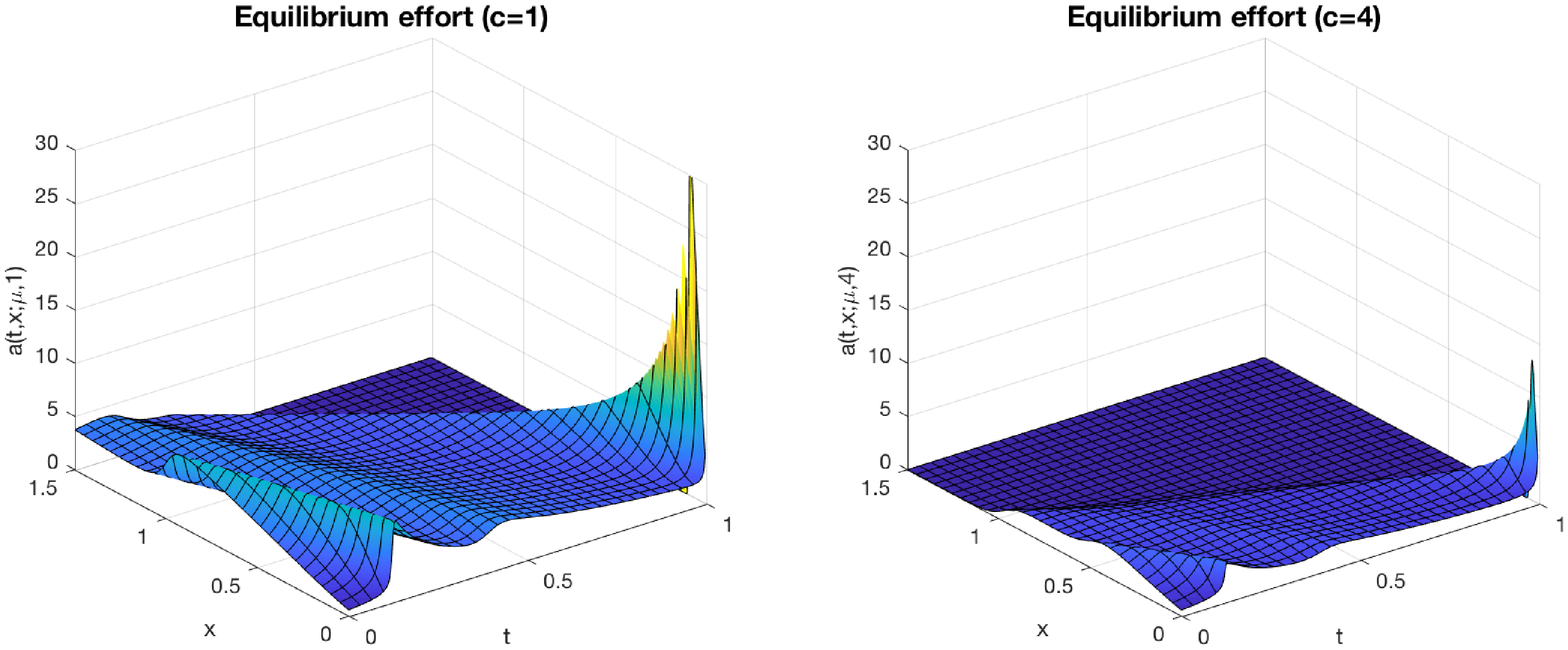}
\caption{Effort function in an equilibrium formed by 80\% low-cost players and 20\% high-cost players (top two panels) and 20\% low-cost players and 80\% high-cost players (bottom two panels). Everyone starts at $x_0=1$.}
\label{fig:het3}
\end{figure}

\section{Appendix}\label{app:A}

\subsection{Numerical method}\label{subsec:numm}

 When the players are heterogeneous, that is, when $m$ is not  degenerate, there is generally no explicit characterization of the equilibrium.
  Instead, we sketch in this section how to numerically solve the fixed point equation
\begin{equation}
F_\mu(t)=E\left[\int_0^t \frac{u(z,0;\mu,c)}{u(0,\xi;\mu,c)}f_{\tau^\circ_{\xi/\sigma}}(z)dz\right], \ t\in[0,T]
\end{equation}
 by discretizing the rank space, which is equivalent to discretizing the purely rank-based reward function.\footnote{There is an advantage of discretizing the rank variable instead of time variable: when $T$ is large, a  fine discretization of $[0,T]$ may result in a high computational burden. In contrast, the rank space $[0,1]$ remains fixed.} From the stability result (Theorem~\ref{thm:stability}), we know that the sequence of Nash equilibria of the discretized games, if converging, will converge weakly to a Nash equilibrium of the original game. Thus, we focus on piecewise constant reward functions:
\[R(t,r)=1_{\{t\le T\}}\left[\sum_{k=1}^{d} R_k 1_{[r_{k-1},r_k)}(r)+R_{d+1} 1_{[r_d,1]}\right]+1_{\{t>T\}}R_\infty \]
where $0=r_0<r_1<\cdots<r_{d}<1$ is a finite partition of the rank space, and $R_1\ge R_2\ge \cdots R_{d+1}\ge R_\infty$. Since such a reward function lies in $\mathcal{R}_D\cap \mathcal{R}_S$, we already know from Section~\ref{subsec:existence} that a Nash equilibrium exists, and that each equilibrium distribution $\mu$ can be represented by a vector $(T^\mu_1, \ldots, T^\mu_d)$, where $T^\mu_k$ is the $(r_k)$-quantile of $\mu$. As before, we set $T^
\mu_0:=0$, $T^\mu_{d+1}:=\Delta$.

Using the piecewise constant structure of $R_\mu(t)$,  for a freshly started game with $(t_0,\pi)=(0,0)$, the  fixed point equation for the c.d.f of $\mu$ becomes
\begin{equation}\label{FPEstepfun}
F_{\mu}(t)=E\left[ \frac{\sum_{k=1}^{d+1} \exp\left(\frac{R_k}{2c\sigma^2}\right) \left(F_{\tau^\circ_{\xi/\sigma}}(T^\mu_k\wedge t)-F_{\tau^\circ_{\xi/\sigma}}(T^\mu_{k-1}\wedge t)\right)}{u(0,\xi;\mu,c)}\right], \quad t\le T,
\end{equation}
where
\[u(0,\xi;\mu,c)=\sum_{k=1}^{d+1} \exp\left(\frac{R_k}{2c\sigma^2}\right)\left(F_{\tau^\circ_{\xi/\sigma}}(T^\mu_k\wedge T)-F_{\tau^\circ_{\xi/\sigma}}(T^\mu_{k-1}\wedge T)\right)+\exp\left(\frac{R_\infty}{2c\sigma^2}\right)\left(1-F_{\tau^\circ_{\xi/\sigma}}(T)\right).\]
In particular, for each $T^\mu_j\le T$, we have
\begin{align*}
r_j&=F_\mu(T^\mu_j)=E\left[ \frac{\sum_{k=1}^{j} \exp\left(\frac{R_k}{2c\sigma^2}\right) \left(F_{\tau^\circ_{\xi/\sigma}}(T^\mu_k)-F_{\tau^\circ_{\xi/\sigma}}(T^\mu_{k-1})\right)}{u(0,\xi;\mu,c)}\right].
\end{align*}

To describe the numerical algorithm, we introduce the quantity\footnote{Note that $A_T$ does not depend on $\mu$, thus can be computed a priori.}
\begin{equation}\label{eq:A_T}
A_T:=E\left[ \frac{ \exp\left(\frac{R_1}{2c\sigma^2}\right)F_{\tau^\circ_{\xi/\sigma}}(T)}{\exp\left(\frac{R_1}{2c\sigma^2}\right)F_{\tau^\circ_{\xi/\sigma}}(T)+\exp\left(\frac{R_\infty}{2c\sigma^2}\right)\left(1-F_{\tau^\circ_{\xi/\sigma}}(T)\right)}\right].
\end{equation}

 We consider two cases.

Case 1. If $r_1>A_T$, then we must have $T^\mu_1>T$, otherwise we get a contradiction:
\begin{align*}
r_1>A_T&\ge E\left[ \frac{ \exp\left(\frac{R_1}{2c\sigma^2}\right)F_{\tau^\circ_{\xi/\sigma}}(T^\mu_1)}{\exp\left(\frac{R_1}{2c\sigma^2}\right)F_{\tau^\circ_{\xi/\sigma}}(T^\mu_1)+\exp\left(\frac{R_\infty}{2c\sigma^2}\right)\left(1-F_{\tau^\circ_{\xi/\sigma}}(T)\right)}\right]\\
&\ge E\left[ \frac{ \exp\left(\frac{R_1}{2c\sigma^2}\right)F_{\tau^\circ_{\xi/\sigma}}(T^\mu_1)}{u(0,\xi;\mu,c)}\right]=r_1.
\end{align*}
In this case, $(\Delta, \ldots, \Delta)$ is the unique equilibrium quantile, that is, no one finishes the project.

Case 2. If $r_1\le A_T$, then we must have $T^\mu_1\le T$, since otherwise $T^\mu_k=\Delta$ for all $k=1, \ldots, d+1$ and $F_\mu(T)=A_T\ge r_1$, yielding a contradiction: $T\ge T^\mu_1=\Delta$. In this case, let $k_0=\max\{k: T^\mu_k\le T\}\in\{1,\ldots, d\}$. Then, we write $u$ as
\begin{equation}\label{FPEu}
\begin{aligned}
u(0,\xi;T^\mu_{[1:k_0]},c)&=\sum_{k=1}^{k_0} \exp\left(\frac{R_k}{2c\sigma^2}\right)\left(F_{\tau^\circ_{\xi/\sigma}}(T^\mu_k)-F_{\tau^\circ_{\xi/\sigma}}(T^\mu_{k-1})\right)\\
&\quad +\exp\left(\frac{R_{k_0+1}}{2c\sigma^2}\right)\left(F_{\tau^\circ_{\xi/\sigma}}(T)-F_{\tau^\circ_{\xi/\sigma}}(T^\mu_{k_0})\right)+\exp\left(\frac{R_\infty}{2c\sigma^2}\right)\left(1-F_{\tau^\circ_{\xi/\sigma}}(T)\right)\\
&=\sum_{k=1}^{k_0}\left(\exp\left(\frac{R_k}{2c\sigma^2}\right) -\exp\left(\frac{R_{k+1}}{2c\sigma^2}\right) \right)F_{\tau^\circ_{\xi/\sigma}}(T^\mu_k)\\
&\quad+\exp\left(\frac{R_{k_0+1}}{2c\sigma^2}\right)F_{\tau^\circ_{\xi/\sigma}}(T)+\exp\left(\frac{R_\infty}{2c\sigma^2}\right)\left(1-F_{\tau^\circ_{\xi/\sigma}}(T)\right).
\end{aligned}
\end{equation}
Here,  we have replaced $\mu$ by $T^\mu_{[1:k_0]}$ to indicate that $u$ depends on $\mu$ only through the first $k_0$ quantiles. The fixed point equations can be rewritten then as
\begin{equation}\label{FPE}
r_{k}-r_{k-1}=E\left[ \frac{\exp\left(\frac{R_{k}}{2c\sigma^2}\right) \left(F_{\tau^\circ_{\xi/\sigma}}(T^\mu_{k})-F_{\tau^\circ_{\xi/\sigma}}(T^\mu_{k-1})\right)}{u(0,\xi;T^\mu_{[1:k_0]},c)}\right], \quad k=1, \ldots, k_0.
\end{equation}
This is a system of nonlinear equations in $T^\mu_1, \ldots, T^\mu_{k_0}$, where $k_0$ is determined by the condition that $T^\mu_{k_0}\le T<T^\mu_{k_0+1}$, or equivalently, $0\le F_\mu(T)-F_{\mu}(T^\mu_{k_0})<r_{k_0+1}-r_{k_0}$ (where we set $r_{d+1}:=1$). By \eqref{FPEstepfun}, this is equivalent to
\begin{equation}\label{eq:k0}
0\le E\left[\frac{ \exp\left(\frac{R_{k_0+1}}{2c\sigma^2}\right) \left(F_{\tau^\circ_{\xi/\sigma}}(T)-F_{\tau^\circ_{\xi/\sigma}}(T^\mu_{k_0})\right)}{u(0,\xi;T^\mu_{[1:k_0]},c)}\right]<r_{k_0+1}-r_{k_0}.
\end{equation}

The procedure can be summarized as follows:
\begin{tcolorbox}
-- If $T<\infty$, compute $A_T$ defined by \eqref{eq:A_T}.
\begin{itemize}
\item If $A_T<r_1$, then the unique equilibrium quantile vector is given by $(\Delta, \ldots, \Delta)$.
\item If $A_T\ge r_1$, solve the nonlinear system \eqref{FPE} for $(T^\mu_{1}, \ldots, T^\mu_{k_0})$, where $u$ is given by \eqref{FPEu}, and   $k_0$ is determined by \eqref{eq:k0}. The  equilibrium quantile vector is given by $(T^\mu_{1}, \ldots, T^\mu_{k_0}, \Delta, \ldots, \Delta)$.
\end{itemize}
-- If $T=\infty$, solve \eqref{FPEu} and \eqref{FPE} with $k_0=d$.
\end{tcolorbox}

\subsection{Proof of Proposition \ref{prop:homcs}}

(i) Let $\psi(y):=\phi^{-1}(y/(1-y))$ where $\phi$ is defined in \eqref{eq:phi}. We have $\beta=\psi(F_{\tau^\circ_{x_0/\sigma}}(T))$. Since $y\mapsto \psi(y)$ is strictly increasing, the strictly monotonicity of $\beta$ in $T$ and $x_0$ follows from that of $F_{\tau^\circ_{x_0/\sigma}}(T)$. Now, fix $T$ and $x_0$. Suppose $c_1\le c_2$, then $\phi(r;c_1)\le \phi(r,c_2)$ for all $r\in[0,1]$, and thus
\[\phi(\beta(c_1);c_1)=\frac{F_{\tau^\circ_{x_0/\sigma}}(T)}{1-F_{\tau^\circ_{x_0/\sigma}}(T)}=\phi(\beta(c_2);c_2)\ge \phi(\beta(c_2);c_1).\]
Since $\phi(\cdot;c_1)$ is increasing, we get $\beta(c_1)\ge \beta(c_2)$.
All the limits follow from straightforward computations using \eqref{eq:BMFPTcdf} and \eqref{eq:NErate}.

(ii) Recall \eqref{eq:NErate}:
\begin{equation}\label{eq:beta}
\frac{(1-\beta) F_{\tau^\circ_{x_0/\sigma}}(T)}{1-F_{\tau^\circ_{x_0/\sigma}}(T)}=\int_0^{\beta} \exp\left(\frac{R_\infty-H(z)}{2c\sigma^2}\right)dz.
\end{equation}
Let $0<T_1\le T_2<\infty$. Since $\beta(T_1)\le \beta(T_2)$ and $H\ge R_\infty$, we have
\begin{align*}
\beta(T_2)-\beta(T_1)&\ge \int_{\beta(T_1)}^{\beta(T_2)}\exp\left(\frac{R_\infty-H(z)}{2c\sigma^2}\right)dz\\
&=\frac{(1-\beta(T_2)) F_{\tau^\circ_{x_0/\sigma}}(T_2)}{1-F_{\tau^\circ_{x_0/\sigma}}(T_2)}-\frac{(1-\beta(T_1)) F_{\tau^\circ_{x_0/\sigma}}(T_1)}{1-F_{\tau^\circ_{x_0/\sigma}}(T_1)}\\
&=\frac{\left(F_{\tau^\circ_{x_0/\sigma}}(T_2)-F_{\tau^\circ_{x_0/\sigma}}(T_1)\right)(1-\beta(T_2))}{\left(1-F_{\tau^\circ_{x_0/\sigma}}(T_2)\right)\left(1-F_{\tau^\circ_{x_0/\sigma}}(T_1)\right)}+(\beta(T_1)-\beta(T_2))\frac{F_{\tau^\circ_{x_0/\sigma}}(T_1)}{1-F_{\tau^\circ_{x_0/\sigma}}(T_1)}.
\end{align*}
This implies
\begin{equation*}\label{eq:dbeta_dT}
\beta(T_2)-\beta(T_1)\ge \frac{\left(F_{\tau^\circ_{x_0/\sigma}}(T_2)-F_{\tau^\circ_{x_0/\sigma}}(T_1)\right)(1-\beta(T_2))}{\left(1-F_{\tau^\circ_{x_0/\sigma}}(T_2)\right)}.
\end{equation*}
Let $\zeta:=(1-F_{\tau^\circ_{x_0/\sigma}}(T))/(1-\beta)$. We then have
\begin{align*}
\zeta(T_2)-\zeta(T_1)=\frac{\left(1-F_{\tau^\circ_{x_0/\sigma}}(T_2)\right)(\beta(T_2)-\beta(T_1))-(1-\beta(T_2))\left(F_{\tau^\circ_{x_0/\sigma}}(T_2)-F_{\tau^\circ_{x_0/\sigma}}(T_1)\right)}{(1-\beta(T_2))(1-\beta(T_1))}\ge 0.
\end{align*}
It follows that $\zeta$, and hence $V=R_\infty+2c\sigma^2 \ln \zeta$, is increasing in $T$. As $T\rightarrow 0$, we have $\beta \rightarrow 0$, $\zeta\rightarrow 1$ and $V\rightarrow R_\infty$. As $T\rightarrow \infty$, we have $\beta \rightarrow 1$, $\zeta\rightarrow \left(\int_0^1 \exp\left(\frac{R_\infty-H(z)}{2c\sigma^2}\right)\right)^{-1}$ (by \eqref{eq:beta}) and $V\rightarrow V_\infty$.

Similarly, let $x_1\le x_2$, \eqref{eq:beta}, together with $\beta(x_1)\ge \beta(x_2)$, implies
\begin{equation*}\label{eq:dbeta_dx0}
\beta(x_{1})-\beta(x_{2})\ge  \frac{\left(F_{\tau^\circ_{x_1/\sigma}}(T)-F_{\tau^\circ_{x_2/\sigma}}(T)\right)(1-\beta(x_2))}{\left(1-F_{\tau^\circ_{x_2/\sigma}}(T)\right)},
\end{equation*}
and
\[\zeta(x_1)-\zeta(x_2)=\frac{\left(1-F_{\tau^\circ_{x_2/\sigma}}(T)\right)(\beta(x_1)-\beta(x_2))-(1-\beta(x_2))\left(F_{\tau^\circ_{x_1/\sigma}}(T)-F_{\tau^\circ_{x_2/\sigma}}(T)\right)}{(1-\beta(x_2))(1-\beta(x_1))}\ge 0.\]
 Thus, $\zeta$ and $V$ are decreasing in $x_0$. As $x_0\rightarrow 0$, we have $\beta \rightarrow 1$, $\zeta\rightarrow \left(\int_0^1 \exp\left(\frac{R_\infty-H(z)}{2c\sigma^2}\right)\right)^{-1}$ and $V\rightarrow V_\infty$. As $x_0\rightarrow \infty$, we have $\beta \rightarrow 0$, $\zeta\rightarrow 1$ and $V\rightarrow R_\infty$.

(iii) Independence of $x_0$ is direct from \eqref{eq:infNEvalue}. Differentiating \eqref{eq:infNEvalue} with respect to $c$, we get
\begin{align*}
\frac{\partial V_\infty}{\partial c}&=\frac{-2\sigma^2 \int_0^1 \exp\left(-\frac{H(z)}{2c\sigma^2}\right)dz \ln \left(\int_0^1 \exp\left(-\frac{H(z)}{2c\sigma^2}\right)dz\right)-2c\sigma^2\int_0^1 \exp\left(-\frac{H(z)}{2c\sigma^2}\right) \frac{H(z)}{2c^2\sigma^2}dz}{\int_0^1 \exp\left(-\frac{H(z)}{2c\sigma^2}\right)dz}\\
&=\frac{-2\sigma^2 \int_0^1 \exp\left(-\frac{H(z)}{2c\sigma^2}\right) \left\{\ln\left(\int_0^1 \exp\left(-\frac{H(r)}{2c\sigma^2}\right)dr\right)+ \frac{H(z)}{2c\sigma^2} \right\} dz}{\int_0^1 \exp\left(-\frac{H(z)}{2c\sigma^2}\right)dz}\\
&=\frac{-2\sigma^2 \int_0^1 h(z) \ln \left(\frac{\int_0^1 h(r)dr}{h(z)}\right) dz}{\int_0^1 h(z) dz}, \quad \text{where}\quad h(z)=\exp\left(-\frac{H(z)}{2c\sigma^2}\right).
\end{align*}
Using $\ln x\le x-1$ for all $x>0$, we have
\[\int_0^1 h(z) \ln \left(\frac{\int_0^1 h(r)dr}{h(z)}\right) dz\le \int_0^1 h(z) \left(\frac{\int_0^1 h(r)dr}{h(z)}-1\right) dz=0.\]
It follows that $\partial V_\infty/\partial c\ge 0$. To compute the limit in $c$, we first observe that
\begin{equation*}
H(1-) \le V_\infty \le E H(F_\mu(\tau))=\int_0^\infty H(F_\mu(t)) d\mu(t)=\int_0^1 H(r)dr.
\end{equation*}
So $\lim_{c\rightarrow 0} V_\infty$ and $\lim_{c\rightarrow \infty}V_\infty$ exist. To compute $\lim_{c\rightarrow 0} V_\infty$, we fix $\eps>0$. There exists a $\delta>0$ (depending on $\eps$) such that $H(z)<H(1-)+\eps$ for all $z\in (1-\delta,1)$. We have
\begin{align*}
H(1-) \le\lim_{c\rightarrow 0} V_\infty&=H(1-)- \lim_{c\rightarrow 0} 2c\sigma^2 \ln\left(\int_0^1 \exp\left(\frac{H(1-)-H(z)}{2c\sigma^2}\right)dz\right)\\
&\le H(1-)- \lim_{c\rightarrow 0} 2c\sigma^2 \ln\left(\int_{1-\delta}^1 \exp\left(\frac{-\eps}{2c\sigma^2}\right)dz\right)\\
&= H(1-)- \lim_{c\rightarrow 0} 2c\sigma^2 \left(\ln\delta -\frac{\eps}{2c\sigma^2}\right)= H(1-)+\eps.
\end{align*}
Since $\eps$ is arbitrary, we conclude that $\lim_{c\rightarrow 0} V_\infty= H(1-)$. To compute $\lim_{c\rightarrow \infty} V_\infty$, we use L'H\^opital's rule:
\begin{align*}
\lim_{c\rightarrow \infty} V_\infty&=\lim_{c\rightarrow \infty}  -2c\sigma^2 \ln\left(\int_0^1 \exp\left(\frac{-H(z)}{2c\sigma^2}\right)dz\right)\\
&=\lim_{c\rightarrow \infty} \frac{\int_0^1 \exp\left(\frac{-H(z)}{2c\sigma^2}\right) H(z) dz}{\int_0^1 \exp\left(\frac{-H(z)}{2c\sigma^2}\right)dz}=\int_0^1 H(z)dz.
\end{align*}

(iv) It is obvious from \eqref{eq:infNEvalue} that $V_\infty(H_1)\ge V_\infty(H_2)$ if $H_1\ge H_2$. When $T<\infty$ and $H_1-R_{1,\infty}\ge H_2-R_{2,\infty}$, we have $\phi(r;R_1)\le  \phi(r;R_2)$ ($\phi$ is defined in \eqref{eq:phi}), which implies
$\phi(\beta(R_1);R_1)=\phi(\beta(R_2);R_2)\ge  \phi(\beta(R_2);R_1).$
Since $\phi(\cdot;R_1)$ is increasing, we must have $\beta(R_1)\ge \beta(R_2)$, and consequently, $V(R_1)-R_{1,\infty}\ge V(R_2)-R_{2,\infty}$ by \eqref{eq:NEvalue}.

\subsection{Proofs of Theorems \ref{thm:optR}, \ref{thm:optRT} and \ref{thm:welfare}}

We first solve an auxiliary problem. Define
\begin{equation}\label{eq:JK}
\mathcal{J}(h):=\int_0^\alpha h(r)dr, \quad \mathcal{K}(h):=\int_0^\alpha -\ln h(r) dr.
\end{equation}
and
\begin{equation}\label{eq:feasibleh}
\mathfrak{h}:=\left\{h: [0,\alpha]\rightarrow \bbR: \text{$h$ is increasing}, 0<h\le \exp\left(-\frac{R_\infty}{2c\sigma^2}\right), \mathcal{K}(h)\le \frac{K-R_\infty(1-\alpha)}{2c\sigma^2} \right\}.
\end{equation}
\begin{lemma}\label{lemma:hstar}
Let $\alpha\in(0,1]$ and $K\ge R_\infty$. Then $\inf_{h\in \mathfrak{h}}\mathcal{J}(h)$ is uniquely attained (up to a.e.\ equivalence)  at
\[h^*\equiv \exp\left(-\frac{K-R_\infty(1-\alpha)}{2\alpha c\sigma^2}\right).\]
\end{lemma}
\begin{proof}
We first show uniqueness. Suppose $h_1, h_2\in\mathfrak{h}$ both minimize $\mathcal{J}$. Then, since $\mathcal{J}$ is linear and $\mathfrak{h}$ is convex, $h_\eps:=(1-\eps) h_1+\eps h_2$ is optimal for any $\eps\in[0,1]$. Note that any optimizer $h$ necessarily satisfies
\[\mathcal{K}(h)=\frac{K-R_\infty(1-\alpha)}{2c\sigma^2},\]
otherwise we can find $\lambda\in(0,1)$ such that $\lambda h\in \mathfrak{h}$ and $\mathcal{J}(\lambda h)<\mathcal{J}(h)$. This implies $\mathcal{K}(h_\eps)= (1-\eps)\mathcal{K}(h_1)+\eps \mathcal{K}(h_2)$ for any $\eps\in[0,1]$. However, by the strict convexity of $z\mapsto -\ln z$ (and thus of $\mathcal{K}$), this can only happen when $h_1=h_2$ a.e.\ on $[0,\alpha]$.

Next, we show $h^*$ is an optimizer. Let $h\in\mathfrak{h}$ be arbitrary. We have by Jensen's inequality that
\[-\frac{1}{\alpha}\mathcal{K}(h)=\frac{1}{\alpha}\int_0^\alpha \ln h(r)dr\le \ln \left(\frac{1}{\alpha} \int_0^\alpha h(r)dr\right).\]
This implies
\[\mathcal{J}(h)\ge \alpha\exp\left(-\frac{1}{\alpha}\mathcal{K}(h)\right)\ge \alpha\exp\left(-\frac{K-R_\infty(1-\alpha)}{2\alpha c\sigma^2}\right)=\mathcal{J}(h^*).\]
\end{proof}

\noindent{\bf Proof of Theorem \ref{thm:optR}.}
By Theorem~\ref{thm:explicit_soln}(ii), we have
\[T_\alpha(H)=F^{-1}_{\tau^\circ_{x_0/\sigma}}\left(\frac{\int_0^\alpha  \exp\left(-\frac{H(r)}{2c\sigma^2}\right) dr}{\int_0^1\exp\left(-\frac{H(r)}{2c\sigma^2}\right) dr}\right).\]
First of all, it is clear  from the above expression that one should only pay $R_\infty$ beyond rank $\alpha$, because $H1_{[0,\alpha]}+R_\infty1_{(\alpha,1]}$ performs no worse than $H$. Assuming $H=R_\infty$ on $(\alpha,1]$, we further let $h(r):=\exp\left(-\frac{H(r)}{2c\sigma^2}\right)$ and write
\[T_\alpha(H)=F^{-1}_{\tau^\circ_{x_0/\sigma}}\left(\frac{\int_0^\alpha h(r )dr}{\int_0^\alpha h(r)dr+(1-\alpha)\exp\left(-\frac{R_\infty}{2c\sigma^2}\right)}\right).\]
Then, to minimize $T_\alpha(H)$, it suffices  to minimize $\mathcal{J}(h)$. Moreover, the feasibility constraint $H\in\mathcal{H}$ precisely translates to $h\in\mathfrak{h}$, under the assumption that $H=R_\infty$ on $(\alpha,1]$. Thus, the solution to our minimum quantile problem is in one-to-one correspondence to the solution to our auxiliary problem, given by Lemma~\ref{lemma:hstar}.
All the statements of the theorem can now be derived in a straightforward manner.

\noindent{\bf Proof of Theorem \ref{thm:optRT}.}
By Theorem~\ref{thm:explicit_soln}(i), we have
\[T_\alpha(H;T)=\begin{cases}
F^{-1}_{\tau^\circ_{x_0/\sigma}}\left(\frac{1-F_{\tau^\circ_{x_0/\sigma}}(T)}{1-\beta(H)}\int_0^\alpha \exp\left(\frac{R_\infty-H(r)}{2c\sigma^2}\right)dr\right)\le T, & \alpha\le \beta(H), \\
\Delta, & \alpha>\beta(H),
\end{cases}\]
where $\beta(H)=F_{\mathcal{E}_T(H)}(T)\in (0,1)$ is the unique solution of
\[\frac{F_{\tau^\circ_{x_0/\sigma}}(T)}{1-F_{\tau^\circ_{x_0/\sigma}}(T)}=\frac{1}{1-z}\int_0^{z} \exp\left(\frac{R_\infty-H(r)}{2c\sigma^2}\right)dr=\phi(z;H).\]

{\bf CLAIM:} The optimal value does not change if we restrict ourselves to reward function $H$ which satisfies $H(r)=R_\infty$ for all $r>\alpha$.

To prove the claim,
let $H\in\mathcal{H}^\alpha_T$ and define $\hat H:=H1_{[0,\alpha]}+R_\infty 1_{(\alpha,1]}$. We first check the feasibility of $\hat H$. Since $\phi(\beta(\hat H);\hat H)=\phi(\beta(H);H)\le \phi(\beta(H);\hat H)$ and $\phi(z;\hat H)$ is increasing, we must have $\beta(\hat H)\le \beta(H)$. Using this, together with $H\ge \hat H\ge R_\infty$, we get
\[\int_0^{\beta(\hat H)} \hat H(r) dr+(1-\beta(\hat H))R_\infty\le \int_0^{\beta(H)} H(r)dr+(1-\beta(H))R_\infty \le K.\]
Thus, $\hat H$ also satisfies the budget constraint. Since $\beta(H)\ge\alpha$, we have by the monotonicity of $\phi$ that
\[\phi(\beta(\hat H);\hat H)=\phi(\beta(H);H)\ge \phi(\alpha;H)=\phi(\alpha;\hat H)\]
and thus, $\beta(\hat H)\ge \alpha$. So we conclude that $\hat H\in \mathcal{H}^\alpha_T$. Next, we have
\begin{align*}
F_{\tau^\circ_{x_0/\sigma}}(T_\alpha(\hat H;T))&=\frac{1-F_{\tau^\circ_{x_0/\sigma}}(T)}{1-\beta(\hat H)}\int_0^\alpha \exp\left(\frac{R_\infty-\hat H(r)}{2c\sigma^2}\right)dr\\
&=\frac{1-F_{\tau^\circ_{x_0/\sigma}}(T)}{1-\beta(\hat H)}\int_0^\alpha \exp\left(\frac{R_\infty- H(r)}{2c\sigma^2}\right)dr\\
&\le \frac{1-F_{\tau^\circ_{x_0/\sigma}}(T)}{1-\beta(H)}\int_0^\alpha \exp\left(\frac{R_\infty-H(r)}{2c\sigma^2}\right)dr=F_{\tau^\circ_{x_0/\sigma}}(T_\alpha(H;T)),
\end{align*}
which is equivalent to $T_\alpha(\hat H;T)\le T_\alpha(H;T)$, and the claim has been verified.

Its immediate consequence  is that if $\beta(H)\ge\alpha$ and $H(r)=R_\infty$ for all $r>\alpha$, then
\[ \int_0^{\beta(H)} H(r) dr+(1-\beta(H))R_\infty= \int_0^{1} H(r) dr,\]
which implies
\[\mathcal{H}^\alpha_T\cap \{H|_{(\alpha,1]}\equiv R_\infty\}=\mathcal{H}\cap\{\beta(H)\ge \alpha\}\cap \{H|_{(\alpha,1]}\equiv R_\infty\}.\]
Thus, we can work with the simpler feasible set $\mathcal{H}$ instead of the equilibrium-dependent $\mathcal{H}^\alpha_T$.

As in the $T=\infty$ case, we let $h(r):=\exp\left(-\frac{H(r)}{2c\sigma^2}\right)$. In view of the claim above, we can set $h(r)\equiv\exp\left(-\frac{R_\infty}{2c\sigma^2}\right)$ for $r\in(\alpha,1]$ and only search for the optimal $h$ on $[0,\alpha]$. With a slight abuse of notation, we also use $\beta(h)$ to denote the unique solution $z$ in $(0,1)$ of
\begin{equation}\label{eq:betah}
C_T:=\frac{F_{\tau^\circ_{x_0/\sigma}}(T)}{1-F_{\tau^\circ_{x_0/\sigma}}(T)}=\frac{1}{1-z} \left(\int_0^{z\wedge \alpha} \exp\left(\frac{R_\infty}{2c\sigma^2}\right) h(r)dr+(z-z\wedge \alpha)\right).
\end{equation}
The feasibility condition translates to $h\in \mathfrak{h}_T:=\mathfrak{h}\cap\{\beta(h)\ge\alpha\}$ where $\mathfrak{h}$ is defined in \eqref{eq:feasibleh}.
Equation \eqref{eq:betah} implies that $\beta(h)\ge \alpha$ if and only if
\[\mathcal{J}(h)=\int_0^\alpha h(r)dr\le C_T(1-\alpha)\exp\left(-\frac{R_\infty}{2c\sigma^2}\right).\]
By Lemma~\ref{lemma:hstar},
\[\inf_{h\in\mathfrak{h}}\mathcal{J}(h)=\mathcal{J}(h^*)=\alpha \exp\left(-\frac{K-(1-\alpha)R_\infty}{2\alpha c\sigma^2}\right).\]
Therefore, we arrived at the following feasibility criteria:
\begin{align*}
\mathfrak{h}_T\neq\emptyset &\Longleftrightarrow \beta(h^*)\ge \alpha\Longleftrightarrow \mathcal{J}(h^*)\le C_T(1-\alpha)\exp\left(-\frac{R_\infty}{2c\sigma^2}\right)\\
&\Longleftrightarrow T \ge F^{-1}_{\tau^\circ_{x_0/\sigma}}\left(\frac{\alpha }{\alpha+(1-\alpha)\exp\left(\frac{K-R_\infty}{2\alpha c\sigma^2}\right)}\right)=T^*_\alpha.
\end{align*}
If $\mathfrak{h}_T=\emptyset $, then there is no feasible reward scheme which attains, in equilibrium, the desired completion rate of $\alpha$ by time $T$.
Suppose $h\in\mathfrak{h}_T\neq\emptyset$. Then
\begin{align*}
F_{\tau^\circ_{x_0/\sigma}}(T_\alpha(H;T))=\left(1-F_{\tau^\circ_{x_0/\sigma}}(T)\right)\exp\left(\frac{R_\infty}{2c\sigma^2}\right) \frac{\int_0^\alpha h(r)dr}{1-\beta(h)}.
\end{align*}
Thus, to minimize $T_\alpha(H;T)$, we only need to minimize
\[\mathcal{L}(h):=\frac{\mathcal{J}(h)}{1-\beta(h)}\exp\left(\frac{R_\infty}{2c\sigma^2}\right), \quad h\in\mathfrak{h}_T.\]
From \eqref{eq:betah}, we also obtain
\[\beta(h)=\frac{C_T+\alpha-\exp\left(\frac{R_\infty}{2c\sigma^2}\right)\mathcal{J}(h)}{1+C_T}, \quad h\in\mathfrak{h}_T.\]
It follows that
\[\mathcal{L}(h)=\frac{\mathcal{J}(h)}{1-\alpha+\exp\left(\frac{R_\infty}{2c\sigma^2}\right)\mathcal{J}(h)}, \quad h\in\mathfrak{h}_T.\]
Hence, to minimize $\mathcal{L}(h)$, it suffices to minimize $\mathcal{J}(h)$ over $h\in\mathfrak{h}_T$. Note that when $\mathfrak{h}_T\neq\emptyset$, we have $\beta(h^*)\ge\alpha$ and thus, $h^*\in\mathfrak{h}_T$. This implies $h^*=\argmin_{h\in\mathfrak{h}}\mathcal{J}(h)$ is also the (unique) optimizer of $\inf_{h\in\mathfrak{h}_T}\mathcal{J}(h)$. The optimal $H^*$ and  $T^*_\alpha(T)$ are then the one corresponding to $h^*$.

As in the case $T=\infty$, all the statements of the theorem but the last two can now be derived in a straightforward manner.
For the last two, setting $T^{*}_\alpha=T$  yields the minimum budget and the maximum terminal completion rate.

\noindent{\bf Proof of Theorem \ref{thm:welfare}.}
(i) By \eqref{eq:infNEvalue}, maximizing $V_\infty$ over $\mathcal{H}$ is equivalent to minimizing $\mathcal{J}(h)=\int_0^1 h(z)dz$ over $\mathfrak{h}$ (defined in \eqref{eq:feasibleh}) with $\alpha=1$. By Lemma~\ref{lemma:hstar}, $\argmin_{h\in\mathfrak{h}}\mathcal{J}(h)=h^*\equiv \exp\left(-K/(2c\sigma^2)\right)$. It follows that $H^*=\argmax_{H\in\mathcal{H}} V_\infty(H)=-2c\sigma^2 \ln h^*\equiv K$ and $V_\infty(H^*)=K$.

(ii) By \eqref{eq:NEvalue} and Proposition~\ref{prop:expected-effort}, maximizing
the welfare or the expected total effort
is equivalent to maximizing the terminal completion rate, {which we have solved in Theorem~\ref{thm:optRT}}.

\bibliographystyle{plainnatnurl}
\bibliography{tournament}{}

\end{document}